\documentclass[10pt,oneside,a4paper]{article} 				

\usepackage[fleqn]{amsmath} 										
\usepackage{amsthm,amsfonts,latexsym,amssymb,amscd} 				
\usepackage[mathscr]{eucal}   										
\usepackage[all,2cell]{xy} \UseAllTwocells 							
\usepackage{stmaryrd} 												
\usepackage{framed}
\usepackage{color}
\definecolor{cmyk}{cmyk}{.0,.4,.9,.5}								
\usepackage{txfonts}
\usepackage{mathtools} 												

\usepackage{bbold}

\usepackage{scalerel} 	

\usepackage{tikz}
\usetikzlibrary{svg.path}

\definecolor{orcidlogocol}{HTML}{A6CE39}
\tikzset{
  orcidlogo/.pic={
    \fill[orcidlogocol] svg{M256,128c0,70.7-57.3,128-128,128C57.3,256,0,198.7,0,128C0,57.3,57.3,0,128,0C198.7,0,256,57.3,256,128z};
    \fill[white] svg{M86.3,186.2H70.9V79.1h15.4v48.4V186.2z}
                 svg{M108.9,79.1h41.6c39.6,0,57,28.3,57,53.6c0,27.5-21.5,53.6-56.8,53.6h-41.8V79.1z M124.3,172.4h24.5c34.9,0,42.9-26.5,42.9-39.7c0-21.5-13.7-39.7-43.7-39.7h-23.7V172.4z}
                 svg{M88.7,56.8c0,5.5-4.5,10.1-10.1,10.1c-5.6,0-10.1-4.6-10.1-10.1c0-5.6,4.5-10.1,10.1-10.1C84.2,46.7,88.7,51.3,88.7,56.8z};
  }
}

\newcommand\orcidicon[1]{\href{https://orcid.org/#1}{\mbox{\scalerel*{
\begin{tikzpicture}[yscale=-1,transform shape]
\pic{orcidlogo};
\end{tikzpicture}
}{|}}}}

\usepackage[draft=false,breaklinks,backref]{hyperref} 				

\hypersetup{
           breaklinks=true,   
           pdfusetitle=true,  
        }

\usepackage{cite}

\renewcommand{\O}{\mathcal{O}}	
 
\newcommand{\T}{\mathcal{T}}\newcommand{\U}{\mathcal{U}}

\newcommand{\Eg}{\mathfrak{E}}\newcommand{\Gg}{\mathfrak{G}}

\newcommand{\Ig}{\mathfrak{I}}

\renewcommand{\AA}{{\mathbb{A}}} 
\newcommand{\CC}{{\mathbb{C}}}

\newcommand{\RR}{{\mathbb{R}}} 
\newcommand{\TT}{{\mathbb{T}}}

\newcommand{\As}{{\mathscr{A}}}\newcommand{\Bs}{{\mathscr{B}}}\newcommand{\Cs}{{\mathscr{C}}}
\newcommand{\Ds}{{\mathscr{D}}}
\newcommand{\Es}{{\mathscr{E}}}\newcommand{\Fs}{{\mathscr{F}}}\newcommand{\Gs}{{\mathscr{G}}}

\newcommand{\Is}{{\mathscr{I}}}\newcommand{\Js}{{\mathscr{J}}}\newcommand{\Ks}{{\mathscr{K}}} 
\newcommand{\Ls}{{\mathscr{L}}}	
\newcommand{\Ms}{{\mathscr{M}}}\newcommand{\Os}{{\mathscr{O}}}
\newcommand{\Rs}{{\mathscr{R}}}\newcommand{\Ss}{{\mathscr{S}}}

\newcommand{\Xs}{{\mathscr{X}}}
\newcommand{\Ys}{{\mathscr{Y}}}

\DeclareFontFamily{U}{rsfs}{\skewchar\font127 }
\DeclareFontShape{U}{rsfs}{m}{n}{%
   <5> <6> rsfs5
   <7> rsfs7
   <8> <9> <10> <10.95> <12> <14.4> <17.28> <20.74> <24.88> rsfs10
}{}
\DeclareSymbolFont{rsfs}{U}{rsfs}{m}{n} 
\DeclareSymbolFontAlphabet{\scr}{rsfs}

\newcommand{\Af}{\scr{A}}

\newcommand{\Tf}{\scr{T}}
\newcommand{\Xf}{\scr{X}}

\DeclareMathOperator{\ke}{Ker}

\DeclareMathOperator{\spa}{span} 
\DeclareMathOperator{\id}{Id} 

\DeclareMathOperator{\Ob}{Ob}

\DeclareMathOperator{\Hom}{Hom}

\DeclareMathOperator{\Aut}{Aut} 
\DeclareMathOperator{\Sp}{Sp}
\DeclareMathOperator{\ev}{ev}

\DeclareMathOperator{\rk}{Rk}

\renewcommand{\emph}{\textbf} 										
\newcommand{\cj}[1]{\overline{#1}}									
\newcommand{\ip}[2]{\langle #1\mid #2\rangle}						
\renewcommand{\iff}{\Leftrightarrow}								
\newcommand{\imp}{\Rightarrow}										
\newcommand{\st}{\ : \ }											

\newcommand{\dya}{\scaleobj{0.8}{\Join}}
\newcommand{\ddya}{\scaleobj{0.8}{\mathrlap{\hspace{0.4pt}\Join}\bigcirc}}
\newcommand{\Xddya}{\overleftrightarrow{\kern -2pt\Xf}_{\kern -2pt\As}^{\ddya}}
\newcommand{\Xdya}{\overleftrightarrow{\kern -2pt\Xf}_{\kern -2pt\As}^{\dya}}
\newcommand{\dpartial}{\partial\kern-4.5pt /}
\newcommand{\yad}{\ \scaleobj{0.6}{\pmb{|}} \kern-2.5pt \scaleobj{0.8}{\Join}}

\newcommand{\hlink}[2]{\href{#1}{\texttt{#2}}} 						

\newcommand{\xqedhere}[2]{%
  \rlap{\hbox to#1{\hfil\llap{\ensuremath{#2}}}}}

\newcommand{\xqed}[1]{%
  \leavevmode\unskip\penalty9999 \hbox{}\nobreak\hfill
  \quad\hbox{\ensuremath{#1}}}

\theoremstyle{plain}
\newtheorem{theorem}{Theorem}[section]							

\newtheorem{lemma}[theorem]{Lemma}
 
\newtheorem{proposition}[theorem]{Proposition}

\newtheorem{definition}[theorem]{Definition}

\theoremstyle{definition} 
\newtheorem{remark}[theorem]{Remark}

\numberwithin{equation}{section}  									
\title{\textbf{Spectral Theory for Non-full Commutative C*-categories}}

\author{
\normalsize  
\orcidicon{0000-0002-1387-9283} Paolo Bertozzini$^a$, 
\orcidicon{0000-0003-3128-6884} Roberto Conti$^b$,
\orcidicon{0000-0003-2575-7184} Wicharn Lewkeeratiyutkul$^c$
\orcidicon{0009-0000-7477-2992} Kasemsun Rutamorn$^d$  
\\  
\normalsize $^a$ \textit{Independent Researcher - Bangkok, Thailand}
\\ \normalsize 
$^b$  \textit{Dipartimento di Scienze di Base e Applicate per l'Ingegneria} 
\\
\normalsize \textit{Sapienza Universit\`a di Roma, Via A. Scarpa 16, I-00161 Roma, Italy}
\\ \normalsize 
$^c$ $^d$ \textit{Department of Mathematics and Computer Science, Faculty of Science} 
\\
\normalsize \textit{Chulalongkorn University, Bangkok 10330, Thailand} 
\\ 
\normalsize e-mail: 
\ 
\quad 
$^a$ \texttt{paolo.th@gmail.com} 
\quad 
$^b$ \texttt{roberto.conti@sbai.uniroma1.it} 
\\ \normalsize \phantom{e-mail:}
\quad 
$^c$ \texttt{Wicharn.L@chula.ac.th}
\quad 
$^d$ \texttt{kasemamorn270@hotmail.com} 
}

\date{\normalsize{
corrected version for publication: 08 July 2025
\\
second version: 18 February 2025
\quad 
first version: 20 October 2024
\quad 
started: 09 November 2023 
}}

\begin{document}

\maketitle


\begin{abstract} \noindent 
We extend the spectral theory of commutative C*-categories to the non-full case, introducing a suitable notion of spectral spaceoid providing a duality between a category of ``non-trivial'' $*$-functors of non-full commutative C*-categories and a category of Takahashi morphisms of ``non-full spaceoids'' (here defined). 
\\ 
As a byproduct we obtain a spectral theorem for a non-full generalization of  imprimitivity Hilbert \hbox{C*-bi}\-modules over commutative unital C*-algebras via continuous sections vanishing at infinity of a Hilbert \hbox{C*-line-bundle} over the graph of a homeomorphism between open subsets of the corresponding Gel'fand spectra of the C*-algebras. 

\medskip


\smallskip

\noindent
\emph{MCS-2020:} 
46L87, 46M15, 46L08, 46M20, 16D90, 18F99

\end{abstract}

\tableofcontents 

\section{Introduction and Motivation}

Duality between ``algebra'' and ``geometry'' is still an essential mathematical theme of investigations at least since~\cite{Des37}. The specific correspondence between commutative unital C*-algebras and compact Hausdorff spaces that is central in this paper was initiated in the fundamental work of~\cite{GeNa43}; it was formulated in categorical duality form in~\cite{Ne69,Ne71} (see also~\cite{Se71}) and it is one of the leading ideological motivations behind the current development of non-commutative differential geometry~\cite{Co80,Co94}. 

\medskip 

Following initial studies by~\cite{Kap51,Fe60,Fe61} significant attempts to generalize such duality to the case of non-commutative C*-algebras (and their Hilbert C*-modules) were initially developed in~\cite{DH68,Ho72a,Ho72b} and further eleborated in~\cite{Va74a,Va74b,Ta79a,Ta79b} (see also~\cite{DG83,FD98,Ho11}).  

\medskip 

C*-categories were initially introduced in~\cite{GLR85} in view of important applications to the theory of superselection sectors~\cite{DR89} in algebraic quantum field theory, where their usage has been constantly expanding, and they have been further investigated in~\cite{Mi02}. 

\medskip 

In the general context of ``categorification'' (see the introduction of~\cite{BCLS20} for details and references), \hbox{C*-categories} are a ``horizontal categorification'' (i.~e.~many-object version) of the notion of \hbox{C*-algebra} and hence it is perfectly natural to investigate, in suitable commutative situations, the dual ``horizontal categorification'' of the notion of Gel'fand spectrum of a commutative C*-algebra. 

\medskip 

In a C*-category, the diagonal Hom-sets, i.e.~the Hom-sets between any object and itself, are unital C*-algebras. Likewise, in a C*-category the off-diagonal Hom-sets are Hilbert C*-bimodules over the corresponding diagonal C*-algebras. 
A C*-category is said to be \textit{commutative} if all the diagonal C*-algebras are commutative.

\medskip 

An extension of Gel'fand-Na\u\i mark duality to commutative C*-categories was announced in~\cite{BCL08a} and later presented in~\cite{BCL11}, under the assumption of fullness (imprimitivity of the Hom-sets bimodules), a result that continues to motivate some further (ongoing) developments in the study of spectral theory of non-commutative C*-algebras (see~\cite{BCP19} and forthcoming work). 

\medskip 

The specific purpose of our paper is to examine the effect of removal of the original fullness assumption in the spectral theory of commutative C*-categories~\cite{BCL11}. We will provide here a notion of non-full topological spectral spaceoids adapted to this more general situation. 
In short, the main result is a duality between the category of non-necessarily full commutative C*-categories, with \hbox{$*$-functors} satisfying a suitable non-degeneracy condition as morphisms, and the category of non-full spaceoids, as topological counterpart playing the role of spectra of commutative C*-categories (see below for more precise statements). 

\medskip 

For important technical reasons, we limit here the discussion of our duality to the category of those $*$-functors, between small commutative non-full C*-categories, that are ``non-degenerate'' (their induced pull-back preserves, on each Hom-set, the non-triviality of functionals). The removal of this restriction will be done in forthcoming work and we anticipate here, for the curious reader, that this will require a ``rethinking'' of some basic aspects of the usual Gel'fand spectral theory also in the case of commutative C*-algebras, whose ``spectra'' will have to be identified with Fell line-bundles over the Alexandroff extension of the usual Gel'fand spectra, with a zero-dimensional fiber over the point at infinity. 

\medskip 

In order to facilitate the reader with some limited familiarity with the subject (who might not want to enter into the many discouraging technical details in the proofs of this paper) we are going to present a rather broad motivational summary of the material here developed. 

\medskip 

C*-categories are an extremely natural and ubiquitous ``many-objects'' generalization of unital C*-algebras (norm closed \hbox{$*$-subalgebras} of bounded linear operators on Hilbert spaces): they correspond to norm closed families of bounded linear operators, potentially between \textit{several Hilbert spaces}, that are algebraically closed under the operations of composition and adjunction. 

\medskip 

Commutative (unital) complex C*-algebras $\As$ are always canonically isomorphic to algebras of continuous $\CC$-valued functions $C(\Sp(\As);\CC)$ defined on a compact Hausdorff topological space $\Sp(\As)$, their Gel'fand spectrum, that consists of the family of unital $\CC$-valued $*$-homomorphisms $\As\xrightarrow{\omega}\CC$ equipped with the \hbox{weak-$*$-to}\-pology. The Gel'fand transform $\As\xrightarrow{\Gg_\As} C(\Sp(\As),\CC)$ provides the canonical isomorphism associating to every  $x\in\As$ a function $\hat{x}:\Sp(\As)\to\CC$ specified by $\hat{x}:\omega\mapsto \omega(x)$.  
Similarly every compact Hausdorff topological space $X$ is canonically homeomorphic, via the evaluation transform $X\xrightarrow{\Eg_X}\Sp(C(X;\CC))$ to the Gel'fand spectrum of the commutative unital C*-algebra of continuous $\CC$-valued functions on $X$: every point $p\in X$ corresponds to an evaluation $*$-homomorphism $\ev_p: f\to f(p)$ for $f\in C(X;\CC)$.  
This celebrated Gel'fand-Na\u \i mark duality result is not only fundamental for all the developments of spectral theory of operators on Hilbert spaces, but also provides the starting point of non-commutative topology as the study of non-commutative C*-algebras. 

\medskip 

Similar spectral results,, in the spirit of Serre-Swan{\color{brown}'s} theorem, have been obtained (see~\cite{Ta79a,Ta79b} for details) for Hilbert C*-modules $\Ms_\As$ over commutative (unital) C*-algebras $\As$, whose ``spectrum'' consists of Hilbert bundles over the $\Sp(\As)$ with fibers over $\omega\in\Sp(\As)$ consisting of the $\CC$-Hilbert space $\Ms/\ke(\omega)$. In the particular case of the module $\As_\As$, the bundle is simply the trivial $\CC$-line bundle $\Sp(\As)\times\CC\to\Sp(\As)$ over the previous Gel'fand spectrum $\Sp(\As)$.   
Imprimitivity bimodules over $\As$ (the bimodules ${}_\As\Ms_\As$ that are invertible, modulo isomorphisms, under the composition provided by $\otimes_\As$, the tensor product of Hilbert C*-bimodules over $\As$) correspond to (possibly non-trivial) modules of continuous sections of arbitrary $\CC$-Hilbert-line-bundles over $\Sp(\As)$. 

\medskip 

In a C*-category $\Cs$, for every pair of objects $A,B$, the Hom-sets $\Cs_{AB}:=\Hom_\Cs(B;A)$ are naturally Hilbert C*-bimodules (that are imprimitivity bimodules whenever they satisfy the fullness condition), on the left over the unital C*-algebra $\Cs_{AA}$ and on the right over the unital C*-algebra $\Cs_{BB}$; as a consequence, whenever the unital C*-algebras $\Cs_{AA}$ are commutative, it is possible to envision a spectral theory that suitably merges the spectra of all the bimodules $\Cs_{AB}$. 
Under the additional technical assumption of fullness (hence imprimitivity) for the bimodules $\Cs_{AB}$, all of this was achieved in~\cite{BCL11}: the spectrum $\Sigma(\Cs)$ of a commutative full \hbox{C*-category} $\Cs$ is what we call a \textit{topological (full) spaceoid}: a saturated Fell $\CC$-line bundle $\Es$ over $\Xs$, the graph of a pair-groupoid (with the same objects of $\Cs$) of homeomorphisms between the Gel'fand spectra of the ``diagonal'' C*-algebras $\Cs_{AA}$. 
The portion $\Es|_{AB}$ of the bundle supported over the graph $\Xs_{AB}$ of the homeomorphism $\Sp(\Cs_{BB})\to\Sp(\Cs_{AA})$ reproduces the previous spectrum of the imprimitivity bimodule ${}_{\Cs_{AA}}(\Cs_{AB})_{\Cs_{BB}}$. 
The obtained result is a duality providing at the same time a horizontal (many-objects) categorification of the Gel'fand-Na\u\i mark duality for commutative unital C*-algebras and of the Takahashi version of Serre-Swan equivalence for Hilbert C*-modules over commutative unital C*-algebras \cite{Ta79b}. 

\medskip 

The technical condition of fullness allowed significant shortcuts and simplifications in the proofs of the previous duality that are satisfactorily eliminated in the present work. 
The spectra of commutative C*-categories $\Cs$ are here given by \textit{(non-full) topological spaceoids} that consist of saturated Fell $\CC$-line bundles $\Es\to\Xs$ over a topological groupoid $\Xs$ whose Hom-sets $\Xs_{AB}$ are graphs of homeomorphisms between a pair of open subsets of the Gel'fand spectra $\Sp(\Cs_{AA})$. 
Notice that the ``non-fulless'' of the topological spaceoid manifests in the possibly strict inclusion of some of the open subsets of the Gel'fand spectra: whenever the open sets coincide with the entire Gel'fand spectra, our results reproduce those of the already known spectral theory for full commutative C*-categories. 

\medskip 

The apparently elementary modification in the above definition of topological spaceoid does not give full justice to the increased complications that are unleashed by this generalization. 

\medskip 

First of all continuous sections of the $\CC$-line-bundle $\Es_{AB}\to\Xs_{AB}$ must be required to ``vanish at infinity'' in order to recover a correct off-diagonal Hom-set non-full bimodule $\Cs_{AB}$; this entails a new definition of the section functor $\Gamma$ associating to a topological spaceoid $\Es$ the C*-category $\Gamma(\Es)$ of its Hom-set continuous sections vanishing at infinity. 

\medskip 

More important and significant are the consequences for the spectrum functor $\Sigma$ associating to every non-full commutative C*-category $\Cs$ its ``spectral spaceoid'' $\Sigma(\Cs)$. Although $*$-functors $\omega:\Cs\to\CC$ on a full \hbox{C*-category} $\Cs$ have always non-trivial restrictions $\omega_{AB}:\Cs_{AB}\to\CC$ on every full C*-bimodule $\Cs_{AB}$, this property can fail in the case of non-full C*-categories; as a consequence the fiber $\Cs_{AB}/\ker(\omega_{AB})$ can be 0-dimensional. 

\medskip 

There are two different alternative choices that can be pursued at this point:
\begin{itemize}
\item
the most elegant and radical (that is not the one pursued in this work) consists in allowing 0-dimensional fibers $\Cs_{AB}/\ke(\omega_{AB})$ over trivial functionals $\omega_{AB}=0$ in defining spectral spaceoids as complex Fell bundles with fibers of variable dimension ``less than or equal to 1''. This move will complicate the nature of the tubular topology (see remark~\ref{rem: tubular}) of the total space of the Fell-bundle, that now must be defined via the well-known Dauns-Hofmann uniformities (see for example~\cite[definitions 1-2]{Va95} and the original references therein), since the bundle might fail to be locally trivial. Furthermore, the introduction of zero-dimensional fibers supported over zero-functionals, necessarily forces some significant conceptual changes in the description of spectral theory, propagating back even to the usual Gel'fand-Na\u\i mark spectrum of a commutative unital C*-algebra (where the Gel'fand spectrum will now additionally contain an isolated 0-dimensional fiber corresponding to the zero character!). The final spectral spaceoid for a non-full C*-cateogy will anyway be described as the \textit{zerification} (a canonical universal process of adding \textit{zero morphisms} with zero-dimensional fibers to every Hom-set) of the very same notion of spectral spaceoid that is defined in the present paper (that hence is an unavoidable preliminary step).   

\item
the more restrictive road (that is the one considered in the present paper) consists in allowing only the presence of 1-dimensional fibers in the definition of spectral spaceoids. This cleverly avoids the deep conceptual issues above, but the price to pay is the need to impose a ``non-degeneracy condition'' on the morphisms between non-full commutative C*-categories: although $*$-functors $\Cs\xrightarrow{\Phi}\Ds$ between full commutative C*-categories are automatically \textit{non-degenerate} (the $\Phi$-pull-back $\omega\circ\Phi$ of a $*$-functor $\omega:\Cs\to\CC$ that is non-trivial on a certain Hom-set $\Ds_{\Phi(A)\Phi(B)}$ is also a non-trivial functional on $\Cs_{AB}$) this is not necessarily true in the non-full situation. \footnote{
For example, for the only object-preserving $*$-functor $\Phi:\begin{bmatrix}\CC & 0 \\ 0 & \CC\end{bmatrix}\to\begin{bmatrix}\CC & \CC \\ \CC & \CC\end{bmatrix}$, non-triviality fails on off-diagonal Hom-sets.  
} 
The non-degeneracy condition on $*$-functors $\Cs\xrightarrow{\Phi}\Ds$ must be required in order for the their spectral counterpart $\Sigma(\Cs)\xleftarrow{\Sigma^\Phi}\Sigma(\Ds)$ to be well-defined as a Takahashi morphism of (Fell) line-bundles between the 1-dimensional spectral spaceoids. 
\end{itemize} 
 
We decided to avoid in this paper the introduction of the above-mentioned zerification process (that would have imposed a doubling of the size of the manuscript), postponing the delicate conceptual (and necessarily controversial) proposals to systematically introduce a zero-character in the spectral theory of operator algebras. As it can be appreciated from the explanation above, the ``non-triviality condition'' on $*$-functors is here temporarily introduced and it will naturally disappear as soon as the zerification process is applied (the missing target of the restriction $\omega\mapsto \omega_{AB}$ will be the canonically added zero-point in the Hom-set $\Xs_{AB}$).  

\medskip 

Since our result is a duality, every possible example of non-full C*-category is necessarily ``spectrally isomorphic'' to one of the non-full spaceoids abstractly defined here; practical examples will in particular cover all the cases considered in the previous paper~\cite{BCL11} with the added generality coming from the avoidance of the condition of fullness: a nice application is the spectral theorem for ``non-full imprimitivity'' Hilbert C*-bimodules treated in our last section. 

\medskip 

In view of the relevance of imprimitivity bimodules in Morita theory and K-theory, one can conjecture a possible interest of our results in generalizations of Morita groupoids and in categorified topological K-theory. 

\bigskip 

The structure of the paper, closely follows the original exposition of~\cite{BCL11}: after this introduction, section~\ref{sec: 2} and~\ref{sec: 3} respectively introduce the basic definitions of categories of commutative C*-categories and (non-full) spaceoids; section~\ref{sec: 4} and~\ref{sec: 5} respectively describe the section and spectral functors between the previous categories; sections~\ref{sec: 6} and~\ref{sec: 7} introduce the Gel'fand and the evaluation natural transforms; section~\ref{sec: 8} describes the final duality result; section~\ref{sec: 9} discusses a spectral theorem for certain ``non-full imprimitivity'' Hilbert C*-bimodules that is obtained as a byproduct of our duality for non-full C*-categories applied to the linking C*-category of the bimodules.

\section{The Category $\Af$ of Small Commutative C*-categories} \label{sec: 2}

We assume familiarity with the basic definitions of (small) category, functor and natural transformation; the reader can find detailed descriptions of all such category-theoretic concepts in any standard reference text, among them we suggest the recent~\cite{Le14} and~\cite{Ri16}. 
 
\medskip 

We only recall here the definition of duality (contravariant equivalence) between categories. 
\begin{definition}
A \emph{duality} of two categories $\Af$ and $\Tf$ is a pair of contravariant functors $\Tf\xrightarrow{\Gamma}\Af$ and $\Af\xrightarrow{\Sigma}\Tf$ such that $\Gamma\circ\Sigma$ and $\Sigma\circ\Gamma$ are naturally equivalent to the respective identity functors $I_{\Af}$ and $I_{\Tf}$ (this means that there exist natural transformation isomorphisms $I_{\Af}\xRightarrow{\Gg} \Gamma\circ\Sigma$ and $I_{\Tf}\xRightarrow{\Eg} \Sigma\circ\Gamma$).  
\end{definition}

\medskip 

We introduce here mostly notational conventions used in the paper. 

\medskip 

For a category $\Cs:=(\Cs,\circ,\iota)$, we will denote by $\Cs^0:=\Ob_\Cs$ its family of objects, by $\Cs^1:=\Hom_\Cs$ its family of morphisms; whenever necessary, $\Cs^1\xrightarrow{s,\ t}\Cs^0$, with  $s(x)\mapsfrom x\mapsto t(x)$, will respectively indicate the source and the target maps and, for all $A,B\in \Cs^0$, Hom-sets are written as $\Cs_{AB}:=\Hom_\Cs(B;A)$; furthermore $\Cs^0\xrightarrow{\iota}\Cs^1$, with $A\mapsto \iota_A$, indicates the identity and
$\Cs^1\times_{\Cs^0}\Cs^1:=\{(x,y) \ | \ s(x)=t(y)\}\xrightarrow{\circ}\Cs^1$, 
with $(x,y)\mapsto x\circ y$, the composition. 

\begin{definition}
An \emph{involutive category} (also known as a \textit{dagger category}) is a category $(\Cs,\circ,\iota)$ equipped with a contravariant functor $\Cs\xrightarrow{*}\Cs$ such that $(x^*)^*=x$, for all $x\in \Cs$, and $(\iota_A)^*=\iota_A$, for all objects $A\in\Cs^0$. 

\medskip 

A \emph{groupoid} is a category where every morphism is an isomorphism. \footnote{
Groupoids are special examples of inverse involutive categories where the involution is the inversion map $\gamma:x\mapsto x^{-1}$, for all $x\in\Cs^1$. 
}

\medskip 

A complex \emph{algebroid} \footnote{
This definition actually describes ``unital algebroids''; assuming that $\Cs$ is a ``non-unital category'' (a category without identities) one can obtain a similar definition of non-unital algebroid. 
} is a category $\Cs$ whose Hom-sets $\Cs_{AB}$, for all $A,B\in\Cs^0$, are equipped with a complex vector spaces structure making the composition $\CC$-bilinear. 
A complex \emph{involutive algebroid} is a complex algebroid that is also an involutive category with involution that is conjugate-linear. 

\medskip 

More generally, a complex \emph{(involutive) algebroid over an (involutive) category {\color{brown} $\Xs$}} is a ($*$-)functor $\Cs\xrightarrow{\pi}\Xs$ between (involutive) categories, such that each fiber $\pi^{-1}(x)$ is a $\CC$-vector space for all $x\in\Xs$, with composition that is fibrewise $\CC$-bilinear (and involution that is fibrewise $\CC$-conjugate linear). \footnote{
Notice, see later remark~\ref{rem: pair-g}, that (involutive) algebroids $\Cs$ are just (involutive) algebroids over the pair groupoid $\Cs^0\times\Cs^0$.} 
\end{definition}

From this point on we assume that the reader is familiar with some basic notions of operator C*-algebras. \footnote{
For reference one can utilize~\cite{FD98} or~\cite{Bl06} 
}  

\medskip 

The following definition originates in~\cite{GLR85}, see also~\cite{Mi02}. 
\begin{definition} 
A (small) \emph{C*-category} $\Cs$ is a (small) complex involutive algebroid that is equipped with a norm $\|\cdot\|:\Cs\to\RR$ that satisfies the following properties: 
\begin{itemize}
\item
$(\Cs_{AB},\|\cdot\|)$ is a Banach space, for all $A,B\in\Cs^0$, 
\item 
$\|x\circ y\|\leq\|x\|\cdot\|y\|, \quad \forall (x,y)\in\Cs^1\times_{\Cs^0}\Cs^1$,  
\item 
$\|x^*\circ x\|=\|x\|^2, \quad \forall x\in \Cs^1$, 
\item 
$x^*\circ x$ is positive in the C*-algebra $\Cs_{AA}$, for all $x\in\Cs_{BA}$. 
\end{itemize}
We say that the C*-category is \emph{commutative} if, for all objects $A\in\Cs^0$, the diagonal Hom-sets $\Cs_{AA}$ are commutative unital C*-algebras. 
The C*-category is \emph{full} if all the bimodules $\Cs_{AB}$ are full (and hence imprimitivity - see the third paragraph of remark~\ref{rem: -imp} for details). 
A C*-category $\Cs$ is said to be \emph{discrete} if $\Cs_{AB}$ is a trivial $\CC$-vector space, for all $A,B\in\Cs^0$ such that $A\neq B$. 
\end{definition}
\begin{remark}\label{rem: -imp}
Notice that the first three properties in the previous definition already imply that all the ``diagonal Hom-sets'' $\Cs_{AA}$ are unital C*-algebras, for all $A\in\Cs^0$ and hence the last positivity condition is meaningful. 

\medskip 

Furthermore, for all $A,B\in\Cs^0$ the ``off-diagonal Hom-sets'' $\Cs_{AB}$ are unital bimodules, left over the unital \hbox{C*-algebra} $\Cs_{AA}$ and right over the unital C*-algebra $\Cs_{BB}$, that are also equipped with a left $\Cs_{AA}$-valued inner product ${}_A\ip{x}{y}:=x\circ y^*$ and a right $\Cs_{BB}$-valued inner product $\ip{x}{y}_B:=x^*\circ y$, for all $x,y\in\Cs_{AB}$, that satisfy the following property: ${}_A\ip{x}{y}z=x\ip{y}{z}_B$, for all $(x,y^*,z)\in\Cs^1\times_{\Cs^0}\Cs^1\times_{\Cs^0}\Cs^1$. 

\medskip 

The Hilbert C*-bimodules $\Cs_{AB}$ might fail to be imprimitivity bimodules only because of the lack of fullness of the previous inner-products: for all $A,B\in\Cs^0$, 
${}_A\ip{\Cs_{AB}}{\Cs_{AB}}:=\cj{\spa\{{}_A\ip{x}{y} \ | \ x,y\in\Cs_{AB}\}}\subset\Cs_{AA}$ and 
$\ip{\Cs_{AB}}{\Cs_{AB}}_B:=\cj{\spa\{\ip{x}{y}_B \ | \ x,y\in\Cs_{AB}\}}\subset\Cs_{BB}$ are closed involutive ideals that do not necessarily coincide with the unital C*-algebras $\Cs_{AA}$ and $\Cs_{BB}$. 
See~\cite{BCL08} for further details.
\xqed{\lrcorner}
\end{remark}

\begin{remark}
For closed involutive ideals $\Is$ of a commutative unital C*-algebra $C(X)$, there is a well-known spectral theory (see for example~\cite[volume I, section 8.9]{FD98}): the ideal (as a symmetric \hbox{$C(X)$-bi}\-module) is isomorphic to the $C(X)$-bimodule of continuous $\CC$-functions in $C(X)$ vanishing on the complement $K^\Is:=X-O^\Is$ of a given open set 
$O^\Is\subset X$. 

\medskip 

We do not have here an immediate generalization of the previous result to the case of Hilbert bimodules ${}_{C(X)}\Ms_{C(Y)}$ over commutative unital C*-algebras, with inner products satisfying ${}_{C(X)}\ip{x}{y}z=x\ip{y}{z}_{C(Y)}$, for all $x,y,z\in \Ms$, 
\footnote{
Such a spectral theorem will be later obtained, in section~\ref{sec: 9}, as a byproduct of the spectral theory for the commutative linking \hbox{C*-category} of the bimodule $\Ms$. 
}
and hence our route to the spectral theorem for a commutative C*-category $\Cs$ will not consists in ``assembling'' individual Hom-set spectral theories of the bimodules $\Cs_{AB}$; still it is already possible to appreciate the presence of pairs $O_A^{AB}\subset X_A$ and $O_B^{AB}\subset X_B$ of open sets associated to the spectrum of the ideals 
${}_{\Cs_{AA}}\ip{\Cs_{AB}}{\Cs_{AB}}$ and $\ip{\Cs_{AB}}{\Cs_{AB}}_{\Cs_{BB}}$. 

\medskip 

The mutual relations between such families of open subsets uniquely determined by the spectral theory of the ideals generated by inner products, say for example between $O_B^{AB}$ and $O_B^{BC}$, is much more subtle and it will be later formally subsumed in definition~\ref{def: spaceoid} by the existence of a topological groupoid (the base $\Xs$ of our spectral spaceoid) whose $AB$-Hom-sets are span of source-target homeomorphisms $O^{AB}_A \simeq \Xs_{AB}\simeq O^{AB}_B$. 
\xqed{\lrcorner}
\end{remark}

\begin{remark}
Recall that a category $\Ss$ is a \emph{full subcategory} \footnote{
The notion of fullness of a sucategory should not be confused with the ``fullness'' of a C*-category and we will always use the term ``full-subcategory'' to avoid problems.
} of the category $\Cs$ if $\Ss\subset\Cs$ is a subcategory (it is a subset closed by all the operations and sources/targets in $\Cs$) and if $\Ss_{AB}=\Cs_{AB}$, for all $A,B\in\Ss^0$. 

\medskip 

For later usage, we will say that two (full) subcategories $\Ss_1$ and $\Ss_2$ of a C*-category $\Cs$ are in \emph{discrete position} whenever $\Cs_{AB}=\{0\}$, for all $A\in\Ss_1^0$ and $B\in\Ss_2^0$. We use the same terminology for a pair of objects in $\Cs$. 
\xqed{\lrcorner}
\end{remark}

We provide the definition of objects in the first category $\Af$ of our duality. 
\begin{definition}\label{def: Af}
The \emph{category $\Af$ of small commutative C*-categories} consists of the following: 
\begin{itemize}
\item
a class of objects $\Af^0$ that are small commutative C*-categories,
\item
for every pair of objects $\Cs,\Ds\in\Af^0$, a set of morphisms $\Hom_\Af(\Cs;\Ds)$ consisting of all the covariant $*$-functors $\Cs\xrightarrow{\Phi:=(\Phi^0,\Phi^1)}\Ds$ that are: 
\begin{itemize}
\item
\emph{object bijective}: $\Phi^0:\Cs^0\to\Ds^0$ is a bijective function,  
\item 
\emph{non-degenerate}: for every $*$-functor $\Ds\xrightarrow{\omega}\CC$, for every $A,B\in\Cs^0$
\begin{equation}\label{eq: non-trivial}  
\omega|_{\Ds_{\Phi^0(A)\Phi^0(B)}}\neq 0 \ \imp \ \omega\circ\Phi^1|_{\Cs_{AB}}\neq 0,
\end{equation} 
\end{itemize} 
\item 
composition of morphisms $\Cs\xrightarrow{\Psi}\Ds\xrightarrow{\Phi}\Es$ that is the usual composition of functors: 
\begin{equation*}
\Phi\circ\Psi=(\Phi^0,\Phi^1)\circ(\Psi^0,\Psi^1):=(\Phi^0\circ\Psi^0,\Phi^1\circ\Psi^1),
\end{equation*}
\item 
identity morphism $\Cs\xrightarrow{\iota_\Cs}\Cs$, for any object $\Cs\in\Af^0$, consisting of the identity covariant $*$-functor: 
\begin{equation*}
\iota_\Cs:=(\id_{\Cs^0},\id_{\Cs^1}). 
\end{equation*}
\end{itemize}
\end{definition}

\begin{remark}
Notice the omission of the fullness condition on the objects of $\Af$ compared to the original definition utilized in~\cite[section~2]{BCL11}. 

\medskip 

We stress the importance of the additional ``non-triviality'' condition~\eqref{eq: non-trivial} for $*$-functors between C*-categories, that was automatically satisfied in the case of full C*-categories, but that now needs to be explicitly stated. 

\medskip 

Omitting the non-degeneracy condition would force a significant and extremely interesting generalization of spectral theory, where the Gel'fand spectra of unital commutative C*-algebras must necessarily also include the trivial zero characters and spectral spaceoids of non-full C*-categories must contain zero-dimensional fibers. 

\medskip 

We decided to postpone these more delicate developments, that require an in-depth discussion of ``Alexandroff-zerifications'' of Fell-bundles, to a separate forthcoming paper. 
\xqed{\lrcorner}
\end{remark}

\bigskip

To prepare the ground for the next section, we need to mention here material on Fell bundles 
(over topological groupoids is sufficient here), 
a further generalization of the notion of C*-category, see~\cite{Ku98} 

\begin{definition}
A \emph{topological (involutive) category} is an (involutive) category $(\Cs,\circ,\iota,*)$ whose families of objects $\Cs^0$ and morphisms $\Cs^1$ are equipped with a topology making composition $\circ:\Cs^1\times_{\Cs^0}\Cs^1\to \Cs^1$, identity $\iota:\Cs^0\to\Cs^1$ (involution $*:\Cs^1\to\Cs^1$) and source/target maps $\Cs^1\xrightarrow{t,\ s}\Cs^0$ continuous. 

\medskip 

In the case of \emph{topological (involutive) algebroids} we also assume that addition $\Cs^1\times\Cs^1\xrightarrow{+}\Cs^1$ and scalar multiplication $\CC\times\Cs^1\xrightarrow{\cdot}\Cs^1$ are continous. 

\medskip 

In the more general case of \emph{topological (involutive) algebroids over topological (involutive) categories} $\Cs\xrightarrow{\pi}\Xs$, we assume the (involutive) functor $\pi$ to be continuous between the topological (involutive) categories $\Cs$ and $\Xs$ in addition to the continuity of addition and scalar multiplication in $\Cs$. 
\end{definition}
\begin{remark}\label{rem: pair-g}
Every (involutive) category $(\Cs,\circ,\iota,*)$ can equivalently be described as a bundle $\Cs^1\xrightarrow{\pi}\Cs^0\times\Cs^0$, where $\pi: x\mapsto(t(x),s(x))$ is a covariant (involutive) functor from the (involutive) category $\Cs$ to the pair groupoid $\Cs^0\times\Cs^0$, with fibers $\Cs_{AB}=\pi^{-1}((A,B))$, for all $A,B\in\Cs^0$. 

\medskip 

Whenever the (involutive) category $\Cs$ is topological, equipping the pair groupoid $\Cs^0\times\Cs^0$ with the product topology, will provide a topological bundle $\Cs\xrightarrow{\pi}\Cs^0\times\Cs^0$, with a continuous projection (involutive) functor $\pi$. 

\medskip 

A C*-category $\Cs$, with the topology induced by its norm, is always a topological involutive categorical bundle over the discrete pair groupoid $\Cs^0\times\Cs^0$. 
Fell bundles over groupoids, are a vast generalization of such structure, where the topological product pair groupoid $\Cs^0\times\Cs^0$ is substituted by an arbitrary topological groupoid $\Xs$. 
\xqed{\lrcorner}
\end{remark}

In order to define Fell bundles, we need to introduce first Banach bundles. 
The definition of Banach bundle is usually based on the notion of ``Dauns-Hofmann Banach bundle'', where the norm is only assumed to be upper semicontinuous, but this kind of generality is not actually necessary for our specific needs in this paper and hence we actually only consider ``Fell Banach bundles'' (see~\cite[section~1]{DG83} for a complete description of the situation). 

\begin{definition}\label{def: Bbundle}
A \emph{Fell type Banach bundle} $(\Es,\pi,\Xs)$ is a continuous surjective open map  $\Es\xrightarrow{\pi}\Xs$, between topological spaces $\Es$, $\Xs$, 
whose fibers $\Es_p:=\pi^{-1}(p)$ are Banach spaces, for all $p\in \Xs$, and such that the following properties are satisfied: 
\begin{itemize}
\item
addition $\Es\times_\Xs\Es\xrightarrow{+}\Es$ is continuous, where $\Es\times_\Xs\Es:=\left\{(e,f)\in \Es\times\Es \ | \ \pi(e)=\pi(f)\right\}$, 
\item
scalar multiplication $\CC\times\Es\xrightarrow{\cdot}\Es$ is continuous, 
\item
the zero section $\Xs\xrightarrow{\zeta}\Es$ given by $p\mapsto \zeta_p:=0_{\Es_p}$ is continuous, 
\item 
the norm $\Es\xrightarrow{\|\cdot\|}\RR$ is continuous, \footnote{
We stress again that whenever we assume continuity of the norm for the usual Euclidean topology on $\RR$ we obtain the definition of \emph{Banach bundle of Fell type}; 
assuming instead the continuity of the norm for the upper semicontinuity topology on $\RR$ is typical of the more general \emph{Banach bundle of Dauns-Hofmann type} see~\cite[definition~1.1]{DG83}. 
} 
\item 
for all $p\in \Xs$, the norm topology on $\Es_p$ coincides with the subspace topology, 
\item 
for all $p\in\Xs$ and every neighborhood $\U$ of $\zeta_p$, there exists $\epsilon >0$ and an open set $\O\in\Xs$ such that 
\begin{equation}\label{eq: tubular}
\T^\Es_{\epsilon,\O}:=\pi^{-1}(\O)\cap\{e\in\Es \ | \ \|e\|<\epsilon\}\subset \U.
\end{equation} 
\end{itemize}
A (continuous) \emph{section} of the Banach bundle $(\Es,\pi,\Xs)$ is a (continuous) map $\sigma:\Xs\to\Es$ such that $\pi\circ\sigma=\id_\Xs$. 

We denote the family of continuous sections by $\Gamma(\pi)$ and use the notation $p\mapsto \sigma_p\in\Es$, for $p\in \Xs$ and $\sigma\in\Gamma(\pi)$. 

\medskip 

A (Banach) bundle is $\Gamma$-\emph{saturated} if, for all $e\in \Es$, there exists a continuous section $\sigma\in\Gamma(\pi)$ such that $\sigma_{\pi(e)}=e$. 
\end{definition}

Following~\cite{Ta79a,Ta79b}, we introduce a notion of morphism between Banach bundles.

\begin{definition}
A \emph{Takahashi morphism between Banach bundles} $(\Es^1,\pi^1,\Xs^1)\xrightarrow{(f,F)}(\Es^2,\pi^2,\Xs^2)$ consists of: 
\begin{itemize}
\item
a continuous map $\Xs^1\xrightarrow{f}\Xs^2$, 
\item
a fibre-preserving, fibrewise $\CC$-linear, continuous map $\Es^1\xleftarrow{F}f^\bullet(\Es^2)$, where $f^\bullet(\Es^2)$ denotes here the total space of the standard $f$-pull-back of the bundle $\Es^2$, defined as $f^\bullet(\Es^2):=\{(p,e)\in \Xs^1\times\Es^2 \ | \ f(p)=\pi^2(e) \}$. 

\end{itemize}
\end{definition}

\begin{remark}
Notice that for every $\sigma\in\Gamma(\pi)$, the map $F_\sigma:e\mapsto e+\sigma(\pi(e))$ is an homeomorphism of the total space $\Es$ (that is fiber-preserving, but in general not fibrewise linear). Furthermore the addition map $+:\Es\times_\Xs\Es\to\Es$ (defined on the Whitney sum of $\Es$ with itself) and the scalar multiplication $\cdot: \CC\otimes_\Xs \Es\to \Es$ (when factorized via the fibrewise $\CC$-tensor product) are examples of Takahashi morphisms.  
\xqed{\lrcorner}
\end{remark}

\begin{remark}\label{rem: tubular}
The technical condition~\eqref{eq: tubular} above, in the case of $\Gamma$-saturated Banach bundles, is equivalent to the following: 
the topology of the total space $\Es$ coincides with the \emph{tubular topology} that is the topology induced by the family of \emph{tubular neighborhoods} of $e_o\in\Es$ of the form 
$\T^\Es_{\epsilon, \O}(\sigma):=\pi^{-1}(\O)\cap \left\{ e \ | \ \|e-\sigma_{\pi(e)}\|<\epsilon \right\}$, where $\O$ is a neighborhood of $\pi(e_o)$ in $\Xs$, $\epsilon >0$ and $\sigma\in\Gamma(\pi)$ with $\sigma_{\pi(e_o)}=e_o$. 

\medskip 

The existence of a tubular topology induced by bases of tubular neighborhoods of $e_o\in\Es$ as described above is assured, for example, by the conditions in~\cite[theorem~3]{Va95}. 
\xqed{\lrcorner}
\end{remark}

\medskip 

In the present work, we will exclusively deal with (Fell type) Banach bundles $(\Es,\pi,\Xs)$ over locally compact Hausdorff spaces $\Xs$. Whenever $\Xs$ is not already compact it is natural to consider, in place of continuous sections $\Gamma(\Es)$, the more restrictive set $\Gamma_o(\Es)$ of continuous sections vanishing at infinity. 
\begin{definition}\label{def: sec-to-0}
Given a Banach bundle $(\Es,\pi,\Xs)$ on a locally compact Hausdorff space $\Xs$ we say that a continuous section $\sigma\in\Gamma(\Es)$ is \emph{vanishing at infinity} if: 
for every $\epsilon>0$ and every compact set $\Ls\subset \Xs$, there exists a compact set $\Ks\subset \Xs$ such that $\sigma(\Xs-\Ks)\subset \T^\Es_{\epsilon,(\Ls-\Ks)}$.\footnote{
Equivalently, $\lim_{p\to\infty}\|\sigma(p)\|_\Es=0$: for every $\epsilon>0$ there exists $\Ks_\epsilon$ compact, such that $p\in (\Xs-\Ks_\epsilon)\imp\|\sigma(p)\|_\Es<\epsilon$. 
} 

\medskip 

The family of continuous sections of $\Es$ vanishing at infinity is denoted by $\Gamma_o(\Es)$. 

\medskip 

We say that $\Es$ is $\Gamma_o$-saturated if, for all $e\in\Es$ there exists $\sigma\in\Gamma_o(\Es)$ such that $\sigma(\pi(e))=e$. 
\end{definition}

Whenever the Banach bundle $(\Es,\pi,\Xs)$ is $\Gamma_o(\pi)$-saturated, there might potentially be two different topologies on the total space $\Es$. 
One is the original tubular topology induced by the family of continuous sections $\Gamma(\pi)$: that is the topology of $(\Es,\pi,\Xs)$ considered as a saturated Banach bundle.  
The other is the tubular topology induced by the continuous sections vanishing at infinity $\Gamma_o(\pi)$. 
These two topologies in general might differ, but they actually coincide for $\Xs$ Hausdorff as described in the following technical lemma. 
\begin{lemma}\label{lem: lcpr} 
Let $(\Es,\pi,\Xs)$ a $\Gamma_o(\pi)$-saturated Banach bundle. 
If the base space $\Xs$ is locally compact and pre-regular, whenever both tubular topologies induced by $\Gamma(\pi)$ and $\Gamma_o(\pi)$ exist,\footnote{
For this is enough that the conditions in~\cite[theorem 3]{Va95} are satisfied for both $\Gamma(\pi)$ and $\Gamma_o(\pi)$.
} 
they necessarily coincide.
\end{lemma}
\begin{proof}
It is known that $\Xs$ locally compact pre-regular implies that the Alexandroff extension $\hat{\Xs}:=\Xs\uplus\{\infty\}$ is compact pre-regular and hence normal. 

\medskip 

Any tubular neighborhood of a point $e\in\Es$ for the tubular topology induced by the family $\Gamma_o(\pi)$ is necessarily a tubular neighborhood for the tubular topology induced by the family $\Gamma(\pi)$. 

\medskip 

Let $e\in\Es$ and $\sigma\in\Gamma(\pi)$ a continuous section such that $\sigma(\pi(e))=e$. 
Consider, for $\epsilon>0$ and $O\subset \Xs$ an open set such that $\pi(e)\in O$, the basic tubular neighborhood $U^\epsilon_O:=\pi^{-1}(O)\cap \{f\in \Es \ | \ \|f-\sigma(\pi(f))\|<\epsilon\}$ of $e$ in the topology induced by $\Gamma(\pi)$. 

Since $\Xs$ is locally compact pre-regular, there exists a compact neighborhood $K$ of $\pi(e)$ such that $K\subset O$. Since $\hat{\Xs}$ is normal, by Urysohn lemma, we can find a continuous function $\phi:\hat{\Xs}\to \RR$ such that $\phi|_K$ is constant equal to $1_\RR$ and $\phi(o)=0_\RR$. It follows that the new section $\phi\cdot \sigma$ is a continuous section vanishing to zero at infinity in $\Gamma_o(\pi)$ that on the compact neighborhood $K$ coincides with $\sigma$: $\phi\cdot\sigma|_K=\sigma_K$ and hence the $\Gamma_o(\pi)$ tubular neigborhood $\pi^{-1}(K)\cap\{f\in \Es \ | \ \|f-\phi\cdot\sigma(\pi(f))\|<\epsilon\}$ is contained in the previous $\Gamma(\pi)$ neighborhood $U^\epsilon_O$ that implies that $U^\epsilon_O$ is also a neighborhood of $e$ in the tubular topology induced by $\Gamma_o(\pi)$. 
\end{proof}

\begin{remark}
Any Banach line-bundle is actually a locally trivial normed vector bundle. 
\xqed{\lrcorner}
\end{remark}

We will use the following definition of unital Fell bundle over a groupoid~\cite{Ku98}.  
\begin{definition}\label{def: Fell} 
A \emph{unital Fell bundle over a topological groupoid} $\Xs$ is a Banach bundle $(\Es,\pi,\Xs)$ that is also a topological involutive $\CC$-algebroid over the topological groupoid $\Xs$ and that also satisfies the following: 
\begin{itemize} 
\item 
$\|e_1\circ e_2\|\leq \|e_1\|\cdot\|e_2\|$, for all $(e_1,e_2)\in\Es\times_\Xs\Es$, 
\item 
$\|e^*\circ e\|=\|e\|^2$, for all $e\in\Es$, 
\item 
for all $e\in\Es_{\pi(e)}$, the element $e^*\circ e$ is positive in the C*-algebra $\Es_{\pi(e^*\circ e))}$. 
\end{itemize}
A Fell bundle is \emph{rank-one} (respectively \emph{rank-zero}) if $\Es_{p}$ is one-dimensional (zero-dimensional) $\forall p\in\Xs$.

\medskip 

A \emph{Takahashi morphism of Fell bundles} $(\Es,\pi,\Xs)\xrightarrow{(f,F)}(\Fs,\nu,\Ys)$   
over topological groupoids $\Xs$ and $\Ys$, consists of the following two ingredients: 
\begin{itemize}
\item
a continuous $*$-functor $\Xs\xrightarrow{f}\Ys$, 
\item 
a continuous fibre-preserving fibrewise-linear $*$-functor $\Es\xleftarrow{F}f^\bullet(\Fs)$. 
\end{itemize}
\end{definition}

\section{The Category $\Tf$ of Non-full Spaceoids} \label{sec: 3}

We present here a version of topological spaceoid (introduced in~\cite[definition~3.2]{BCL11} for full commutative C*-categories) that is adapted to the case of non-full commutative \hbox{C*-cat}\-egories. 
The main visible difference consists in the replacement of the product topological groupoid $\Delta_\Xs\times\Rs_\Os$ (that is a groupoid fibered over the discrete pair groupoid $\Rs_\Os$ with fibers that are all isomorphic to the trivial groupoid $\Delta_\Xs$ for a compact Hausdorff space $\Xs$) with a more general groupoid $\Xs$ fibered over $\Rs_\Os$. 

\medskip 

The actual motivation for the present definition comes from the construction of the ``spectral non-full spaceoid'' of a commutative C*-category that will be considered in section~\ref{sec: 5}. 
\begin{definition} \label{def: spaceoid} 
A \emph{non-full spaceoid (over the non-full base spaceoid $\Xs_{\theta,\Os}$)} is a saturated unital Fell line-bundle $(\Es,\pi,\Xs_{\theta,\Os})$, where the (non-full) base spaceoid $\Xs_{\theta,\Os}:=(\Xs\xrightarrow{\theta}\Rs_\Os)$ is a topological groupoid $\Xs$ fibered over the discrete pair groupoid $\Rs_\Os:=\Os\times\Os$ of a discrete set $\Os$,~\footnote{
With some abuse of notation we denote by $AB:=(A,B)\in\Rs_\Os$ the elements of the discrete pair groupoid of $\Os$. 
}
with fibers $\Xs_{A}:=\Xs_{AA}$, for $A\in\Os$, that are compact Hausdorff trivial groupoids of identities, and source/target maps $\Xs_A\xleftarrow{t_{AB}}\Xs_{AB}\xrightarrow{s_{AB}}\Xs_B$ that are homeomorphisms from $\Xs_{AB}:=\theta^{-1}(AB)$ to open subsets.~\footnote{
The source and target of $p\in\Xs_{AB}$ are the unique identities $s(p)\in \Xs_B$ and $t(p)\in\Xs_A$ that are left/right composable with $p$.
} 
\end{definition}
Since the definition above is fundamental, we give a more detailed description of its ingredients:
\begin{itemize}
\item 
$\Os$ is a discrete topological space and $\Rs_\Os:=\Os\times\Os$ denotes the discrete pair groupoid with objects in $\Os$, 
\item 
$\Xs$ is a topological groupoid, $\Xs\xrightarrow{\theta}\Rs_\Os$ is a continuous $*$-functor with fibers $\Xs_{AB}:=\theta^{-1}(AB)$, $AB\in\Rs_\Os$,\footnote{ 
Recall that in a groupoid $\Xs$ the involution $*$ is the inverse map $p\mapsto p^{-1}$ of morphisms $p\in \Xs^1$. 
} 
\item 
$\Xs_A:=\Xs_{AA}$ is a compact Hausdorff topological groupoid consisting only of identity isomorphisms,\footnote{
Here, with abuse of notation, we identify the compact Hausdorff space $\Xs_A$ of objects of the group $\Xs_{AA}$ with $\Xs_{AA}$ itself: every object $p\in\Xs_A$ corresponding to its identity $\iota(p):=(p,p)\in\Xs_{AA}$. 
} 
\item 
the sources and target maps $\Xs_A\xleftarrow{t_{AB}}\Xs_{AB}\xrightarrow{s_{AB}}\Xs_B$ are homeomorphisms onto open subsets, 
\item 
$\Es\xrightarrow{\pi}\Xs$ is a saturated unital Fell line-bundle 
\footnote{
Recall that in our case (as for every Banach line-bundle) the Banach bundle $(\Es,\pi,\Xs)$ is of Fell-type, as in definition~\ref{def: Bbundle}.}
over the topological groupoid $\Xs$ (see definition~\ref{def: Fell}). 
\end{itemize}

The following remark is important to formalize the notation that we will be subsequently using. 
\begin{remark}\label{rem: p}
With a mild abuse of notation, let $p_{AB}\in\Xs_{AB}$ be an arbitrary point of the $AB$-Hom-set of the groupoid $\Xs$; its unique source and target elements will be written as $\Xs_A\ni p_A\xleftarrow{t_{AB}}p_{AB}\xrightarrow{s_{AB}}p_B\in\Xs_B$. 

\medskip 

More generally, whenever $p_{AC}\in\Xs_{AB}\circ\Xs_{BC}\subset \Xs_{AC}$, from definition~\ref{def: spaceoid}, we have that there exists a unique pair of elements, denoted as $(p_{AB},p_{BC})\in\Xs_{AB}\times\Xs_{BC}$, such that $p_{AC}=p_{AB}\circ p_{BC}$.  

\medskip 

Similarly, whenever $p_{AB}\in\Xs_{AB}$, there exists a unique element $p_{BA}\in\Xs_{BA}$ such that $p_{BA}=(p_{AB})^*$, furthermore $p_{AB}\circ p_{BA}=p_A$ and $p_{BA}\circ p_{AB}=p_B$. 

\medskip 

As a consequence, for any given point $p_{XY}\in\Xs=\bigcup_{AB\in\Rs_\Os}{\Xs_{AB}}$ there exists a unique maximal pair subgroupoid $\Xs^p\subset\Xs$ (here denoted on purpose with a sup-script $p$, omitting the ``object indexes'' of the point $p_{XY}\in\Xs_{XY}$) and, for $A,B\in\Os^p:=\theta(\Xs^p)$, the point $p_{AB}$ is the unique element in the singleton set $\Xs^p_{AB}:=\Xs^p\cap \Xs_{AB}$. 

\medskip 

In practice, whenever the same letter, say $p$, is used to denote points $p_{AB}$ and $p_{CD}$ in different Hom-sets of the groupoid $\Xs$, we always assume that they are related by at least one chain of, necessarily unique, composable elements $p_{AB}=p_{AX_1}\circ\cdots\circ p_{X_nB}$. Such notational choice is self-consistent due to the ``triviality of holonomies'': whenever we have two chains of composable elements $p_{X_1X_2}\circ\cdots\circ p_{X_{n-1}X_n}$ and $q_{Y_1Y_2}\circ\cdots\circ q_{Y_{m-1}Y_m}$ such that $p_{X_1X_2}=q_{Y_1Y_2}$, if $X_{n-1}X_n=Y_{m-1}Y_m$, we necessarily have $p_{X_{n-1}X_n}=q_{Y_{m-1}Y_m}$.  
\xqed{\lrcorner}
\end{remark}

The relevant notion of morphism is similar to that used in~\cite[page 589]{BCL11} that was essentially inspired by~\cite{Ta79a,Ta79b} applied to the case of Fell line-bundles. 

\begin{definition}\label{def: m-spaceoid}
A \emph{morphism of non-full spaceoids} $(\Es^1,\pi^1,\Xs^1_{\theta^1,\Os^1})\xrightarrow{(f,F)}(\Es^2,\pi^2,\Xs^2_{\theta^2,\Os^2})$ 
is actually a special kind of Takahashi morphism of Fell bundles $(\Es^1,\pi^1,\Xs^1_{\theta^1,\Os^1})\xrightarrow{(f,F)}(\Es^2,\pi^2,\Xs^2_{\theta^2,\Os^2})$.  
Specifically a morphism is given by a pair $(f,F)$, with $f:=(f_\Xs,f_\Rs)$, where we have that:
\begin{itemize}
\item 
$\Rs_{\Os^1}\xrightarrow{f_\Rs:=(f_\Rs^0,f_\Rs^1)}\Rs_{\Os^2}$, denoted as 
$AB\overset{f^1_\Rs}{\mapsto} f^0_\Rs(A)f^0_\Rs(B)$,  
is an isomorphism of discrete pair groupoids,~\footnote{
This is equivalent to say that $f_\Rs:=(f_\Rs^0,f_\Rs^1)$ is functor with $f_\Rs^0:\Os^1\to\Os^2$ that is bijective. 
}
\item 
$\Xs^1\xrightarrow{f_\Xs}\Xs^2$, is a continuous $*$-functor between topological groupoids 
such that $\theta^2\circ f_\Xs=f_\Rs\circ\theta^1$, 
that is also \emph{converging at infinity}  
\begin{equation*}
\forall AB\in\Rs_{\Os^1}, \ \forall \Ks\subset \Xs^2_{f_\Rs(AB)} \ \text{compact}, \  \exists \Ls \subset \Xs^1_{AB} \ \text{compact} \st f_\Xs(\Xs^1_{AB}-\Ls)\subset \Xs^2_{f_\Rs(AB)}-\Ks,  
\end{equation*} 
\item
$F:f^\bullet(\Es^2)\to\Es^1$ is a continuous fibre-preserving fibrewise-linear $*$-functor that is \emph{vanishing at infinity} 
\begin{equation*}
\forall AB\in \Rs_{\Os^1}, \ \forall \epsilon>0, \ \forall \Ks\subset \Xs^1_{AB} \ \text{compact}, \ 
\exists \delta>0, \ \exists \Ls \subset\Xs^1_{AB} \ \text{compact} \st F(\T^{f^\bullet(\Es^2)}_{\delta,\Xs^1_{AB}-\Ls}) \subset \T^{\Es^1}_{\epsilon,\Xs^1_{AB}-\Ks}.   
\end{equation*}
\end{itemize}
\end{definition}

\begin{remark}
The additional ``convergence to zero at infinity requirements'' in the previous definition is meant to guarantee that morphisms of spaceoids preserve continuous sections vanishing at infinity: 
\begin{equation*}
\sigma\in\Gamma_o(\Es^2) \imp F\circ f^\bullet(\sigma)\in\Gamma_o(\Es^1). 
\end{equation*} 

We anticipate that such property actually means that a morphism between spaceoids actually uniquely lifts to certain universal {\it Alexandroff zerification extensions} of the involved spaceoids, with zero-dimensional fibers over the zero morphisms at infinity; such point of view will be pursued in detail in forthcoming work. 
\xqed{\lrcorner}
\end{remark}

\begin{definition}\label{def: Tf}
The category $\Tf$ of non-full spacoids consists of the following data: 
\begin{itemize}
\item
a class of objects $\Tf^0$ that are the non-full spaceoids as presented in definition~\ref{def: spaceoid},
\item
a class of morphisms $\Tf^1$ that are the Takahashi morphisms between non-full spaceoids as in~\ref{def: m-spaceoid}, 
\item
a composition of morphisms $\Es^1\xrightarrow{(g,G)}\Es^2\xrightarrow{(f,F)}\Es^3$ defined by
$\Es^1\xrightarrow{(f,F)\circ(g,G)}\Es^3$: 
\begin{equation*}
(f,F)\circ(g,G):=\left(f\circ g, \ G\circ g^\bullet(F)\circ \Theta^{\Es^3}_{f,g}\right), 
\end{equation*}
where, 
$\Theta^{\Es^3}_{f,g}:(f\circ g)^\bullet(\Es^3)\to g^\bullet(f^\bullet(\Es^3))$ 
is the canonical isomorphism of standard pull-backs and 
\begin{equation*}
f\circ g=(f_\Xs,f_\Rs)\circ(g_\Xs,g_\Rs):= (f_\Xs\circ g_\Xs,f_\Rs\circ g_\Rs), 
\end{equation*}
\item
an identity associating to every object $(\Es,\pi,\Xs_{\theta,\Os})$ the identity morphism $(\Es,\pi,\Xs_{\theta,\Os})\xrightarrow{\iota_{\Es}}(\Es,\pi,\Xs_{\theta,\Os})$: 
\begin{equation*}
\quad \quad 
\iota_\Es:=(\iota_\Xs,\Theta_\Es),
\quad \quad 
\iota_\Xs:=(\id_\Xs,\id_\Rs):\Xs_{\Os,\theta}\to \Xs_{\Os,\theta}
\end{equation*} 
where we denote by $\Theta_\Es:\iota_\Xs^\bullet(\Es)\to\Es$ the canonical isomorphism between $\iota_\Xs^\bullet(\Es)=\id_\Xs^\bullet(\Es)$ and $\Es$. 
\end{itemize}
\end{definition}

\begin{proposition}
The data listed in definition~\ref{def: Tf} constitute a category $\Tf$. 
\end{proposition}
\begin{proof}
We must prove associativity of composition and unitality axioms. 

\medskip 

Consider the morphisms $(\Es^1,\pi^1,\Xs^1_{\Os^1,\theta^1})\xrightarrow{(h,H)}(\Es^2,\pi^2,\Xs^2_{\Os^2,\theta^2}) \xrightarrow{(g,G)}(\Es^3,\pi^3,\Xs^3_{\Os^3,\theta^3}) \xrightarrow{(f,F)}(\Es^4,\pi^4,\Xs^4_{\Os^4,\theta^4})$: 
\begin{align*}
(f,F)\circ\left((g,G)\circ(h,H)\right)
&=(f,F)\circ \left(g\circ h,\ H\circ h^\bullet(G)\circ\Theta^{\Es^3}_{g,h}\right)
\\
&=\left(f\circ(g\circ h), \ 
H\circ h^\bullet(G)\circ\Theta^{\Es^3}_{g,h}\circ (g\circ h)^\bullet(F)\circ\Theta^{\Es^4}_{f,(g\circ h)}
\right)
\\
&
=
\left((f\circ g)\circ h, \ 
H\circ h^\bullet(G)\circ h^\bullet(g^\bullet(F))\circ h^\bullet(\Theta^{\Es^4}_{f,g})
\circ\Theta^{\Es^4}_{(f\circ g), h} \right)
\\
&=\left((f\circ g)\circ h, \ 
H\circ h^\bullet(G\circ g^\bullet(F)\circ\Theta^{\Es^4}_{f,g})
\circ\Theta^{\Es^4}_{(f\circ g), h} \right)
\\
&=\left((f\circ g)\circ h, \ 
H\circ h^\bullet(G\circ g^\bullet(F)\circ\Theta^{\Es^4}_{f,g})
\circ\Theta^{\Es^4}_{(f\circ g), h}
\right)
\\
&=\left(f\circ g,\ G\circ g^\bullet(F)\circ\Theta^{\Es^4}_{f,g}\right)\circ(h,H)
=\left((f,F)\circ(g,G)\right)\circ(h,H);
\end{align*}
finally given the morphisms $(\Es^1,\pi^1,\Xs^1_{\Os^1,\theta^1})\xrightarrow{(g,G)}(\Es,\pi,\Xs_{\Os,\theta})\xrightarrow{\iota_\Es}(\Es,\pi,\Xs_{\Os,\theta})\xrightarrow{(f,F)}(\Es^2,\pi^2,\Xs^2_{\Os^2,\theta^2})$ we have:
\begin{align*}
(f,F)\circ (\iota_\Xs,\Theta_\Es)
=(f\circ\iota_\Xs, \ \Theta_\Es\circ\iota_\Xs^\bullet(F)\circ\Theta^{\Es^2}_{f,\iota_\Xs})
=(f, F), 
\\ 
(\iota_\Xs,\Theta_\Es)\circ(g,G)=(\iota_\Xs\circ g, \ 
G\circ g^\bullet(\Theta_\Es)\circ\Theta^{\Es}_{\iota_\Xs, g})
=(g,G),
\end{align*}
using the properties of canonical isomorphisms of pull-backs.  

\medskip 

Notice, in the case of identity morphisms, that $\iota_\Xs$ is continuous and converging at infinity and $\Theta^\Es$ is continuous vanishing at infinity; similarly the composition of continuous $*$-functors converging at infinity is a continuous $*$-functor converging at infinity and furthermore the composition of continuous morphisms vanishing at infinity is continuous and vanishing at infinity. 
\end{proof} 

\section{The Section Functor $\Gamma$} \label{sec: 4}

To every non-full spaceoid we can functorially associate a commutative C*-category of certain sections. 
\begin{definition}\label{def: Gamma}
The \emph{section functor} $\Tf\xrightarrow{\Gamma}\Af$, from the category $\Tf$ introduced in definition~\ref{def: Tf} to the category $\Af$ introduced in definition~\ref{def: Af}, is defined as follows: 
\begin{itemize}
\item
to every object $(\Es,\pi,\Xs_{\theta,\Os})\in\Tf^0$, $\Gamma$ associates the C*-category $\Gamma(\Es):=\Gamma(\Es,\pi,\Xs_{\theta,\Os})\in\Af$ given by: 
\begin{itemize}
\item
objects $\Gamma(\Es)^0:=\Os$ coinciding with the elements of the set $\Os$,  
\item
for $A,B\in\Os$, morphisms from $A$ to $B$ consisting of the continuous sections vanishing at infinity~\footnote{
Since sources and targets are homeomorphisms onto open subsets of compact Hausdorff spaces, it follows that $\Xs_{AB}$ is a locally compact Hausdorff space and hence (having Alexandroff compactification with point at infinity $\infty_{AB}$) it makes sense to define sections vanishing at infinity in the usual way, as in definition~\ref{def: sec-to-0}:  
for every $\epsilon>0$ there exists a compact subset $\Ks_\epsilon\subset \Xs_{AB}$ such that $\|\sigma(p)\|_\Es<\epsilon$, whenever $p\in \Xs_{AB}-\Ks_\epsilon$. 
}
\begin{equation*}
\Gamma(\Es)^1_{AB}:=\left\{\sigma:\Xs_{AB}\to \Es_{AB} \ | \ \pi_{AB}\circ\sigma=\id_{\Xs_{AB}}, \ \sigma \ \text{continuous}, \ \lim_{p_{AB}\to \infty_{AB}}\|\sigma(p_{AB})\|_\Es=0\right\}
\end{equation*}
of the $AB$-restriction $(\Es_{AB},\pi_{AB},\Xs_{AB}):=(\Es\cap\pi^{-1}({\Xs_{AB}}),\pi|^{\Xs_{AB}}_{\Es_{AB}},\Xs_{AB})$ of the spaceoid.
\item 
addition $\Gamma(\Es)_{AB}+\Gamma(\Es)_{AB}\subset\Gamma(\Es)_{AB}$, for all $A,B\in\Os$, that for all $\sigma,\rho\in\Gamma(\Es)_{AB}$, is given by:
\begin{equation*}
(\sigma+\rho)(p_{AB}):=\sigma(p_{AB})+\rho(p_{AB}), \quad \forall p_{AB}\in\Xs_{AB}, 
\end{equation*}
\item 
scalar multiplication $\CC\cdot\Gamma(\Es)_{AB}\subset\Gamma(\Es)_{AB}$, for $A,B\in\Os$, that for  $(\alpha,\sigma)\in\CC\times\Gamma(\Es)_{AB}$, is given by:
\begin{equation*}
(\alpha\cdot \sigma)_{p_{AB}}:=\alpha\cdot (\sigma(p_{AB})), \quad \forall p_{AB}\in\Xs_{AB}, 
\end{equation*}
\item 
composition $\Gamma(\Es)_{AB}\circ\Gamma(\Es)_{BC}\subset \Gamma(\Es)_{AC}$, for $A,B,C\in\Os$ defined $\forall (\sigma,\rho)\in \Gamma(\Es)_{AB}\times\Gamma(\Es)_{BC}$ by: 
\begin{equation*}
(\sigma\circ \rho)(p_{AC}):=
\begin{cases}
\sigma(p_{AB})\circ_\Es \rho(p_{BC}), \quad \quad \forall p_{AC}\in\Xs_{AB}\circ\Xs_{BC}\subset\Xs_{AC}, 
\\
\zeta^\Es_{AC}(p_{AC}):=0_{\Es_{p_{AC}}}, \quad \quad\forall p_{AC}\in \Xs_{AC}-\Xs_{AB}\circ\Xs_{BC},
\end{cases}
\end{equation*}
where we make use of the notation justified in remark~\ref{rem: p}. 
\item 
involution $\Gamma(\Es)_{AB}^*\subset \Gamma(\Es)_{BA}$, for $A,B\in\Os$ given by: 
\begin{equation*}
(\sigma^*)(p_{BA}):=\sigma(p_{AB})^*, \quad \forall p_{AB}\in \Xs_{AB}, \ \forall \sigma\in\Gamma(\Es)_{AB}, 
\end{equation*}
\item 
identities $\iota_A$, for all $A\in\Os$ given by the constant unit sections  
\begin{equation*}
\iota^{\Es}_A(p_{AA}):=1_{\Es_{p_{AA}}}, \quad \forall p_A\in \Xs_A, 
\end{equation*}
where $1_{\Es_{p_{AA}}}$ denotes the identity element of the unital C*-algebra $\Es_{p_{AA}}$ of the Fell bundle $\Es$, 
\item 
the norm $\|\cdot\|:\Gamma(\Es)\to\RR$ is defined, for all $A,B\in\Os$ and all $\sigma\in\Gamma(\Es)_{AB}$ by: 
\begin{equation}\label{eq: norm-G}
\|\sigma\|:=\sup_{p_{AB}\in\Xs_{AB}}\|\sigma(p_{AB})\|_\Es. 
\end{equation}
\end{itemize}
\item 
to every morphism of spaceoids $(\Es^1,\pi^1,\Xs^1)\xrightarrow{(f,F)}(\Es^2,\pi^2,\Xs^2)$ in $\Tf^1$, $\Gamma$ associates the morphism of C*-categories $\Gamma(\Es^1)\xleftarrow{\Gamma_{(f,F)}}\Gamma(\Es^2)$ given by: \footnote{ 
That $\Gamma_{(f,F)}$ is a well-defined function is described in the subsequent proposition. 
}
\begin{equation*}
\Gamma_{(f,F)}: \sigma\mapsto F(f^\bullet(\sigma)), \quad \forall \sigma\in \Gamma(\Es^2).   
\end{equation*} 
\end{itemize}
\end{definition} 

The following is our ``non-full version'' of~\cite[proposition~4.1]{BCL11}. 
\begin{proposition}
The previous definitions provide a well-defined contravariant functor $\Tf\xrightarrow{\Gamma}\Af$. 
\end{proposition}
\begin{proof} 
We see that $\Gamma(\Es)$ in definition~\ref{def: Gamma} is a well-defined commutative small C*-category for any $(\Es,\pi,\Xs_{\theta,\Os})$. 

\medskip 

From definition~\ref{def: Fell}, all the operations of addition, scalar multiplication, composition, involution, additive and multiplicative identity are continuous in the unital Fell-line bundle $(\Es,\pi,\Xs)$ and hence, by composition of continuous functions, for all $A,B,C\in\Os$ all $\sigma_1,\sigma_1,\sigma\in\Gamma(\Es)_{AB}, \rho\in\Gamma(\Es)_{BC}$ and all $\alpha\in\CC$, we have that $\sigma_1+\sigma_2$, $\alpha \cdot \sigma$, are continuous sections of $\Es_{AB}$; $\sigma\circ\rho$ is a continuous section of $\Es_{AC}$; $\sigma^*$ is a continuous section of $\Es_{BA}$; $\zeta^\Es_A$ and $\iota^\Es_A$ are all continuous sections of $\Es_{AA}$. Since $\Xs_A$ is already compact (hence all continuous sections vanish at infinity), $\zeta^\Es_A,\ \iota^\Es_A\in\Gamma(\Es)_{AA}$; furthermore  $\sigma_1+\sigma_2,\ \alpha\cdot \sigma\in\Gamma(\Es)_{AB}$ because
\begin{equation*} \lim_{p_{AB}\to\infty_{AB}}\|\sigma_1+\sigma_2\|\leq\lim_{p_{AB}\to\infty_{AB}}\|\sigma_1\|+\lim_{p_{AB}\to\infty_{AB}}\|\sigma_2\|=0, 
\quad  
\lim_{p_{AB}\to\infty_{AB}}\|\alpha\cdot\sigma\|=|\alpha|\cdot\lim_{p_{AB}\to\infty_{AB}}\|\sigma\|=0;
\end{equation*} 
finally $\sigma^*\in\Gamma(\Es)_{BA}$ and $\sigma\circ\rho\in\Gamma(\Es)_{AC}$ because: 
\begin{equation*}
\lim_{p_{AB}\to\infty_{BA}}\|\sigma^*\|=\lim_{p_{AB}\to\infty_{AB}}\|\sigma\|=0, 
\quad 
\lim_{p_{AC}\to\infty_{AC}}\|\sigma\circ\rho\|\leq \lim_{p_{AB}\to\infty_{AB}}\|\sigma\|\cdot\lim_{p_{BC}\to\infty_{BC}}\|\rho\|=0.
\end{equation*}

With the above well-defined operations, $\Gamma(\Es)$ is an involutive complex unital algebroid: 
\begin{itemize}
\item 
for all $A,B\in\Os$, the Hom-set $\Gamma(\Es)_{AB}$ with pointwise addition and scalar multiplication, is a complex vector space with zero vector $0_{\Gamma(\Es)_{AB}}=\zeta^\Es|_{AB}$ and opposite given as $(-\sigma)(p_{AB})=-(\sigma(p_{AB}))$, for all $p_{AB}\in\Xs_{AB}$, for all $\sigma\in\Gamma(\Es)_{AB}$,  
\item
the composition $\circ:\Gamma(\Es)_{AB}\times\Gamma(\Es)_{BC}\to\Gamma(\Es)_{AC}$, for $A,B,C\in\Os$, is $\CC$-bilinear:  
\begin{itemize}
\item 
whenever $p_{AC}\in\Xs_{AC}-\Xs_{AB}\circ\Xs_{BC}$, we have, $\forall \sigma,\sigma_1,\sigma_2\in\Gamma(\Es)_{AB}, \ \rho,\rho_1,\rho_2\in\Gamma(\Es)_{BC}$,  
\begin{align*}
(\sigma\circ (\rho_1+\rho_2))(p_{AC})=\zeta^\Es_{AC}(p_{AC})=\sigma\circ\rho_1 (p_{AC})+\sigma\circ\rho_2 (p_{AC})
=(\sigma\circ\rho_1+\sigma\circ\rho_2)(p_{AC}),
\\
((\sigma_1+\sigma_2)\circ\rho)(p_{AC}) 
=\zeta_{AC}^\Es(p_{AC})
=\sigma_1\circ\rho(p_{AC})+\sigma_2\circ\rho(p_{AC})
=(\sigma_1\circ\rho+\sigma_2\circ\rho)(p_{AC}), 
\end{align*}
\item
otherwise, whenever $\forall p_{AC}\in\Xs_{AB}\circ\Xs_{BC}$, we obtain, for same sections,  
\begin{align*}
(\sigma\circ (\rho_1+\rho_2))(p_{AC})
&=\sigma(p_{AB})\circ (\rho_1+\rho_2)(p_{BC})
=\sigma(p_{AB})\circ \rho_1(p_{BC})+\sigma(p_{AB})\circ\rho_2(p_{BC})
\\
&=(\sigma\circ\rho_1+\sigma\circ\rho_2)(p_{AC}), 
\\ 
((\sigma_1+\sigma_2)\circ\rho)(p_{AC}) 
&=(\sigma_1+\sigma_2)(p_{AB})\circ\rho(p_{BC})
=\sigma_1(p_{AB})\circ\rho(p_{BC})+\sigma_2(p_{AB})\circ\rho(p_{BC})
\\
&=(\sigma_1\circ\rho+\sigma_2\circ\rho)(p_{AC}),
\end{align*}
\end{itemize} 
\item  
the composition is also unital and associative: 
\begin{itemize}
\item 
for any section $\sigma\in\Gamma(\Es)_{AB}$, for all $p_{AB}\in\Xs_{AB}$, since $\Xs_A\circ\Xs_{AB}=\Xs_{AB}=\Xs_{AB}\circ\Xs_B$, we have:
\begin{align*}
(\sigma\circ \iota^\Es_B)(p_{AB})=\sigma(p_{AB})\circ\iota^\Es_{B}(p_{BB})=\sigma(p_{AB})
=\iota^\Es_A(p_{AA})\circ\sigma(p_{AB});  
\end{align*}
\item 
whenever $p_{AD}\in\Xs_{AB}\circ\Xs_{BC}\circ\Xs_{CD}$, for all $\sigma\in\Gamma(\Es)_{AB}$, $\rho\in\Gamma(\Es)_{BC}$, $\tau\in\Gamma(\Es)_{CD}$, 
\begin{align*}
((\sigma\circ\rho)\circ\tau)(p_{AD})
&=(\sigma\circ\rho)(p_{AC})\circ\tau(p_{CD})
=\sigma(p_{AB})\circ\rho(p_{BC})\circ\tau(p_{CD})
\\
&=\sigma(p_{AB})\circ(\rho\circ\tau)(p_{BD})
=(\sigma\circ(\rho\circ\tau))(p_{AD}),
\end{align*}
\item 
whenever $p_{AD}\in\Xs_{AD}-\Xs_{AB}\circ\Xs_{BC}\circ\Xs_{CD}$, since $(\Xs_{AB}\circ\Xs_{BC})\circ\Xs_{CD}=\Xs_{AB}\circ(\Xs_{BC}\circ\Xs_{CD})$, by definition of composition and its $\CC$-linearity, both compositions $(\sigma\circ\rho)\circ\tau(p_{AD})$ and $\sigma\circ(\rho\circ\tau)(p_{AD})$ coincide with the zero section $\zeta_{AD}^\Es(p_{AD}):=0_{\Es_{p_{AD}}}$, 
\end{itemize}
\item 
the involution $\Gamma(\Es)_{AB}\xrightarrow{*}\Gamma(\Es)_{BA}$, for $A,B\in\Os$, is involutive, conjugate-linear: 
\begin{align*}
((\sigma^*)^*)(p_{AB})&=(\sigma^*(p_{BA}))^*=((\sigma(p_{AB}))^*)^*=\sigma(p_{AB}), \quad \forall \sigma\in\Gamma(\Es)_{AB}, \ p_{AB}\in\Xs_{AB}, 
\\
(\sigma+\rho)^*(p_{BA})&=(\sigma+\rho)(p_{AB})^*=(\sigma(p_{AB})+\rho(p_{AB}))^*
=\sigma(p_{AB})^*+\rho(p_{AB})^*=\sigma^*(p_{BA})+\rho^*(p_{BA}) \\
&=(\sigma^*+\rho^*)(p_{BA}), \quad \forall \sigma,\rho\in\Gamma(\Es)_{AB}, \ p_{AB}\in\Xs_{AB} 
\\
(\alpha\cdot\sigma)^*(p_{AB})&=((\alpha\cdot\sigma)(p_{BA}))^*=\cj{\alpha}\cdot (\sigma(p_{BA}))^*
=\cj{\alpha}\cdot\sigma^*(p_{AB}), 
\ \ \forall \alpha\in\CC, \ \sigma\in\Gamma(\Es)_{AB}, \ p_{AB}\in\Xs_{AB},
\end{align*}
\item 
the involution is antimultiplicative: for all $A,B,C\in\Os$, for all $\sigma\in\Gamma(\Es)_{AB}$, and $\rho\in\Gamma(\Es)_{BC}$, 
\begin{itemize}
\item
since $(\Xs_{AB}\circ\Xs_{BC})^{*}=\Xs_{CB}\circ\Xs_{BA}$, we have $p_{AC}\in\Xs_{AB}\circ\Xs_{BC}\iff p_{CA}\in\Xs_{CB}\circ\Xs_{BA}$ and in this case: 
\begin{align*}
(\sigma\circ\rho)^*(p_{CA})&=(\sigma\circ\rho(p_{AC}))^*=(\sigma(p_{AB})\circ\rho(p_{BC}))^*
=(\rho(p_{BC}))^*\circ(\sigma(p_{AB}))^*
\\
&=\rho^*(p_{CB})\circ\sigma^*(p_{BA})=\rho^*\circ\sigma^*(p_{CA}), 
\end{align*}
\item 
otherwise, since $p_{AC}\in\Xs_{AC}-\Xs_{AB}\circ\Xs_{BC}\iff p_{CA}\in\Xs_{CA}-\Xs_{CB}\circ\Xs_{BA}$, we have
\begin{equation*}
(\sigma\circ\rho)^*(p_{CA})=\zeta^\Es_{CA}(p_{CA})=\rho^*\circ\sigma^*(p_{CA}).
\end{equation*} 
\end{itemize}
\end{itemize}

\medskip 

The norm on $\Gamma(\Es)$ is well-defined by equation~\eqref{eq: norm-G} because $\sigma\in\Gamma(\Es)_{AB}$ is continuous (on the locally compact Hausdorff space $\Xs_{AB}$) vanishing at infinity and hence the supremum is actually the maximum of the continuous map $p_{AB}\mapsto \|\sigma(p_{AB})\|$ vanishing at infinity. 

\medskip 

For $A,B,C\in\Os$ and all $(\sigma,\rho)\in\Gamma(\Es)_{AB}\times\Gamma(\Es)_{BC}$, since the support of $\sigma\circ\rho$ is inside $\Xs_{AB}\circ\Xs_{BC}\subset\Xs_{AC}$, we get:
\begin{align*}
\|\sigma\circ\rho\|
&=\sup_{p_{AC}\in\Xs_{AC}}\|(\sigma\circ\rho)(p_{AC})\|
=\sup_{p_{AC}\in\Xs_{AB}\circ\Xs_{BC}}\|(\sigma\circ\rho)(p_{AC})\|
=\sup_{(p_{AB},p_{BC})\in\Xs_{AB}\times_{X_B}\Xs_{BC}}\|\sigma(p_{AB})\circ\rho(p_{BC})\|
\\
&\leq \sup_{(p_{AB},p_{BC})\in\Xs_{AB}\times_{X_B}\Xs_{BC}}\|\sigma(p_{AB})\|\cdot\|\rho(p_{BC})\|
\leq\sup_{p_{AB}\in\Xs_{AB}}\|\sigma(p_{AB})\|\cdot\sup_{p_{BC}\in\Xs_{BC}}\|\rho(p_{BC})\|
=\|\sigma\|\circ\|\rho\|. 
\end{align*}

\medskip 

For all $A,B\in\Os$ and all $\sigma\in\Gamma(\Es)_{BA}$, using the fact that $\sigma^*\circ\sigma$ has support inside $\Xs_{AB}\circ\Xs_{BA}\subset\Xs_{A}$, we have: 
\begin{align*}
\|\sigma^*\circ\sigma\|
&=\sup_{p_{AA}\in\Xs_{A}} \|(\sigma^*\circ\sigma)(p_{AA})\| 
=\sup_{p_{AA}\in\Xs_{AB}\circ\Xs_{BA}} \|(\sigma^*\circ\sigma)(p_{AA})\| 
=\sup_{p_{BA}\in \Xs_{BA}} \|\sigma^*(p_{AB})\circ\sigma(p_{BA})\|
\\
&=\sup_{p_{BA}\in\Xs_{BA}} \|\sigma(p_{BA})^*\circ\sigma(p_{BA})\|
=\sup_{p_{BA}\in\Xs_{BA}}\|\sigma(p_{BA})\|^2
=\|\sigma\|^2.
\end{align*}

\medskip 

As a consequence of the previous steps, for all $A\in\Os$, $\Gamma(\Es)_{AA}$ is a unital C*-algebra, consisting of all the continuous sections of the trivial complex line bundle $(\Es_A,\pi,\Xs_A)$ over the compact Hausdorff topological space $\Xs_A$, and hence it is isomorphic to the commutative unital C*-algebra $C(\Xs_A)$. 
It follows that an element $\sigma\in\Gamma(\Es)_{AA}$ is positive if and only if, for all $p_{AA}\in\Xs_A$, the element $\sigma(p_{AA})$ is positive in the fiber $\Es_{p_{AA}}$ of the Fell bundle. 
Since $(\sigma^*\circ\sigma)(p_{BB})=\sigma^*(p_{BA})\circ\sigma(p_{AB})=(\sigma(p_{AB}))^*\circ(\sigma(p_{AB}))$ is a positive element the fiber $\Es_{p_{BB}}$ of the Fell bundle, for every $A,B\in\Os$, $\sigma\in\Gamma(\Es)_{AB}$ and $p_{AB}\in\Xs_{AB}$, we obtain the positivity condition of  $\sigma^*\circ\sigma$ and hence $\Gamma(\Es)$ is a commutative C*-category. 

\medskip 

Since $\Gamma(\Es)^0:=\Os$ is a set, the commutative C*-category $\Gamma(\Es)$ is small. 

\medskip 

We notice that $\Gamma_{(f,F)}$ is well-defined, since for any section $\sigma\in\Gamma(\Es^2)$ continuous vanishing at infinity the $f$-pull-back of $\sigma$ is defined as a continuous section $f^\bullet(\sigma)\in\Gamma(f^\bullet(\Es^2))$ vanishing at infinity and, since also  $F:f^\bullet(\Es^2)\to\Es$ vanishes at infinity, we obtain by composition a continuous section $F\circ f^\bullet(\sigma)$ in $\Gamma(\Es^1)$, vanishing at infinity as required. 

\medskip 

If $\omega:\Gamma(\Es^1)\to\CC$ is a $*$-functor such that $\omega_{BA}\neq 0$, also $\omega_{AA}\neq0$ because, whenever $\omega_{BA}(\sigma)\neq0$,  $\omega_{AA}(\sigma^*\circ\sigma)=\omega_{AB}(\sigma^*)\circ\omega_{BA}(\sigma)>0$. By Gel'fand duality for the commutative unital C*-algebra $\Gamma(\Es^1)_{AA}$ we have the existence of a point $p_A\in\Xs_A$ such that $\omega_{AA}=\xi_{p_A}\circ\ev_{p_A}$, where $\xi^1_{p_A}:\Es_{p_A}\to\CC$ is the canonical Gal'fand-Mazur isomorphism between 1-dimensional C*-algebras; using the fact that $F_{p_A}:f_\Xs^\bullet(\Es^2)_{p_A}\to \Es^1_{p_A}$ is an isomorphism of 1-dimensional C*-algebras, 
it follows that $(\Gamma_{(f,F)}^\bullet(\omega))_{f^1_\Rs(AA)}\simeq\ev_{f_\Xs(p_A)}$ modulo the unique Gel'fand-Mazur isomorphism $\xi^2_{f_\Rs(p_A)}:\Es^2_{f_\Rs(p_A)}\to\CC$ and hence, by saturation of the spaceoid $\Es^2$, we obtain $(\Gamma_{(f,F)}^\bullet(\omega))_{f^1_\Rs(AA)}\neq 0$. Since $f_\Xs(p_{AB})\circ f_\Xs(p_{BA})=f_\Xs(p_A)$ and, for all $\rho\in\Gamma(\Es^2)_{f^1_\Rs(BA)}$, we have $(\Gamma_{(f,F)}^\bullet(\omega))_{f^1_\Rs(AA)}(\rho^*\circ\rho) =(\Gamma_{(f,F)}^\bullet(\omega))_{f^1_\Rs(AB)}(\rho^*)\circ(\Gamma_{(f,F)}^\bullet(\omega))_{f^1_\Rs(BA)}(\rho)$, we see that $\Gamma^\bullet_{(f,F)}(\omega)\neq 0$ on the Hom-set $\Gamma(\Es^2)_{f_\Rs(BA)}$ and this completes the proof of the non-degeneracy of the covariant $*$-functor $\Gamma_{(f,F)}$.  

\medskip 

Given $(\Es^1,\pi^1,\Xs^1_{\theta^1,\Os^1})\xrightarrow{(f,F)}(\Es^2,\pi^2,\Xs^2_{\theta^2,\Os^2})$, 
for all sections $\sigma,\rho\in\Gamma(\Es^2)$ and for all $A\in\Os^1$, we have: 
\begin{align*}
&
\Gamma_{(f,F)}(\sigma\circ\rho)
= 
F(f^\bullet(\sigma\circ\rho))
=F(f^\bullet(\sigma)\circ f^\bullet(\rho))
=F(f^\bullet(\sigma))\circ F(f^\bullet(\rho))
= \Gamma_{(f,F)}(\sigma)\circ\Gamma_{(f,F)}(\rho), 
\\ 
& 
\Gamma_{(f,F)}(\iota^{\Es^2}_{f^0_\Rs(A)})
=F(f^\bullet(\iota^{\Es^2}_{f^0_\Rs(A)}))
=F(\iota^{f^\bullet(\Es^2)}_{A})
=\iota^{\Es^1}_A, 
\\
&
\Gamma_{(f,F)}(\sigma^*)
=F(f^\bullet(\sigma^*))
=F(f^\bullet(\sigma)^*)
=F(f^\bullet(\sigma))^*
=(\Gamma_{(f,F)}(\sigma))^*,
\end{align*}
and hence $\Gamma(\Es^2)\xrightarrow{\Gamma_{(f,F)}}\Gamma(\Es^1)$ is a $*$-functor. 

\medskip 

Let $(\Es^1,\pi^1,\Xs^1_{\theta^1,\Os^1})\xrightarrow{(g,G)}(\Es^2,\pi^2,\Xs^2_{\theta^2,\Os^2}) \xrightarrow{(f,F)}(\Es^3,\pi^3,\Xs^3_{\theta^3,\Os^3})$ be morphisms in $\Tf$, we have for all $\sigma\in\Gamma(\Es^3)$: 
\begin{align*} 
G\circ g^\bullet(F)\circ \Theta^{\Es^3}_{f,g} ((f\circ g)^\bullet(\sigma))
=G\circ g^\bullet(F)(g^\bullet(f^\bullet(\sigma)))
=G(g^\bullet(F(f^\bullet(\sigma))))
=G(g^\bullet(F(f^\bullet(\sigma)))), 
\end{align*}
and hence $\Gamma_{(f,F)\circ(g,G)}=\Gamma_{(f\circ g, G\circ g^\bullet(F)\circ \Theta^{\Es^3}_{f,g} )} 
=\Gamma_{(g,G)}\circ\Gamma_{(f,F)}$. 

\medskip 

Similarly, given the identity morphism  $(\Es,\pi,\Xs_{\theta,\Os})\xrightarrow{(\iota_\Xs,\Theta_\Es)}(\Es,\pi,\Xs_{\theta,\Os})$, for all $\sigma\in\Gamma(\Es)$: 
\begin{align*}
\Gamma_{((\iota_\Xs,\Theta_\Es))}(\sigma)=\Theta_\Es(\iota_\Xs^\bullet(\sigma))=\sigma
=\iota_{\Gamma(\Es)}(\sigma).
\end{align*}
It follows that $\Gamma$ is a contravariant functor. 
\end{proof}

\section{The Spectrum functor $\Sigma$} \label{sec: 5}

Let us start with the long construction of the spectral spaceoid of a commutative non-full C*-category. 
\begin{definition}
Given a commutative non-full C*-category $\Cs\in\Af^0$, we define 
\begin{equation*}
[\Cs;\CC]:=\left\{\omega \ | \ \Cs\xrightarrow{\omega}\CC \ \text{is a $*$-functor that is $\CC$-linear on each Hom-set $\Cs_{AB}$, for all $A,B\in\Cs^0$} \right\}, 
\end{equation*}
where we consider $\CC$ as a commutative C*-category with only one-object $\bullet$. 

\medskip 

We equip the set $[\Cs;\CC]$ with the weak-$*$-topology induced by the set $\{\ev_x \ | \ x\in\Cs^1\}$ of all the evaluation maps $\ev_x:[\Cs;\CC]\to\CC$ defined, for all $x\in\Cs^1$, by $\ev_x:\omega\mapsto \omega(x)$. 
\end{definition}

\begin{lemma}\label{lem: cH}
The topological space $[\Cs;\CC]$ is compact and Hausdorff. 
\end{lemma}
\begin{proof}
Exactly as in the proof of~\cite[proposition~5.4]{BCL11}, since for all $x\in\Cs^1$ (using the fact that the restriction of $\omega$ to the unital C*-algebra $\Cs_{s(x)s(x)}$ is a state and hence $\|\omega\|=1$), 
$|\omega(x)|^2=\cj{\omega(x)}\cdot\omega(x)=\omega(x^*\circ x)\leq \|x\|^2$, we have that $[\Cs;\CC]$ with the weak-$*$-topology is a subspace of the space $\prod_{x\in\Cs}\cj{B(0_\CC,\|x\|)}$, product of closed bounded disks in $\CC$, that is compact for the Tychonoff product topology. 
Since $[\Cs;\CC]$ coincides with the subset of $\prod_{x\in\Cs}\cj{B(0_\CC,\|x\|)}$ obtained intersecting all the closed subsets $\bigcap_{x,y\in\Cs, \ \alpha\in\CC}(\ev_{\alpha\cdot x+y}-\alpha\cdot\ev_x-\ev_y)^{-1}(0_\CC)$,  $\bigcap_{(x,y)\in\Cs^1\times_{\Cs^0}\Cs^1}(\ev_{x\circ y}-\ev_x\cdot\ev_y)^{-1}(0_\CC)$, $\bigcap_{A\in\Cs^0}(\ev_{\iota_A})^{-1}(1_\CC)$ and $\bigcap_{x\in\Cs} (\ev_{x^*}-\ev_x)^{-1}(0_\CC)$, we have that $[\Cs;\CC]$ is compact. 

\medskip 

To show the Hausdorff property of $[\Cs;\CC]$ notice that if $\omega_1\neq\omega_2$, there exists at least one $x\in\Cs^1$ such that $\ev_x(\omega_1)=\omega_1(x)\neq\omega_2(x)=\ev_x(\omega_2)$ and hence the continuous map $\ev_x$ takes different values in $\omega_1$ and $\omega_2$. 
\end{proof}

\begin{definition}
For every C*-category $\Cs\in\Af^0$ we define $\Rs_{\Cs^0}:=\Cs^0\times\Cs^0$ the pair groupoid with objects $\Cs^0$ equipped with the discrete topology; and we denote its morphisms by $AB:=(A,B)\in\Hom_{\Rs^\Cs}(B;A)$. 

\medskip 

We also define $\Rs^\Cs:=\Cs/\Cs$ the quotient C*-category of $\Cs$ by the closed categorical $*$-ideal $\Cs$ (see definition~\ref{def: C*-id}). 
\end{definition}
Notice that $\Rs^\Cs$ is a discrete groupoid isomorphic to the discrete pair groupoid $\Rs_{\Cs^0}:=\Cs^0\times\Cs^0$ of objects of $\Cs$. We also redundantly use the notation $\Os^\Cs:=\Cs^0$ and hence $\Rs^\Cs\simeq\Rs_{\Os^\Cs}$. 

\begin{remark}
Considering the pull-back $!^\bullet(\CC)$ of the Fell line-bundle $\CC\to \bullet$ via the terminal map $\Rs^\Cs\xrightarrow{!}\{\bullet\}$, we have a bijective correspondence $[\Cs;\CC]\ni \omega\mapsto \hat{\omega}\in\Hom_\Af(\Cs;!^\bullet(\CC))$ between $*$-functors $\omega\in[\Cs;\CC]$ and object-bijective $*$-functors $\Cs\xrightarrow{\hat{\omega}}!^\bullet(\CC)$ between commutative C*-categories. 

\medskip 

This ``desingularization'' of the target C*-category of the $*$-functors $\omega\in[\Cs;\CC]$ allows for the introduction of ``gauge fluctuations'' induced by the group $\Aut_\Af(!^\bullet(\CC))$ of automorphisms of the C*-category $!^\bullet(\CC)\in\Af^0$. 
\xqed{\lrcorner}
\end{remark}

\begin{definition} \label{def: -B}
We define $\Bs^\Cs$ as the set of orbits in $[\Cs;\CC]$ of the action $\Aut_\Af(!^\bullet(\CC))\times[\Cs;\CC]\to[\Cs;\CC]$ given by $\omega\mapsto\alpha\circ\omega$; 
considering the equivalence relation $E^\Cs$ and quotient map $q^\Cs_E$ 
\begin{equation}\label{eq: E}
E^\Cs:=\{(\omega_1,\omega_2)\in[\Cs;\CC]\times[\Cs;\CC] \ \ | \  \ \exists \alpha\in\Aut_\Af(!^\bullet(\CC)) \st \omega_2=\alpha\cdot\omega_1\}, 
\quad 
[\Cs;\CC]\xrightarrow{q^\Cs_E}\Bs^\Cs,
\end{equation}
we equip $\Bs^\Cs$ with its quotient topology.  
We will denote by $[\omega]\in\Bs^\Cs$ the equivalence class of $\omega\in[\Cs;\CC]$ and we will write $\omega_1\sim\omega_2$ whenever $[\omega_1]=[\omega_2]$.  
\end{definition}

\begin{lemma}\label{lem: aut}
The group $\Aut_\Af(!^\bullet(\CC))$ is isomorphic to the group $[\Rs^\Cs;!^\bullet(\TT)]$ of  
$*$-functorial sections of the torsor $!^\bullet(\TT)\xrightarrow{}\Rs^\Cs$. 
We have $\omega_1\sim\omega_2$ if and only if $\hat{\omega}_2=\alpha\cdot\hat{\omega}_1$, where $\alpha\in [\Rs^\Cs;!^\bullet(\TT)]$.  
The group $\Aut_\Af(!^\bullet(\CC))$ acts by homeomorphisms onto $[\Cs;\CC]$. With the topology of pointwise convergence $\Aut_\Af(!^\bullet(\TT))$ is a compact topological group. 
\end{lemma}
\begin{proof}
Every $\alpha\in\Aut_\Af(!^\bullet(\CC))$ acts (modulo a permutation) as the identity on objects and is a fibrewise \hbox{$\CC$-linear} map; since the fibers of the bundle $!^\bullet(\CC)$ are one-dimensional, for all $AB\in\Rs^\Cs$, $\alpha$ acts as multiplication by a complex number $\alpha_{AB}$ on $\CC_{AB}$. 
The $*$-homomorphism properties of $\alpha$ immediately imply: $\alpha_{AB}\circ\alpha_{BC}=\alpha_{AC}$, $\alpha_{BA}=\alpha_{AB}^*$ and $\alpha_{AB}\circ\alpha_{BA}=\alpha_{AA}=\id_{\CC_{AA}}$, for all $A,B,C\in\Cs^0$. This means that $\alpha_{AB}\in\TT_{AB}$, for all $AB\in\Rs^\Cs$ and hence, modulo a permutation on objects, $\alpha$ is in bijective correspondence with a unique object preserving functorial map $\Rs^\Cs\to !^\bullet(\TT)$ (that with abuse of notation we will denote with the same letter) $AB\mapsto \alpha_{AB}$. 

\medskip 

Whenever $\omega_1\sim\omega_2$ we have, using the ubused notation above,  $\hat{\omega}_2(x)=\alpha(\hat{\omega}_1(x))=\alpha\cdot \omega_1(x)$, for all $x\in\Cs^1$. 

\medskip 

Since the action of $\alpha$ on $[\Cs;\CC]$ is bijective, with inverse the action of $\alpha^{-1}$, and since $[\Cs;\CC]$ is compact Hausdorff, we only need to show that the map $\mu_\alpha:\omega\mapsto \alpha\cdot \omega$ is continuous for the weak-$*$-topology of $[\Cs;\CC]$ and this follows from the fact that $\mu_\alpha^\bullet(\ev_x):=\ev_x\circ\mu_\alpha=\ev_{\alpha\cdot x}$, for all $x\in\Cs^1$. 

\medskip 

Finally $\Aut_\Af(!^\bullet(\CC))$ is a compact topological group: since it is isomorphic to $[\Rs^\Cs;!^\bullet(\TT)]$, that (with the pointwise convergence topology) is a subset of the compact Tychonoff product $\prod_{AB\in\Rs^\Cs}\TT$, and the $*$-functorial axioms, as in lemma~\ref{lem: cH}, imply that $[\Rs^\Cs;!^\bullet(\TT)]$ corresponds to a closed (and hence compact) subset (actually subgroup) of the compact topological group $\prod_{AB\in\Rs^\Cs}\TT$.
\end{proof}

\begin{lemma} \label{lem: E-closed}
The space $\Bs^\Cs$ is compact Hausdorff. 
The quotient map $[\Cs;\CC]\xrightarrow{q^\Cs_E}\Bs^\Cs$ is closed and open. 
\end{lemma}
\begin{proof}
Since by lemma~\ref{lem: cH} $[\Cs;\CC]$ is compact, its quotient space $\Bs^\Cs$ is compact as well. 

\medskip 
 
In order to prove that the quotient space $\Bs^\Cs$ is Hausdorff, we make use of~\cite[exercise 17N]{Wi04} \footnote{
If $X$ is compact Hausdorff space and $f:X\to Y$ is a surjective quotient map for the equivalence relation $E\subset X\times X$, then the following three properties are  equivalent: $Y$ is Hausdorff $\iff$ $f$ is closed $\iff$ the equivalence relation $E$ is closed in $X\times X$. 
} just proving that the equivalence relation $E^\Cs\subset[\Cs;\CC]\times[\Cs;\CC]$ defined in~\eqref{eq: E} is closed:   
since, for all triples $(\omega_1,\alpha,\omega_2)\in[\Cs;\CC]\times\Aut_\Af(!^\bullet(\CC))\times[\Cs;\CC]$, we have $|\omega_2-\alpha\cdot\omega_1(x)|\leq |\omega_2(x)|+|\omega_1(x)|\leq2\|x\|$, for all $x\in\color{brown}\Cs^1$, 
the map $[\Cs,\CC]\times\Aut_\Af(!^\bullet(\CC))\times[\Cs;\CC]\xrightarrow{F} \prod_{x\in\Cs}\cj{B(0_\CC,2\|x\|)}$ given by $(\omega_1,\alpha,\omega_2)\mapsto \omega_2-\alpha\cdot \omega_1$ is well-defined and continuous for the respective product topologies and hence 
$\{(\omega_1,\alpha,\omega_2) \ | \ \omega_2=\alpha\cdot\omega_1\}=F^{-1}(0_{\CC^\Cs})$ is closed; 
since $\Af(!^\bullet(\CC))$ and $[\Cs;\CC]$ are compact Hausdorff, the continuous projection $(\omega_1,\alpha,\omega_2)\xrightarrow{P} (\omega_1,\omega_2)$ is necessarily closed and so is the set $P(F^{-1}(0_{\CC^\Cs}))=E^\Cs$. 

\medskip

The quotient map $[\Cs;\CC]\xrightarrow{q^\Cs_E}\Bs^\Cs$ is also open (apart from closed): since, by lemma~\ref{lem: aut}, the group $\Aut(!^\bullet(\CC))$ is acting by homeomorphisms onto $[\Cs;\CC]$, if $U\subset [\Cs;\CC]$ is open, also $\bigcup_{\alpha\in\Aut(!^\bullet(\CC))} \alpha\cdot U=(q_E^\Cs)^{-1}(q^\Cs_E(U))$ is open and so $q_E(U)\subset\Bs^\Cs$ is open in the quotient. This property, by \cite[theorem 13.12]{Wi04}, provides another proof that $\Bs^\Cs$ is Hausdorff if and only if $E^\Cs\subset [\Cs;\CC]\times[\Cs;\Cs]$ is closed. 
\end{proof}

We need to recall some material on $*$-ideals in a C*-category (see~\cite{Mi02}). 
\begin{definition}\label{def: C*-id}
A \emph{categorical $*$-ideal} in a C*-category (or more generally in an involutive algebroid) $\Cs$ is a family $\Js\subset \Cs^1$ that satisfies the following ``algebraic closure'' conditions, that are a horizontal categorification of the the usual properties of $*$-ideals in a C*-algebra (or in an involutive ring): 
\begin{gather*}
\forall A\in\Cs^0 \st o_A\in \Js, 
\quad \quad 
\forall x,y\in\Js \st x+y\in\Cs^1 \imp x+y \in \Js, 
\quad \quad 
\forall x\in\Js \st -x\in\Js, 
\\
\forall x\in\Js \st x^*\in\Js, 
\quad \quad 
\forall x,z\in\Cs^1, \ \forall y\in\Js \st (x\circ y\in\Cs^1 \imp x\circ y\in\Js)\wedge
(y\circ z\in\Cs^1 \imp y\circ z\in\Js), 
\end{gather*} 
the first line simply says that, for all $A,B\in\Cs^0$, $\Js_{AB}$ is an Abelian subgroup of $\Cs_{AB}$; the remaining conditions impose the closure under involution and the usual ``stability'' under arbitrary compositions for ideals.

\medskip 

We will use the notation $\Cs/\Is$ to denote the quotient category of the category $\Cs$ by its categorical ideal $\Is$. 
\end{definition} 

The following special family of C*-categories (Fell bundles over pair groupoids of rank less or equal to $1$) plays in our spectral theory of non-full C*-categories the same role (target of characters) of complex numbers for the usual spectral theory of commutative unital C*-algebras. 
\begin{definition}
A \emph{C*-category of rank less or equal to $1$} is a C*-category $\Cs$ such that $\max_{AB\in\Rs^\Cs}\dim(\Cs_{AB})=1$. 
\end{definition}
Classifications, modulo isomorphism, of full C*-categories of rank less or equal to one are known: every equivalence class corresponds to a $\CC$-line bundle associated to a $\TT$-torsor over $\Rs^\Cs$ and hence it is uniquely determined by a choice of a functorial section in $[\Rs^\Cs;\TT]$. Here we provide a description in the non-full case. 
\begin{lemma}
A C*-category of rank less or equal to 1 consists of full C*-subcategories in discrete position. 
\end{lemma}
\begin{proof}
Define on the class $\Cs^0$ of objects the relation $A\sim B$ if and only if $\dim(\Cs_{AB})=1$ or equivalently $A$ and $B$ are not in discrete position. Since $\Cs_{AB}\circ\Cs_{BC}\subset\Cs_{AC}$, we see that the previous is an equivalence relation on the class $\Cs^0$. All the full-subcategories corresponding to an equivalence class of objects are full C*-categories of rank 1 and any two of them are in discrete position (if they do not coincide). 
\end{proof}

\begin{lemma}\label{lem: ke}
For all $\omega\in[\Cs;\CC]$, the kernel $\ke(\omega):=\{x\in\Cs^1 \ | \ \omega(x)=0_\CC\}$ is a closed categorical $*$-ideal in $\Cs$. 
Furthermore, using the equivalence relation introduced in definition~\ref{def: -B},  
$\omega_1\sim\omega_2\imp \ke(\omega_1)=\ke(\omega_2)$ and $\Cs/\ke(\omega)$ is a C*-category of rank less or equal to one. 
\end{lemma}
\begin{proof}
Since $\omega$ is $\CC$-linear when restricted to any Hom-set, $\ke(\omega)$ is closed under addition, opposite and zeros. 

Whenever $x\circ y$ and $y\circ z$ exist and $y\in\ke(\omega)$, by the functorial properties of $\omega$, we have $x\circ y,y\circ z\in\ke(\omega)$, since 
$\omega(x\circ y)=\omega(x)\cdot\omega(y)=0_\CC=\omega(y)\cdot\omega(z)=\omega(y\circ z)$. 
Similarly the fact that $\omega$ is involutive gives $\omega(x^*)=\cj{\omega(x)}$, hence $x^*\in\ke(\omega)$ whenever $x\in\ke(\omega)$. 
Since $\ke(\omega)=\omega^{-1}(0_\CC)$, the topological closure of $\ke(\omega)$ is a consequence of the continuity of $\omega$ and the fact that $\{0_\CC\}$ is a closed set in $\CC$. 

\medskip 

Since, by lemma~\ref{lem: aut}, $\omega_1\sim\omega_2$ if and only if there exists $\alpha\in \Aut(!^\bullet(\CC))$ such that $\omega_2=\alpha\cdot \omega_1$, since $\alpha$ is invertible, we have $x\in\ke(\omega_1)\iff x\in\ke(\omega_2)$. It follows that the closed $*$-ideal $\ke(\omega)$ is independent from the choice of the representative in the equivalence class $[\omega]\in\Bs^\Cs:=[\Cs;\CC]/\sim$.  

\medskip 

It is a general ``algebraic fact'' (see~\cite{Mi02}) that the quotient $\Cs/\ke(\omega)$ of a C*-category $\Cs$ by a closed $*$-ideal $\ke(\omega)$ is again a C*-category, with norm defined as $\|x+\ke(\omega)\|_{\Cs/\ke(\omega)}:=\inf_{k\in\ke(\omega)}\|x+k\|_\Cs$, and the quotient map $\Cs\xrightarrow{\pi^\Cs_{[\omega]}}\Cs/\ke(\omega)$ is a $*$-functor. Since $\Cs$ and $\Cs/\ke(\omega)$ share the same objects, they are both Fell bundles over the discrete pair groupoid $\Rs^\Cs:=\Cs^0\times\Cs^0$ with fibers $\Cs_{AB}$ and $\Cs_{AB}/\ke(\omega)_{AB}$. 

\medskip 

Since, for all $AB\in\Rs^\Cs$, there is an injective $\CC$-linear map $\Cs_{AB}/\ke(\omega)_{AB}\xrightarrow{}\CC$ induced via first-isomorphism theorem and $\dim(\CC)=1$, the fibers $\Cs_{AB}/\ke(\omega)_{AB}$ can only have dimension $0$ of $1$ and so $\rk(\pi^\Cs_{[\omega]})\leq 1$.
\end{proof}

\begin{remark}
We could immediately develop a spectral theory with spectral spaceoids defined as bundles of C*-categories (of rank less or equal to one) over $\Bs^\Cs$ (see definition~\ref{def: -B}) or equivalently as Fell bundles, of rank less or equal to one, on the product topological groupoids $\Delta_{\Bs^\Cs}\times\Rs^\Cs$; the unpleasant feature of such choice is that the topological space $\Bs^\Cs$ is outrageously redundant: for example, in the case of a discrete commutative C*-category $\Cs$, it is homeomorphic to the Tychonoff product of the Gel'fand spectra of all the C*-algebras $\Cs_{AA}$, with $A\in\Cs^0$. 
In order to make more explicit contact with the Fell line-bundle spaceoids used in~\cite{BCL11}, we decided to pursue here a more ``economic'' direction.
\xqed{\lrcorner} 
\end{remark}

\medskip 

We are going to further analyze the base spectrum of our spectral spaceoids associating to $\Bs^\Cs$ a certain topological groupoid fibered over a discrete pair groupoid, where fibers are ``graphs of homeomorphisms'' between open sets in Gel'fand spectra of the C*-algebras $\Cs_{AA}$, for $A\in\Cs^0$. 

\begin{lemma}
For $AB\in\Rs^\Cs$, consider the restriction map $[\Cs;\CC]\xrightarrow{|_{AB}}[\Cs;\CC]_{AB}$ defined by $\omega\mapsto \omega_{AB}:=\omega|_{\Cs_{AB}}$. 
The space $[\Cs;\CC]_{AB}$ with the weak-$*$-topology induced by the maps $\{\ev_x \ | \ x\in\Cs_{AB}\}$ 
is canonically homeomorphic to the quotient space of $[\Cs;\CC]$ induced by the restriction map $|_{AB}$ and is compact Hausdorff.
\end{lemma}
\begin{proof} 
The set $[\Cs;\CC]_{AB}$ is by definition consisting of restrictions of $*$-functors in $[\Cs;\CC]$ and hence the restriction map is surjective. The weak-$*$-topology induced by evaluations is Hausdorff and makes the restriction map continuous (since it preserves evaluations: $\ev_x^{\Cs_{AB}}\circ \ |_{AB}=\ev^\Cs_x$, for all $x\in\Cs_{AB}\subset \Cs$). Compactness follows since $|_{AB}$ is continuous surjective and $[\Cs;\CC]$ is already compact with its weak-$*$-topology.

\medskip 

By surjectivity of the restriction map $|_{AB}$, the set $[\Cs;\CC]_{AB}$ is already in bijective correspondence with the quotient of $[\Cs;\CC]$ induced by the restriction. Since $|_{AB}$ is continuous, the weak-$*$-topology on $[\Cs;\CC]_{AB}$ is weaker than the $|_{AB}$-quotient topology. Since both spaces are compact, they will be homeomorphic via $|_{AB}$ as soon as we show that the quotient topology is Hausdorff. This follows, again by~\cite[exercise 17N]{Wi04}, showing the closure of the equivalence relation $\Big\{(\omega,\rho)\in[\Cs;\CC]\times[\Cs;\CC]\ | \ \omega_{AB}=\rho_{AB}\Big\}=\bigcap_{x\in\Cs_{AB}}(\ev^\Cs_x\circ\varpi_1 -\ev^\Cs_x\circ\varpi_2)^{-1}(0_\CC)$, where $\varpi_j$, for $j=1,2$, denote the projections of the product $[\Cs;\CC]\times[\Cs;\CC]$.  
\end{proof}

\begin{remark}\label{rem: X-topology}
Notice that $\omega_{AB}$, for $\omega\in[\Cs;\CC]$ and $AB\in\Rs^\Cs$, already denotes the equivalence class in the product $[\Cs;\CC]\times\Rs^\Cs$ induced by the restriction maps $|_{AB}$: 
$\omega_{AB}=\rho_{AB}$ if and only if $\omega,\rho\in[\Cs;\CC]$ and $\omega|_{AB}=\rho|_{AB}$.  
We have the following commutative diagram of quotient (open and closed) maps  
\begin{equation*}
\xymatrix{
[\Cs;\CC] \times\Rs^\Cs \ar[rrrr]^{\text{Hom-set restriction}} \ar[d]|{\text{global torsor orbits}} 
& & & & \bigcup_{AB\in\Rs^\Cs} [\Cs;\CC]_{AB} \ar[d]|{\text{local torsor orbits}} 
& 
(\omega,AB) \ar@{|->}[rrr] \ar@{|->}[d] & & &  \omega_{AB} \ar@{|->}[d] 
\\
\Bs^\Cs\times\Rs^\Cs \ar[rrrr]^{\text{quotient Hom-set restriction}} & & & &  \bigcup_{AB\in\Rs^\Cs} \Bs^\Cs_{AB}
& 
([\omega],AB) \ar@{|->}[rrr] & & &  [\omega]_{AB}
}
\end{equation*}

\medskip 

We have $[\omega_1]_{AB}\sim[\omega_2]_{AB}$ if and only if $\omega_1|_{AB}=u\cdot\omega_2|_{AB}$, for a phase $u\in\TT$ hence $\ke(\omega_1|_{AB})=\ke(\omega_2|_{AB})$. 

\medskip 

Notice that, for $A\in\Os^\Cs$, the equivalence class $[\omega]_{AA}=\{\omega_{AA}\}$ is a singleton consisting of a pure-state on $\Cs_{AA}$. 

\medskip 

All the topological spaces $\Bs^\Cs_{AB}$, for all $AB\in\Rs^\Cs$ are compact Hausdorff in their quotient topology (to show the Hausdorff property of $\Bs^\Cs_{AB}$ just apply to the compact Hausdorff space $[\Cs;\CC]_{AB}$ the same proof of lemma~\ref{lem: E-closed} showing that the equivalence relation $E^\Cs_{AB}\subset[\Cs;\CC]_{AB}\times[\Cs;\CC]_{AB}$ is closed) and hence all the continuous restriction and quotient maps of the previous commutative diagram are closed (and also open, by surjectivity).  
\xqed{\lrcorner}
\end{remark}
 
The following family of topological spaces generalizes the unique compact Hausdorff space in the case of spaceoids for full commutative C*-categories in \cite[definition 3.2]{BCL11} (where $\Xs^\Cs_{AB}$ are all homeomorphic, for all $AB\in\Rs^\Cs$). 

\begin{definition}
Given a small non-full C*-category $\Cs$, for all $AB\in\Rs^\Cs$ we define: 
\begin{equation*}
\Xs^\Cs_{AB}:=\Big\{[\omega]_{AB} \in\Bs^\Cs_{AB} \ | \ \dim_\CC\Big(\Cs_{AB}/\ke(\omega|_{AB})\Big)=1 \Big\}. 
\end{equation*}
Equivalently $\Xs^\Cs_{AB}$ consists of those $[\omega]_{AB}\in\Bs^\Cs_{AB}$ such that $\Cs/\ke(\omega)$ is a C*-category that, restricted to the discrete pair groupoid with objects $A,B$, is full. 
\end{definition}

\begin{proposition}\label{prop: X-delta}  
For all $A\in\Cs^0$, we have $\Xs^\Cs_A=[\Cs;\CC]_{AA}$ and the space $\Xs^\Cs_A$ is compact Hausdorff homeomorphic to the Gel'fand spectrum $\Sp(\Cs_{AA})=[\Cs_{AA};\CC]$ of the commutative unital C*-algebra $\Cs_{AA}$. 
\end{proposition}
\begin{proof}
By unitality, $\ke(\omega_{AA})\neq\{0\}$ for all $\omega\in[\Cs;\CC]$ hence $\Bs^\Cs_{A}=\Xs^\Cs_{A}$. 

\medskip 

Since $\alpha_{AA}=1$, for all $A$, we have that $[\Cs;\CC]_{AA}\simeq\Bs^\Cs_A$ (the equivalence classes are singletons). 

\medskip 

The restriction map $|_{AA}:\omega\mapsto \omega_{AA}$ is a continuous map $|_{AA}:[\Cs;\CC]\to\Sp(\Cs_{AA})$ between compact Hausdorff spaces in their weak-$*$-topology. 
It will be a homeomorphisms if it is bijective. 

\medskip 

Now $|_{AA}$ is injective, since $\omega_{AA}\neq \rho_{AA}\imp \exists x\in\Cs_{AA} \st \omega(x)\neq\rho(x)$ and so $\omega\neq\rho$. 

\medskip 

For the surjectivity, given $\omega_o\in[\Cs_{AA};\CC]$, consider $\ker(\omega_o)\subset\Cs_{AA}$ and let $\Is_{\omega_o}$ be the smallest closed categorical $*$-ideal containing $\ker(\omega_o)$. 
We explicitly have $(\Is_{\omega_o})_{BC}=\Cs_{BA}\circ\Is_{\omega_o}\circ\Cs_{AC}$, for all $B,C\in\Cs^0$, and hence $(\Is_{\omega_o})_{AA}=\ke(\omega_o)$.  
The quotient $\frac{\Cs}{\Is_{\omega_o}}$ is a C*-category. We choose an arbitrary element $\kappa\in[\frac{\Cs}{\Is_{\omega_o}};\CC]$ and notice that, by Gel'fand-Mazur applied to $(\Cs/\Is_{\omega_o})_{AA}\simeq\CC$, we have $\omega_o=\kappa\circ q\circ |_{AA}$, where $\Cs\xrightarrow{q}\frac{\Cs}{\Is_{\omega_o}}$ is the canonical quotient $*$-functor onto the quotient C*-category. Since $\kappa\circ q\in[\Cs;\CC]$, we obtain the surjectivity. 
\end{proof}

\begin{proposition}\label{prop: X-topology}
For all $AB\in\Rs^\Cs$, the space $\Xs^\Cs_{AB}\subset\Bs^\Cs_{AB}$ is an open locally compact Hausdorff completely regular space with Alexandroff compactification $\Bs^\Cs_{AB}$. 
For all $AB\in\Rs^\Cs$ the source/target maps given by 
\begin{equation}\label{eq: s-t-m}
\Xs^\Cs_A\xleftarrow{t^\Cs_{AB}}\Xs^\Cs_{AB}\xrightarrow{s^\Cs_{AB}}\Xs^\Cs_B
\quad \quad   
\xymatrix{[\omega]_{AA} & \ar@{|->}[l]_{t^\Cs_{AB}} [\omega]_{AB} \ar@{|->}[r]^{s^\Cs_{AB}} &[\omega]_{BB}} 
\end{equation}
are homeomorphisms onto open subsets of $\Xs^\Cs_A$ and $\Xs^\Cs_B$. 
\end{proposition}
\begin{proof}
For all $AB\in\Rs^\Cs$, since for all $x\in\Cs_{AB}$ the map $\ev_x:[\Cs;\CC]\to\Cs$ is continuous in the weak-$*$-topology, the set $\bigcap_{x\in\Cs_{AB}}\ev_x^{-1}(0_\CC)$ is closed in $[\Cs;\CC]$ and, by remark~\ref{rem: X-topology} its image inside $\Bs^\Cs_{AB}$ is closed, consisting of only one point (corresponding to the unique zero $\CC$-linear map on $\Cs_{AB}$). By definition, $\Xs^\Cs_{AB}$ coincides with the complement of the previous closed image and hence it is open in $\Bs^\Cs_{AB}$. 
Being an open set in $\Bs^\Cs_{AB}$ that is compact Hausdorff and hence T4, the topological space $\Xs^\Cs$ is always locally compact Hausdorff and completely regular. 
Whenever $\varnothing\neq\bigcap_{x\in\Cs_{AB}}\ev_x^{-1}(0_\CC)$, the compact Hausdorff space $\Bs^\Cs_{AB}=\Xs^\Cs_{AB}\cup\{\bullet_{AB}\}$ is the Alexandroff compactification of $\Xs^\Cs_{AB}$ with ``point at infinity'' $\bullet_{AB}\in\Bs^\Cs_{AB}-\Xs^\Cs_{AB}$ that is the equivalence class $\bullet_{AB}:=[\omega]_{AB}$ of all the $\omega\in[\Cs;\CC]$ that restrict to the zero linear functional on $\Cs_{AB}$. 

\medskip 

We first justify that source and target maps in~\eqref{eq: s-t-m} are well-defined. 
Whenever $[\omega]_{AB}\in\Xs^\Cs_{AB}$, the C*-category $\Cs/\ke(\omega)$ restricted to the two objects $A,B$ is full, hence $\dim_\CC(\Cs_{AA}/\ke(\omega_{AA}))=1=\dim_\CC(\Cs_{BB}/\ke(\omega_{BB}))$ and so we obtain $[\omega]_{AA}\in\Xs^\Cs_{AA}$ and $[\omega]_{BB}\in\Xs^\Cs_{BB}$. 
In order to show that sources/target maps are well-defined functions, given $[\omega]_{AB}=[\rho]_{AB}\in\Xs^\Cs_{AB}$, we must prove that $[\omega]_{AA}=[\rho]_{AA}$ and similarly $[\omega]_{BB}=[\rho]_{BB}$. 
From the $*$-functorial properties of $\omega,\rho$, we have $\omega(x\circ y^*)=\rho(x\circ y^*)$ and $\omega(x^*\circ y)=\rho(x^*\circ y)$, for all $x,y\in\Cs_{AB}$, and hence $\omega_{AA}$ concides with $\rho_{AA}$ on the closed $*$-ideal $\Is^\Cs_{AA}$ generated by all the ``inner products'' $\{x\circ y^* \ | \ x,y\in\Cs_{AB}\}$; similarly $\omega_{BB}$ coincides with $\rho_{BB}$ on the closed $*$-ideal $\Is^\Cs_{BB}$ generated by $\{x^*\circ y\ | \ x,y\in\Cs_{AB}\}$. 
Using the spectral theory for closed ideals of a unital commutative C*-algebra, we know that $\Is^\Cs_{AA}$, via Gel'fand transform $\Gg_{\Cs_{AA}}:\Cs_{AA}\to C(\Sp(\Cs_{AA}))$, is isomorphic to the closed ideal in $C(\Sp(\Cs_{AA}))$ consisting of continuous functions vanishing on a closed subset of $\Sp(\Cs_{AA})$. Since $\omega_{AA}$ is a unital $*$-homomorphism on $\Cs_{AA}$, it is a pure state and there exists a unique point $p_\omega\in\Sp(\Cs_{AA})$ such that $\omega_{AA}=\ev_{p_\omega}\circ\Gg_{\Cs_{AA}}$, 
(where $\Gg_{\Cs_{AA}}$ denotes the Gel'fand transform of the C*-algebra $\Cs_\AA$) 
similarly there exists a unique point $p_\rho\in\Sp(\Cs_{AA})$ such that $\rho_{AA}=\ev_{p_\rho}\circ\Gg_{\Cs_{AA}}$. Since we have $\omega_{AA}|_{\Is^\Cs_{AA}}=\rho_{BB}|_{\Is^\Cs_{AA}}$, and such restriction is non-trivial, we necessarily have $p_\omega=p_\rho$ and $\omega_{AA}=\rho_{AA}$; in perfectly similar way we get $\omega_{BB}=\rho_{BB}$. Recalling, from remark~\ref{rem: X-topology}, that $[\omega]_{CC}=\{\omega_{CC}\}$ and $[\rho]_{CC}=\{\rho_{CC}\}$, for all $C\in\Os^\Cs$, we completed the task. 

\medskip 

For the injectivity of $t^\Cs_{AB}$, suppose that $\omega,\rho\in[\Cs;\CC]$ satisfy $[\omega]_{AA}=[\rho]_{AA}$ and hence $\omega_{AA}=\rho_{AA}$, for all $x\in\Cs_{AB}$, we obtain $|\omega_{AB}(x)|^2=\omega_{AA}(x^*\circ x)=\rho_{AA}(x^*\circ x)=|\rho_{AB}(x)|^2$, that implies $|\omega_{AB}(x)|=|\rho_{AB}(x)|$, for all $x\in\Cs_{AB}$; this further implies that $\ke(\omega_{AB})=\ke(\rho_{AB})$ and hence, since the quotients by the kernels are 1-dimensional, there exists $u\in\CC$ such that $\rho_{AB}=u\cdot\omega_{AB}$; finally, from the equality $|\rho_{AB}|=|u|\cdot|\omega_{AB}|$, we obtain $u\in\TT$ and, using remark~\ref{rem: X-topology}, we get $[\omega]_{AB}=[\rho]_{AB}$. The injectivity of $s^\Cs_{AB}$ is obtained in a perfectly similar way. 

\medskip 

For continuity and openness of the target map $t^\Cs_{AB}$ (the case of the source map $s^\Cs_{AB}$ is treated in a perfectly similar way), consider the following commuting diagram:
\begin{equation*}
\xymatrix{
& & [\Cs;\CC]-\{\omega \ | \ \omega|_{AB}=0\} \ar[rrd]^{|_{AB}} \ar[lld]_{ \quad \quad |_{AA}}
\\
[\Cs;\CC]_{AA}=\Bs^\Cs_A=\Xs^\Cs_A & & \ar[ll]^{\quad \quad \quad t^\Cs_{AB}} \Xs^\Cs_{AB} 
& & \ar[ll]^{q^\Cs_{E_{AB}}} [\Cs;\CC]_{AB}-q^\Cs_{E_{AB}}(\bullet_{AB}). 
}
\end{equation*}
The restriction maps $[\Cs;\CC]\xrightarrow{|_{AB}}[\Cs;\CC]_{AB}$, $[\Cs;\CC]\xrightarrow{|_{AA}}[\Cs;\CC]_{AA}$ and the quotient map $[\Cs;\CC]_{AB}\xrightarrow{q^\Cs_{E_{AB}}}\Bs^\Cs_{AB}$
are onto continuous open closed between compact Hausdorff spaces. 
The target map $t^\Cs_{AB}$ is only defined on the open subspace $\Xs^\Cs_{AB}\subset\Bs^\Cs_{AB}$ and hence we need to restrict to open subspaces in the commuting diagram above. 
Using the quotient map $|_{AB}$, we see that $t^\Cs_{AB}\circ q^\Cs_{E_{AB}}$ is continuous because, 
$\ev_x\circ \left(t^\Cs_{AB}\circ q^\Cs_{E_{AB}}\right) |_{AB}=\ev_x$, for all $x\in\Cs_{AA}$, and hence using the quotient map $q^\Cs_{E_{AB}}$ we obtain the continuity of $t^\Cs_{AB}$.   
The target map $t^\Cs_{AB}$ is also open because, using the commuting diagram above (and the surjectivity continuity and openness of the quotient maps), for every open set $\U\subset\Xs^\Cs_{AB}$, we have $t^\Cs_{AB}(\U)=|_{AA}(|_{AB}^{-1}(q^\Cs_{E_{AB}})^{-1}(\U)))$. 
\end{proof}

\begin{proposition}\label{prop: X-groupoid}
The bundle $\Xs^\Cs\xrightarrow{\theta^\Cs}\Rs^\Cs$ with fibers $\Xs_{AB}$, for $AB\in\Rs^\Cs$, is a topological groupoid fibered over the discrete groupoid $\Rs^\Cs$ with well-defined operations given, for all $A,B,C\in\Cs^0$, $\omega\in [\Cs;\CC]$, by:
\begin{gather}\label{eq: operations}
[\omega]_{AB}\circ[\omega]_{BC}:=[\omega]_{AC},   
\quad 
([\omega]_{AB})^*:=[\omega]_{BA},   
\quad   
\iota^{\Xs^\Cs}([\omega]_{AA})=[\omega]_{AA}.
\end{gather}

Furthermore $\Xs^\Cs_{AB}\circ \Xs^\Cs_{BC}\subset\Xs^\Cs_{AC}$, $\Xs^\Cs_{BA}=(\Xs^\Cs_{AB})^*$ and $\Xs^\Cs_{AA}=\Delta_{\Xs^\Cs_A}$ is a trivial groupoid of identities. 
\end{proposition}
\begin{proof} 
Once the formulas in~\eqref{eq: operations} are fully justified, the groupoid properties of $\Xs^\Cs$ are trivial: 
associativity follows from $[\omega]_{AB}\circ([\omega]_{BC}\circ[\omega]_{CD})=([\omega]_{AB}\circ[\omega]_{BC})\circ[\omega]_{CD}$; 
unitality follows from $[\omega]_{AB}\circ[\omega]_{BB}=[\omega]_{AB}$ and $[\omega]_{AA}\circ[\omega]_{AB}=[\omega]_{AB}$; 
for all $[\omega]_{AB}$, the inverse is given by $([\omega]_{AB})^{-1}=[\omega]_{BA}=([\omega]_{AB})^*$.
Similarly $\theta^\Cs$ will be a morphism of groupoids: $AC=\theta^\Cs([\omega]_{AC})=\theta^\Cs([\omega]_{AB}\circ[\omega_{BC}])=\theta^\Cs([\omega]_{AB})\cdot\theta^\Cs([\omega]_{BC}) =AB\cdot BC$, $\theta^\Cs(([\omega]_{AB})^*)=\theta^\Cs([\omega]_{BA})=BA=(AB)^*$, $\theta^\Cs(\iota^{\Xs^\Xs}([\omega]_{AA}))=\theta^\Cs([\omega]_{AA})=AA=\iota^{\Rs^\Cs}(A)$. 

\medskip 

We need to clarify how the operations of identity, involution and composition in~\eqref{eq: operations} are well-defined: 
\begin{itemize}
\item 
whenever $A=B$, from the fact that $[\omega]_{AB}=\{\omega_{AB}\}$, for all $\omega\in[\Cs;\CC]$, we see that source and target maps $s^\Cs_{AA}=t^\Cs_{AA}$ are identities and hence $\Xs^\Cs_{A}$ is a trivial groupoid (consisting only of identities), for all $A\in\Os^\Cs$; 
\item 
if $[\omega]_{AB}\in\Xs^\Cs_{AB}$, for any $\omega'\in[\Cs;\CC]$ such that $[\omega]_{AB}=[\omega']_{AB}$, the C*-category $\left(\Cs/\ke(\omega')\right)$ restricted to the two objects $A,B$ is full and hence $\dim_\CC(\Cs_{AB}/\ke(\omega'_{AB}))=1=\dim_\CC(\Cs_{BA}/\ke(\omega'_{BA}))$ and hence $[\omega']_{BA}\in\Xs^\Cs_{BA}$, furthermore $[\omega']_{BA}=[\omega]_{BA}$: since $\omega'(x^*)=u\cdot\omega(x^*)$, with $u\in\TT$, for all $x^*\in\Cs_{AB}$, we obtain $\omega'(x)=\cj{u}\cdot\omega(x)$, for all $x\in\Cs_{BA}$, with $\cj{u}\in\TT$; this shows that the involution map is well-defined and $(\Xs^\Cs_{AB})^*=\Xs^\Cs_{BA}$;  
\item
if $[\omega]_{AB}\in\Xs^\Cs_{AB}$ and $[\rho]_{BC}\in\Xs^\Cs_{BC}$ are composable, we have $[\omega]_{BB}=s_{AB}^\Cs([\omega]_{AB})=t_{AB}^\Cs([\rho]_{BC})=[\rho]_{BB}$ and hence $\omega_{BB}=\rho_{BB}$. The categorical $*$-ideal $\Is$ generated by $\ker(\omega_{BB})=\ker(\rho_{BB})$ coincides with the categorical $*$-ideals generated respectively by $\ker(\omega_{AB})$ and by $\ke(\rho_{BC})$ and, since both categorical \hbox{$*$-ideals} $\ker(\omega)$ and $\ker(\rho)$ coincide with $\Is$ at least on the pair sub-groupoid $\Gs$ of $\Rs^\Cs$ generated by the objects $A,B,C\in\Os^\Cs$, 
the quotient C*-categories $\frac{\Cs}{\ke(\omega)}|_{\Gs}\simeq\frac{\Cs}{ke(\rho)}|_{\Gs}$ are isomorphic. 
As a consequence, $\omega|_{\Gs}$ and $\rho|_{\Gs}$ are related by a unique unitary groupoid 
$\Gs\xrightarrow{u}!^\bullet(\TT)$ such that $\omega_{XY}=u_{XY}\cdot\rho_{XY}$, for all $XY\in\Gs$, so  $[\omega]_{XY}=[\rho]_{XY}$ and we necessarily have $[\omega]_{AC}=[\rho]_{AC}$, hence the composition is well-defined. 

Finally, given $[\omega_1]_{AB}\in\Xs_{AB}^\Cs$ and $[\omega_2]_{BC}\in\Xs^\Cs_{BC}$ that are composable, from the previous argument, there exists at least one $\omega\in[\Cs;\CC]$ such that $[\omega_1]_{AB}\circ[\omega_2]_{BC}=[\omega]_{AC}\in\Xs^\Cs_{AC}$ with $[\omega]_{AB}=[\omega_1]_{AB}$ and $[\omega]_{BC}=[\omega_2]_{BC}$, hence $\Xs^\Cs_{AB}\circ\Xs^\Cs_{BC}\subset\Xs^\Cs_{AC}$.
\end{itemize}

Since $\Xs^\Cs_{AB}$ is open  in $\Xs^\Cs$, for all $AB\in\Rs^\Cs$, the projection onto the discrete pair groupoid $\Rs^\Cs$ is continuous. 

\medskip 

All the operations of identity, involution and composition are continuous in the discrete groupoid $\Rs^\Cs$ and (since they only act on the indexes in $\Rs^\Cs$) they can be similarly defined and are continuous on the topological product $[\Cs;\CC]\times\Rs^\Cs$. Considering the quotient diagram in remark~\ref{rem: X-topology} (and the properties of quotient topologies) we obtain the continuity of the operations on $\Xs^\Cs\subset \bigcup_{AB\in\Rs^\Cs}\Bs^\Cs_{AB}$.  
\end{proof}

\begin{definition}\label{def: E}
The \emph{base spectral spaceoid} $\Xs^\Cs_{\theta^\Cs,\Os^\Cs}$ of a commutative non-full C*-category $\Cs$ is the bundle: 
\begin{equation*}
\Xs^\Cs\xrightarrow{\theta^\Cs} \Rs^\Cs, \quad \theta^\Cs:[\omega]_{AB}\mapsto AB, 
\quad \Xs^\Cs:=\biguplus_{AB\in\Rs^\Cs}\Xs^\Cs_{AB}, \quad \text{with $AB$-fibers} \quad \Xs^\Cs_{AB}, 
\end{equation*}
where $\theta^\Cs$ is a functor from the topological groupoid $\Xs^\Cs$ to the discrete pair groupoid $\Rs^\Cs=\Os^\Cs\times\Os^\Cs$. 

\medskip 

The \emph{bundle spectral spaceoid} of a commutative non-full C*-category $\Cs$ is the bundle: 
\begin{gather*}
\Es^\Cs\xrightarrow{\pi^\Cs}\Xs^\Cs, \quad  \text{with total space} \quad
\Es^\Cs:=\biguplus_{[\omega]_{AB}\in\Xs^\Cs} \Es^\Cs_{[\omega]_{AB}}, \quad \text{and $[\omega]_{AB}$-fibers} \quad 
\Es^\Cs_{[\omega]_{AB}}:=\left(\frac{\Cs}{\ke(\omega)}\right)_{AB},  
\end{gather*}
equipped with the tubular topology induced by the family of \emph{Gel'fand transforms} local sections of $\Es^\Cs$: 
\begin{equation*}
\forall AB\in\Rs^\Cs, \quad\forall x\in \Cs_{AB} \st 
\Xs^\Cs_{AB}\xrightarrow{\hat{x}}\Es_{AB}^\Cs, \quad \hat{x}([\omega]_{AB}):= x+\ke(\omega)_{AB}. 
\end{equation*}
\end{definition}

\begin{remark}\label{rem: Var}
The previous definition requires several explanations: 
\begin{itemize}
\item 
The Fell line-bundle $\Es^\Cs\xrightarrow{\pi^\Cs}\Xs^\Cs$, topologically, is a disjoint union of line-bundles $\left(\Es^\Cs_{AB}\xrightarrow{\pi^\Cs_{AB}}\Xs^\Cs_{AB}\right)_{AB\in\Rs^\Cs}$.  
\item 
For all $x\in\Cs_{AB}$ the Gel'fand transform $\hat{x}$ is a section of the bundle $\Es^\Cs_{AB}\xrightarrow{\pi^\Cs_{AB}}\Xs_{AB}^\Cs$ since, for all $AB\in\Rs^\Cs$, for all 
$[\omega]_{AB}\in\Xs_{AB}^\Cs \st \pi^\Cs(\hat{x}([\omega]_{AB}))=[\omega]_{AB}$. 
\item 
The \emph{tubular topology}, induced by the Gel'fand transforms, is defined as the topology of the total space $\Es^\Cs$ uniquely determined by the bases of neighborhoods $\U_{AB}(e)$, for $e\in\Es_{AB}^\Cs$, $AB\in\Rs^\Cs$ where:
\begin{gather*}
\U_{AB}(e):=\Big\{U^{x,\epsilon}_\O \ | \ x\in \Cs_{AB}, \ \epsilon>0, \ \O\ \text{open in} \ \Xs_{AB}^\Cs, \ \pi^\Cs(e)\in\O \Big\}, 
\\ 
U^{x,\epsilon}_\O:=\Big\{e\in\Es^\Cs \ | \ \ \|e-\hat{x}(\pi^\Cs(e))\|<\epsilon\Big\}\cap(\pi^\Cs)^{-1}(\O).
\end{gather*}
\end{itemize}
Suitable conditions for the existence of such tubular topology, that in our case is determined by the family of Gel'fand transforms and by the Dauns-Hofmann uniformity provided by the norm, have been carefully considered in~\cite{Va84} and~\cite[theorem~3]{Va95}. Such conditions are valid in the present situation, since our bundles $\Es^\Cs_{AB}\xrightarrow{\pi^\Cs_{AB}}\Xs^\Cs_{AB}$ are all saturated and the topologies on each base space $\Xs^\Cs_{AB}$ (that are quotient topologies of a weak-$*$-topologies) guarantee that $\Big\{[\omega]_{AB} \ | \ \|(\hat{x}-\hat{y})([\omega]_{AB})\|_{\Es^\Cs}<\epsilon\Big\}$ is always an open set, for all $\epsilon>0$ and all $x,y\in\Cs_{AB}$. 

\medskip

Since $\Es^\Cs_{AB}$, for $AB\in\Rs^\Cs$, with the tubular topology induced by the Dauns-Hofmann fibre-uniformity determined by the norm and the family of Gel'fand transforms are $\Gamma_o$-saturated Banach bundles (actually locally trivial line-bundles), by lemma~\ref{lem: lcpr} they are also $\Gamma$-saturated Banach bundles with their usual tubular topology induced by continuous sections. 
As a consequence $\Es^\Cs$ is a Fell line-bundle and $\Sigma(\Cs):=(\Es^\Cs,\pi^\Cs,\Xs^\Cs_{\theta^\Cs,\Os^\Cs})$ is a spectral non-full spaceoid as specified in definition~\ref{def: spaceoid}. 
\xqed{\lrcorner}
\end{remark}

\begin{definition}
The \emph{spectrum functor} $\Tf\xleftarrow{\Sigma}\Af$ is defined as follows: 
\begin{itemize}
\item 
to every object $\Cs\in\Af^0$, $\Sigma$ associates the \emph{total spectral spaceoid}  $\Sigma(\Cs):=(\Es^\Cs,\pi^\Cs,\Xs^\Cs_{\theta^\Cs,\Os^\Cs})\in\Tf^0$ (according to definition~\ref{def: spaceoid}),
\item
to every morphism $\Cs^1\xrightarrow{\phi}\Cs^2$ in $\Af^1$, $\Sigma$ associates the morphism $\Sigma(\Cs^1)\xleftarrow{\Sigma_\phi:=(\lambda^\phi,\Lambda^\phi)}\Sigma(\Cs^2)$ in $\Tf^1$, where: 
\begin{gather}\notag
\lambda^\phi_\Rs: \Rs^{\Cs^2} \to \Rs^{\Cs^1}, \quad \lambda^\phi_\Rs: AB\mapsto (\phi^0)^{-1}(A)(\phi^0)^{-1}(B), 
\\ \label{eq: lambda} 
\lambda^\phi_\Xs:\Xs^{\Cs^2}\to\Xs^{\Cs^1}, \quad \lambda_\Xs^\phi:[\omega]_{AB}\mapsto[\omega\circ\phi]_{\lambda^\phi_\Rs(AB)}, 
\\ \notag 
\Lambda^\phi:(\lambda^\phi)^\bullet(\Es^{\Cs^1}) \to \Es^{\Cs^2}, 
\quad 
\Lambda^\phi:=\biguplus_{[\omega]_{AB}\in\Xs^{\Cs^2}} \Lambda^\phi_{[\omega]_{AB}}, 
\quad 
\Lambda^\phi_{[\omega]_{AB}}: \left(\Es^{\Cs^1}_{\lambda^\phi\left([\omega]_{AB}\right)}\right)_{AB}\to\left(\Es^{\Cs^2}_{[\omega]_{AB}}\right)_{AB},
\\ \notag
\Lambda^\phi_{[\omega]_{AB}}: \left(x+\ke(\omega\circ\phi)_{\lambda^\phi_\Rs(AB)}\right)_{AB} \mapsto \Big(\phi(x)+\ke(\omega)_{AB}\Big)_{AB}, \quad x\in\Cs^1_{\lambda^\phi_\Rs(AB)}, \quad  \omega\in[\Cs^2;\CC], 
\end{gather}
where, in order to explicitly see the effect of canonical isomorphisms for pull-backs, we denote by $\left(\Es^{\Cs^1}_{\lambda^\phi([\omega]_{AB})}\right)_{AB} :=\left((\lambda^\phi)^\bullet(\Es^{\Cs^1})\right)_{AB}$ the $AB$-fiber of the $\lambda^\phi$-pull-back bundle of $\Es^{\Cs^1}$ and use a similar notation for its elements.
\end{itemize}
\end{definition} 

The following is just our version of~\cite[proposition~5.8]{BCL11}. 
\begin{proposition}\label{prop: spectral}
The previous definitions provide a well-defined contravariant spectrum functor $\Tf\xleftarrow{\Sigma}\Af$. 
\end{proposition}
\begin{proof}
We start showing that $\Sigma(\Cs):=(\Es^\Cs,\pi^\Cs,\Xs^\Cs_{\theta^\Cs,\Os^\Cs})$ is a non-full spaceoid as in definition~\ref{def: spaceoid}. 

\medskip 

From propositions~\ref{prop: X-topology}, \ref{prop: X-delta} and \ref{prop: X-groupoid}, we have that the base spectral spaceoid consists of a topological groupoid $\Xs^\Cs$, equipped with a continuous functor $\Xs^\Cs\xrightarrow{\theta^\Cs}\Rs^\Cs$ over the discrete pair groupoid $\Rs^\Cs=\Os^\Cs\times\Os^\Cs$, with fibers $\Xs^\Cs_A$ that are compact Hausdorff, and with source and target maps $\Xs^\Cs_A\xleftarrow{t_{AB}}\Xs^\Cs_{AB}\xrightarrow{s_{AB}}\Xs^\Cs_B$ that are homeomorphisms with open subsets. 

\medskip 

From the definition~\ref{def: E}, we already know that, for all $[\omega]_{AB}\in\Xs^\Cs$, the fibers $\Es^\Cs_{[\omega]_{AB}}:=\left(\frac{\Cs}{\ke(\omega)}\right)_{AB}$ are Banach spaces with linear structure and norms given by:
\begin{align*}
&(x+\ke(\omega))_{AB} + \alpha\cdot(x'+\ke(\omega))_{AB}:=((x+\alpha\cdot x')+\ke(\omega))_{AB}, 
\quad x,x'\in\Cs_{AB}, \ \alpha\in \CC, 
\\
&\|(x+\ke(\omega))_{AB}\|_{\Es^\Cs}:=\inf \Big\{\|x+k\|_\Cs \ | \ k\in\ke(\omega_{AB}) \Big\}. 
\end{align*}

The total space $\Es^\Cs$ can be seen to be an involutive category with composition, identity and involution  
\begin{align*}
&\left(x+\ke(\omega)\right)_{AB}\circ\left(x+\ke(\omega)\right)_{BC} 
:=\left(x\circ y+\ke(\omega)\right)_{AC},
\quad x\in\Cs_{AB}, \ y\in \Cs_{BC}, 
\\
&\iota^{\Es^\Cs}_{[\omega]_{AA}}:=\left(\iota^\Cs_A+\ke(\omega)\right)_{AA}, 
\\
&\left(\left(x+\ke(\omega)\right)_{AB}\right)^*:=\left(x^*+\ke(\omega)\right)_{BA}, 
\end{align*}
that are all well-defined thanks to the fact that $\ke(\omega)$ is a $*$-categorical ideal in $\Cs$. 

\medskip 

Since $\frac{\Cs}{\ke(\omega)}$ is a C*-category, we always have, for $x\in\Cs_{AB}$ and $y\in\Cs_{BC}$,  
\begin{align*}
&\|(x+\ke(\omega))_{AB}\circ(y+\ke(\omega))_{BC}\|
\leq\|(x+\ke(\omega))_{AB}\|\cdot\|(y+\ke(\omega))_{AB}\|
\\
&\|\left((x+\ke(\omega))_{AB}\right)^*\circ\left(x+\ke(\omega)\right)_{AB}\|
=\|(x+\ke(\omega))_{AB}\|^2 
\end{align*}
and hence such properties are satisfied inside $\Es^\Cs$, simply restricting to the non-trivial fibers over $[\omega]_{AB}\in\Xs^\Cs$. 

\medskip 

Similarly, since $x^*\circ x\in\Cs_{BB}$, for all $x\in\Cs_{AB}$, is already a positive element in the C*-algebra $\Cs_{BB}$, also the element $(x^*\circ x+\ke(\omega))_{BB}=\left((x+\ke(\omega))_{AB}\right)^*\circ\left(x+\ke(\omega)\right)_{AB}$ is a positive element in the quotient C*-algebra $\Es^\Cs_{[\omega]_{AA}}:=\left(\frac{\Cs}{\ke(\omega)}\right)_{BB}$, for every $\left(x+\ke(\omega)\right)_{AB}\in\Es^\Cs$. 

\medskip 

The projection $\pi^\Cs:\Es^\Cs\to\Xs^\Cs$ given by $(x+\ke(\omega))_{AB}\mapsto [\omega]_{AB}$ is a $*$-functor: 
\begin{align*}
\pi^{\Es^\Cs}\left(\left(x+\ke(\omega)\right)_{AB}\circ\left(x+\ke(\omega)\right)_{BC} \right)
&=\pi^{\Es^\Cs}\left(
(x\circ y+\ke(\omega))_{AC}
\right)=[\omega]_{AC}=[\omega]_{AB}\circ[\omega]_{BC}
\\
&=\pi^{\Es^\Cs}\left((x+\ke(\omega))_{AB}\right)\circ\pi^{\Es^\Cs}\left((y+\ke(\omega))_{BC}
\right), 
\\
\pi^{\Es^\Cs}\left(
\iota^{\Es^\Cs}_{[\omega]_{AA}}
\right)
=\pi^{\Es^\Cs}\left(
(\iota^\Cs_A+\ke(\omega))_{AA}
\right)=[\omega]_{AA}
&=\iota^{\Xs^\Cs}_{[\omega]_{AA}}, 
\\
\pi^{\Es^\Cs}\Big(((x+\ke(\omega))_{AB})^*\Big)
=\pi^{\Es^\Cs}\Big((x^*+\ke(\omega))_{BA}\Big)
&=[\omega]_{BA}=([\omega]_{AB})^*
=\pi^{\Es^\Cs}\Big((x+\ke(\omega))_{AB}\Big)^*.
\end{align*} 

The existence of a Varela tubular topology on $\Es^\Cs$, induced by the family of Gel'fand transforms and by the Dauns-Hofmann uniformity determined by the norm on the fibers of $\Es^\Cs$, is already mentioned in remark~\ref{rem: Var}. 
With such topology $\Es^\Cs$ is a (Fell type) Banach bundle and hence $\pi^\Cs$ is continuous open. 

\medskip 

To show the continuity of compositions in $\Es^\Cs$, consider $e_1,e_2\in\Es^\Cs$, by saturation we can always find $x_1,x_2\in{\color{brown}\Cs^1}$ such that $\hat{x}_j(\pi^\Cs(e_j))=e_j$, for $j=1,2$, and hence $\widehat{x\circ y}(\pi(e_1\circ e_2))=e_1\circ e_2$, and for every $\epsilon>0$ consider any $\delta<\min\{\epsilon,1/2\}$ and notice the following inclusion of tubular neighborhoods  $\T^{\Es^\Cs}_{\delta,x}\circ\T^{\Es^\Cs}_{\delta,y}\subset \T^{\Es^\Cs}_{\delta,x\circ y}$. 
Continuity of involution addition and scalar multiplication are obtained with similar easier arguments. 

\medskip 

This completes the proof that the bundle $(\Es^\Cs,\pi^\Cs,\Xs^\Cs)$ equipped with the tubular topology induced by Gel'fand transforms is a complex (Fell type) Banach (locally trivial) line-bundle that is also a Fell bundle over the topological groupoid $\Xs^\Cs$. 

\medskip 

Next we show that, for any morphism $\Cs^1\xrightarrow{\phi}\Cs^2$ in the category $\Af$, the pair $\Sigma_\phi:=(\lambda^\phi,\Lambda^\phi)$ is a well-defined morphism $\Sigma(\Cs^1)\xleftarrow{\Sigma_\phi}\Sigma(\Cs^2)$ in the category $\Tf$. 

\medskip 

First of all $\Rs^{\Cs^1}\xleftarrow{\lambda^\phi_\Rs}\Rs^{\Cs^2}$ is trivially an isomorphism of discrete pair groupoids: for all $AB,BC\in\Rs^{\Cs^2}$:
\begin{equation*}
\lambda_\Rs^\phi(AB\circ BC)=\lambda_\Rs^\phi(AC)=(\phi^0)^{-1}(A)(\phi^0)^{-1}(C)=(\phi^0)^{-1}(A)(\phi^0)^{-1}(B)\circ(\phi^0)^{-1}(B)(\phi^0)^{-1}(C)=\lambda_\Rs^\phi(AB)\circ\lambda_\Rs^\phi(BC), 
\end{equation*}
furthermore $(\lambda_\Rs^\phi)^1(\iota^{\Rs^{\Cs^2}}(A)) =(\lambda_\Rs^\phi)^1(AA)=(\phi^0)^{-1}(A)(\phi^0)^{-1}(A)=\iota^{\Rs^{\Cs^1}}((\phi^0)^{-1}(A))=\iota^{\Rs^{\Cs^1}}((\lambda^\phi_\Rs)^{0}(A))$, for all $A\in\Os^{\Cs^2}$, and the invertibility of $\lambda^\phi_\Rs$ follows from the invertibility of $\phi^0$. 

\medskip 

We need to show that $\Xs^{\Cs^1}\xleftarrow{\lambda^\phi_\Xs}\Xs^{\Cs^2}$ is actually well-defined by equation~\eqref{eq: lambda}. The $\phi$-pull-back $\omega\mapsto \omega\circ\phi$ is a well-defined map $\phi^\bullet:[\Cs^2;\CC]\to[\Cs^1;\CC]$, as well as its restrictions $[\Cs^2;\CC]_{AB}\xrightarrow{}[\Cs^1;\CC]_{\lambda^\phi_\Rs}(AB)$, given by $\omega_{AB}\mapsto(\omega\circ\phi)_{\lambda^\phi_\Rs(AB)}$, for all $AB\in\Rs^{\Cs^2}$. 
The $\phi$-pull-back passes to the quotient (globally and locally), with respect to the equivalence relations induced by the actions of $\Aut_\Af(!^\bullet(\CC))$, because $\omega_2=\alpha\cdot\omega_1$ implies $\omega_2\circ\phi=\alpha\cdot \omega_1\circ\phi$, for all $\omega_1,\omega_2\in[\Cs^2;\CC]$ and $\alpha\in\Aut_\Af(!^\bullet(CC))$; as a consequence we have induced ``quotiented $\phi$-pull-back'' maps $\lambda^\phi_{\Bs_{AB}}:\Bs^{\Cs^2}_{AB}\xrightarrow{}\Bs^{\Cs^1}_{\lambda^\phi_\Rs(AB)}$, well-defined by  $[\omega]_{AB}\mapsto[\omega\circ\phi]_{\lambda^\phi_\Rs(AB)}$, for all $AB\in\Rs^{\Cs^2}$. 
We need to show that, for all $AB\in\Rs^{\Cs^2}$, the map $\lambda^\phi_{\Bs_{AB}}:\Bs^{\Cs^2}_{AB}\to\Bs^{\Cs^1}_{\lambda^\phi_\Rs(AB)}$ correctly restricts to a map $\lambda^\phi_{\Xs_{AB}}:\Xs^{\Cs^2}_{AB}\to\Xs^{\Cs^1}_{\lambda^\phi_\Rs(AB)}$. 

This is assured by the assumed non-degeneracy condition~\eqref{eq: non-trivial} on the morphism $\phi\in\Af^1$. 

\medskip 

We show now that $\Xs^{\Cs^1}\xleftarrow{\lambda^\phi_\Xs}\Xs^{\Cs^2}$ is a morphism of groupoids such that $\theta^{\Cs^1}\circ\lambda^\phi_\Xs=\lambda^\phi_\Rs\circ\theta^{\Cs^2}$. 

First of all it is immediate to check the equation above, evaluating in an arbitrary element $[\omega]_{AB}\in\Xs^{\Cs^2}$:
\begin{equation*}
\theta^{\Cs^1}\circ\lambda^\phi_\Xs([\omega]_{AB})=
\theta^{\Cs^1}([\omega\circ\phi]_{\lambda^\phi_\Rs(AB)}) =\lambda^\phi_\Rs(AB)=\lambda^\phi_\Rs\circ\theta^{\Cs^2}([\omega]_{AB}).  
\end{equation*}
Next, we check that $\lambda^\phi_\Xs:\Xs^{\Cs^2}\to \Xs^{\Cs^1}$ is a $*$-functor between groupoids: for all $[\omega]_{AB},[\omega]_{BC}\in\Xs^{\Cs^2}$, we have  
\begin{align*}
\lambda^\phi_\Xs([\omega]_{AB}\circ[\omega]_{BC})
&=\lambda^\phi_\Xs([\omega]_{AC})
=[\omega\circ\phi]_{\lambda^\phi_\Rs(AC)}
=[\omega\circ\phi]_{\lambda^\phi_\Rs(AB)}\circ[\omega\circ\phi]_{\lambda^\phi_\Rs(BC)}
\\
&=\lambda^\phi_\Xs([\omega]_{AB})\circ\lambda^\phi_\Xs([\omega]_{BC}). 
\end{align*}

Since $[\omega]_{AA}$ is an identity morphism in $\Xs^{\Cs^2}$ and $\lambda^\phi_\Rs([\omega])_{\lambda^\phi_\Rs(AA)}$ is an identity morphism in $\Xs^{\Cs^1}$
\begin{align*}
\lambda^\phi_\Xs([\omega]_{AA})
=[\omega\circ\phi]_{\lambda^\phi_\Rs(AA)}
=\lambda^\phi_\Rs([\omega])_{\lambda^\phi_\Rs(AA)}, 
\quad \forall A\in\Os^{\Cs^2}, \ \omega\in[\Cs^2;\CC], 
\end{align*}
we have that $\lambda^\phi_\Xs$ preserves identity morphism and hence it is a covariant functor. 

\medskip 

The functor $\lambda^\phi_\Xs:\Xs^{\Cs^2}\to\Xs^{\Cs^1}$ is involutive: 
\begin{align*}
\lambda^\phi_\Xs\left(([\omega]_{AB})^*\right)
=\lambda^\phi_\Xs\left([\omega]_{BA}\right)
=[\omega\circ\phi]_{\lambda^\phi_\Rs(BA)}
=\left([\omega\circ\phi]_{\lambda^\phi_\Rs(AB)}\right)^*
=\lambda^\phi_\Xs\left(([\omega]_{AB})\right)^*. 
\end{align*}

For the continuity of $\Xs^{\Cs^1}\xleftarrow{\lambda^\phi_\Xs}\Xs^{\Cs^2}$, we first observe that the $\phi$-pull-back $[\Cs^2;\CC]\xrightarrow{\phi^\bullet}[\Cs^1;\CC]$ given as $\omega\mapsto \omega\circ\phi$ is a continuous map with respect to the respective weak-$*$-topologies: 
\begin{equation*}
\ev_x\circ\phi^\bullet(\omega)=\ev_x(\phi^\bullet(\omega))=\omega(\phi(x))=\ev_{\phi(x)}(\omega), 
\quad \forall x\in\Cs^1; 
\end{equation*}
then we notice that the $\phi$-pull-back passes continuously to the quotients: the continuity of the map $[\omega]\mapsto [\omega\circ\phi]$ follows from the continuity of $\omega\mapsto [\omega\circ\phi]$ that is the composition of $\phi^\bullet$ with a quotient map. 

\medskip 

To prove that $\lambda^\phi$ is converging to infinity, it is enough to remember that $\lambda^\phi$ is actually defined (and continuous) as a map $\Bs^{\Cs^1}\xleftarrow{\lambda^\phi}\Bs^{\Cs^2}$ between the Alexandroff compactifications of $\Xs^{\Cs_j}$, with $j=1,2$, and it is preserving their points at infinity.

\medskip 

The following algebraic calculations show that $(\lambda^\phi)^\bullet(\Es^{\Cs^1})\xrightarrow{\Lambda^\phi}\Es^{\Cs^2}$ is a fibrewise linear $*$-functor: for all $x,x'\in\Cs^1_{\lambda^\phi_\Rs(AB)}$, $y\in\Cs^1_{\lambda^\phi_\Rs(BC)}$ and $A,B,C\in(\Cs^2)^0$ we have 
\begin{align*}
\Lambda^\phi_{[\omega]_{AB}} &\left(\left(x+\ke(\omega\circ\phi)_{\lambda^\phi_\Rs(AB)}\right)_{AB}
+\left(x'+\ke(\omega\circ\phi)_{\lambda^\phi_\Rs(AB)}\right)_{AB}
\right)
=\Lambda^\phi_{[\omega]_{AB}} \left(
(x+x')+\ke(\omega\circ\phi)_{\lambda^\phi_\Rs(AB)}
\right)_{AB}
\\
&=\Big(
\phi(x+x') +\ke(\omega\circ\phi)_{AB}
\Big)_{AB}
=\Big(
\phi(x) +\ke(\omega\circ\phi)_{AB}
\Big)_{AB} + \Big(
\phi(x') +\ke(\omega\circ\phi)_{AB}
\Big)_{AB}
\\
&=\Lambda^\phi_{[\omega]_{AB}}
\left(x+\ke(\omega\circ\phi)_{\lambda^\phi_\Rs(AB)}\right)_{AB}
+\Lambda^\phi_{[\omega]_{AB}}
\left(x'+\ke(\omega\circ\phi)_{\lambda^\phi_\Rs(AB)}\right)_{AB}, 
\end{align*}

\begin{align*}
\Lambda^\phi_{[\omega]_{AC}}&\left(
(x+\ke(\omega\circ\phi)_{\lambda^\phi_\Rs(AB)} \circ (y+\ke(\omega\circ\phi)_{\lambda^\phi_\Rs(BC)}
\right)
=\Lambda^\phi_{[\omega]_{AC}}\left(
x\circ y +\ke(\omega\circ\phi)_{\lambda^\phi_\Rs(AC)}
\right)
\\
&=\Big(\phi(x\circ y)+\ke(\omega)_{AC}\Big)_{AC}
=\Big(\phi(x)\circ\phi(y)+\ke(\omega)_{AC}\Big)_{AC}
\\
&=\Big(\phi(x)+\ke(\omega)_{AB}\Big)_{AB} \circ \Big(\phi(y)+\ke(\omega)_{BC}\Big)_{BC}
\\
&=\Lambda^\phi_{[\omega]_{AB}}\left((x+\ke(\omega\circ\phi)_{\lambda^\phi_\Rs(AB)} \right)
\circ\Lambda^\phi_{[\omega]_{BC}}\left((y+\ke(\omega\circ\phi)_{\lambda^\phi_\Rs(BC)} \right),
\end{align*}

\begin{align*}
\Lambda^\phi_{[\omega]_{AA}}\left(
\iota_{\lambda^\phi_\Rs(AA)}+\ke(\omega\circ\phi)_{\lambda^\phi_\Rs(AA)} 
\right)_{AA}
=\Big(
\phi(\iota_{\lambda^\phi_\Rs(AA)}) +\ke(\omega\circ\phi)_{AA}
\Big)_{AA}
=\Big(\iota_{(\phi^0)^{-1}(A)} + \ke(\omega\circ\phi)_{AA}\Big)_{AA}, 
\end{align*}

\begin{align*}
\Lambda^\phi_{[\omega]_{BA}}\left(
(x+\ke(\omega\circ\phi)_{\lambda^\phi_\Rs(AB)})^*
\right)_{BA}
&=
\Lambda^\phi_{[\omega]_{BA}}\left(
x^*+\ke(\omega\circ\phi)_{\lambda^\phi_\Rs(AB)}
\right)_{BA}
=
\Big(\phi(x^*)+\ke(\omega)_{BA}\Big)_{BA}
\\
&
=\Big(\phi(x)^*+\ke(\omega)_{BA}\Big)_{BA}
=\Big(\phi(x)+\ke(\omega)_{AB}\Big)_{AB}^*
\\
&
=\left(\Lambda^\phi_{[\omega]_{AB}}\left(x+\ke(\omega\circ\phi)_{\lambda^\phi_\Rs(AB)}\right)_{AB} \right)^*.
\end{align*}

Recallling the definition of $\Lambda^\phi$ from~\eqref{eq: lambda}, 
\begin{equation*}
\Lambda^\phi_{[\omega]_{AB}}: \left(x+\ke(\omega\circ\phi)_{\lambda^\phi_\Rs(AB)}\right)_{AB} \mapsto \Big(\phi(x)+\ke(\omega)_{AB}\Big)_{AB}, \quad x\in\Cs^1_{\lambda^\phi_\Rs(AB)}, \quad  \omega\in[\Cs^2;\CC], 
\end{equation*} 
we see that $\Lambda^\phi$ is fibrewise contractive: for all $k\in\ke(\omega\circ\phi)$, we have 
$\|x+k\|\leq \|\phi(x)+\phi(k)\|$ and hence 
\begin{align*} 
\|\Lambda^\phi\left(x+\ke(\omega\circ\phi)_{\lambda^\phi_\Rs(AB)}\right)_{AB}\|
&=\|\Big(\phi(x)+\ke(\omega)_{AB}\Big)_{AB}\| 
= \inf_{l\in\ker(\omega)}\|\phi(x)+l\, \|
\\
&\leq \inf_{k\in\phi^{-1}(\ke(\omega))} \|\phi(x)+\phi(k)\|
= \inf_{k\in\ke(\omega\circ\phi)}\|\phi(x)+\phi(k)\|
\\
&\leq \inf_{k\in\ker(\omega\circ\phi)}\|x+k\|
=\|\left(x+\ke(\omega\circ\phi)_{\lambda^\phi_\Rs(AB)}\right)_{AB}\|. 
\end{align*} 
Furthermore $\Lambda^\phi$ is mapping the $\lambda^\phi$-pull-back of the Gel'fand transforms of $x\in\Cs^1$ into the Gel'fand transforms of $\phi(x)\in\Cs^2$ and hence it maps $\epsilon$-tubes around $\hat{x}$ into $\epsilon$-tubes around $\phi(x)$:  $\Lambda^\phi(\T^{(\lambda^\phi)^\bullet(\Es^{\Cs^1})}_{\epsilon,\phi(x)})\subset\T^{\Es^{\Cs^1}}_{\epsilon,x}$ and this implies that $\Lambda^\phi$ is continuous with respect to the tubular topologies of the total spaces of the Fell bundles as well as vanishing at infinity. 

\medskip 

We show now the contravariance of the spectrum functor $\Tf\xleftarrow{\Sigma}\Af$. 

Given a pair of composable $*$-functors $\Cs^1\xrightarrow{\psi}\Cs^2\xrightarrow{\phi}\Cs^3$ in the category $\Af$, we get $\Sigma(\Cs^1)\xleftarrow{\Sigma_\psi}\Sigma(\Cs^2)\xleftarrow{\Sigma_\phi}\Sigma(\Cs^3)$ in the category $\Tf$ and we must show $\Sigma_{\phi\circ\psi}=\Sigma_\psi\circ\Sigma_\phi$; in more explicit form, from the following morphisms in $\Tf$  
\begin{equation*} 
\xymatrix{
(\Es^{\Cs^1},\pi^{\Cs^1},\Xs^{\Cs^1}_{\theta^{\Cs^1},\Os^{\Cs^1}})
& \ar[l]^{(\lambda^\psi,\Lambda^\psi)}
(\Es^{\Cs^2},\pi^{\Cs^2},\Xs^{\Cs^2}_{\theta^{\Cs^2},\Os^{\Cs^2}}) 
& \ar[l]^{(\lambda^\phi,\Lambda^\phi)}
(\Es^{\Cs^3},\pi^{\Cs^3},\Xs^{\Cs^3}_{\theta^{\Cs^3},\Os^{\Cs^3}})
\ar@/_0.5cm/[ll]_{(\lambda^{\phi\circ\psi},\Lambda^{\phi\circ\psi})} ,
}
\end{equation*}
we must see that: $\Sigma_{\phi\circ\psi}:=(\lambda^{\phi\circ\psi},\ \Lambda^{\phi\circ\psi})
=(\lambda^\psi\circ\lambda^\phi,\ \Lambda^\phi\circ(\lambda^\phi)^\bullet(\Lambda^\psi)\circ\Theta^{\Es^{\Cs^1}}_{\lambda^\psi,\lambda^\phi})
=:(\lambda^\psi,\Lambda^\psi)\circ(\lambda^\phi,\Lambda^\phi)
=:\Sigma_\psi\circ\Sigma_\phi$:
\begin{align*}
\lambda^\psi_\Rs\circ\lambda^\phi_\Rs(AB)&=\lambda_\Rs^\psi\left((\phi^0)^{-1}(A)(\phi^0)^{-1}(AB)\right)
=(\psi^0)^{-1}((\phi^0)^{-1}(A))(\psi^0)^{-1}((\phi^0)^{-1}(AB))
\\ 
&=(\phi^0\circ\psi^0)^{-1}(A)(\phi^0\circ\psi^0)^{-1}(B)
=\lambda_\Rs^{\phi\circ\psi}(AB), 
\quad \forall AB\in\Rs^{\Cs^3}, 
\\ 
\lambda_\Xs^\psi\circ\lambda_\Xs^\phi([\omega]_{AB})
& 
=\lambda_\Xs^\psi([\omega\circ\phi]_{\lambda_\Rs^\phi(AB)})
=[\omega\circ\phi\circ\psi]_{\lambda_\Rs^\psi\circ\lambda_\Rs^\phi(AB)}
=[\omega\circ(\phi\circ\psi)]_{\lambda_\Rs^{\phi\circ\psi}(AB)}
\\
&=\lambda^{\phi\circ\psi}_\Xs([\omega]_{AB}),
\quad \forall [\omega]_{AB}\in\Xs^{\Cs^3}, 
\\
\text{and hence}\quad 
&\lambda^\psi\circ\lambda^\phi =(\lambda^\psi_\Rs,\lambda^\psi_\Xs)\circ(\lambda^\phi_\Rs,\lambda^\phi_\Xs) =(\lambda^{\phi\circ\psi}_\Rs,\lambda^{\phi\circ\psi}_\Xs)=\lambda^{\phi\circ\psi}, 
\end{align*} 
furthermore, for all $AB\in\Rs^{\Cs^3}$ for all $\omega\in[\Cs^3;\CC]$, and all $x\in\Cs^1_{\lambda^{\phi\circ\phi}_\Rs(AB)}$, we get: 
\begin{align*} 
\Lambda^\phi\circ(\lambda^\phi)^\bullet(\Lambda^\psi)\circ \Theta^{\Es^{\Cs^1}}_{\lambda^\psi,\lambda^\phi}
&\left(x+\ke(\omega\circ\phi\circ\psi)_{\lambda_\Rs^{\phi\circ\psi}(AB)}\right)_{AB}
\\
&=\Lambda^\phi\circ(\lambda^\phi)^\bullet(\Lambda^\psi)
\left(x+\ke(\omega\circ\phi\circ\psi)_{\lambda_\Rs^{\phi\circ\psi}(AB)}\right)_{\lambda^{\phi\circ\psi}_\Rs(AB)}
\\
&=\Lambda^\phi\circ(\lambda^\phi)^\bullet(\Lambda^\psi) \left(
x+\ke(\omega\circ\phi\circ\psi)_{\lambda_\Rs^{\psi}\circ\lambda_\Rs^{\phi}(AB)}
\right)_{\lambda^{\psi}_\Rs\circ\lambda^\phi_\Rs(AB)}
\\
&=\Lambda^\phi_{} \left(\psi(x)+\ke(\omega\circ\phi)_{\lambda_\Rs^{\phi}(AB)}\right)_{\lambda^\phi_\Rs(AB)} 
\\
&=\left(\phi\circ\psi(x)+\ke(\omega)_{AB}\right)_{AB}
\\
&=\Lambda^{\phi\circ\psi}_{[\omega]_{AB}} \left(x+\ke(\omega\circ\phi\circ\psi)_{\lambda_\Rs^{\phi\circ\psi}(AB)}\right)_{AB}. 
\end{align*}
Finally for the identity $*$-functor $\Cs\xrightarrow{\iota_\Cs}\Cs$ we check that 
$\Sigma_{\iota_\Cs}:=(\lambda^{\iota_\Cs},\Lambda^{\iota_\Cs}) =(\iota_{\Xs^\Cs_{\theta^\Cs,\Os^\Cs}},\Theta_{\Es^\Cs})=:\iota_{\Sigma(\Es^\Cs)}$ as follows:
\begin{align*}
\lambda^{\iota_\Cs}_\Rs(AB)
&=((\iota^\Cs)^0)^{-1}(A)((\iota^\Cs)^0)^{-1}(B)=AB 
=\id_{\Rs^\Cs}(AB), \quad \forall AB\in\Rs^\Cs, 
\\
\lambda^{\iota^\Cs}_\Xs([\omega]_{AB})
&= [\omega\circ\iota^{\Cs}]_{\lambda^{\iota^\Cs}_\Rs(AB)}=[\omega]_{AB}
=\id_{\Xs^\Cs}([\omega]_{AB}), \quad \forall [\omega]_{AB}\in\Xs^\Cs, 
\end{align*}
and hence $\lambda^{\iota^\Cs}=(\lambda^{\iota^\Cs}_\Rs,\lambda^{\iota^\Cs}_\Xs)=(\id_{\Rs^\Cs},\id_{\Xs^\Cs}) =\iota_{\Xs^\Cs_{\theta^\Cs,\Os^\Cs}}$ and furthermore, for all $[\omega]_{AB}\in\Xs^\Cs$ and all $x\in\Cs_{AB}$: 
\begin{align*} 
\Lambda^{\iota^\Cs}_{[\omega]_{AB}}\left(x+\ke(\omega\circ\iota^\Cs)_{\lambda^{\iota^\Cs}_\Rs(AB)}\right)_{AB}
&=\left(\iota^\Cs(x)+\ke(\omega)_{AB}\right)_{AB} 
=\left(x+\ke(\omega)_{AB}\right)_{AB}
\\
&=\left(
x+\ke(\omega\circ\iota^\Cs)_{\lambda^{\iota^\Cs}_\Rs(AB)}
\right)_{\lambda^{\iota^\Cs}_\Rs(AB)}
=\Theta_{\Es^\Cs}\left(
x+\ke(\omega\circ\iota^\Cs)_{\lambda^{\iota^\Cs}_\Rs(AB)}
\right)_{AB}, 
\end{align*}
that completes the proof of contravariant functoriality of $\Sigma$. 
\end{proof}

\section{The Gel'fand Natural Isomorphism $\Gg$} \label{sec: 6}
 
We are ready to prove our non-full generalization of~\cite[theorem~6.2]{BCL11}. 
\begin{definition}
Given a C*-category $\Cs\in\Af^0$, we define its \emph{Gel'fand transform} $\Gg_\Cs\in \Hom_\Af(\Cs;\Gamma(\Sigma(\Cs)))$ by: 
\begin{gather*}
\Gg^0_\Cs: A\mapsto A, \quad \forall A\in\Cs^0=\Gamma(\Sigma(\Cs))^0, 
\\ 
\Gg^1_\Cs: x\mapsto \hat{x}, \quad \hat{x}: [\omega]_{AB}\mapsto x+\ke(\omega)_{AB}, \quad \forall  
[\omega]_{AB}\in \Xs^\Cs_{AB}, \quad \forall x\in\Cs^1.  
\end{gather*}
\end{definition}

\begin{theorem}\label{th: Gg}
The Gel'fand transform $\Cs\mapsto \Gg_\Cs$ is a natural isomorphism $\Ig_\Af\xrightarrow{\Gg}\Gamma\circ\Sigma$ between the identity functor $\Ig_\Af$ of the category $\Af$, discussed in section~\ref{sec: 2}, and the endofunctor $\Gamma\circ\Sigma$, obtained composing the spectrum functor $\Sigma$, defined in section~\ref{sec: 5} with the section functor $\Gamma$, defined in section~\ref{sec: 4}. 
\end{theorem}
\begin{proof}
Given a small C*-category $\Cs$, the proof that $\Cs\xrightarrow{\Gg_\Cs}\Gamma(\Sigma(\Cs))$ is an object-bijective $*$-functor and hence a morphism in $\Af^1$ is standard: by definition $\Gamma(\Sigma(\Cs))^0=\Sigma(\Cs)^0=\Cs^0$ and $\Gg_\Cs^0$ is the identity map of $\Cs^0$; furthermore, for all $x,y\in\Cs^1$ and $A\in\Cs^0$:
\begin{gather*}
(\widehat{x+y})({[\omega]_{AB}})=(x+y)+\ke(\omega)_{AB}=(x+\ke(\omega)_{AB})+(y+\ke(\omega)_{AB})
=(\hat{x}+\hat{y})({[\omega]_{AB}}), 
\\
(\widehat{x\circ y})({[\omega]_{AB}})=(x\circ y)+\ke(\omega)_{AB}=(x+\ke(\omega)_{AB})\circ (y+\ke(\omega)_{AB})
=(\hat{x}\circ\hat{y})({[\omega]_{AB}}), 
\\
(\widehat{x^*})({[\omega]_{AB}})=x^*+\ke(\omega)_{AB}=(x+\ke(\omega))^*)_{[\omega]_{AB}} =((\hat{x})^*)({[\omega]_{AB}}), 
\\
(\widehat{\iota^\Cs}_A)({[\omega]_{AB}})=\iota^\Cs_A +\ke(\omega)_{AB} =(\iota^{\Gamma(\Sigma(\Cs))}_A)({[\omega]_{AB}}). 
\end{gather*} 
The natural transformation property of $\Gg$ is obtained in exactly the same way as in the proof of the commuting square in~\cite[page~596]{BCL11}: 
for all $x\in\Cs_1^1$, for all $[\omega]_{AB}\in\Xs^{\Sigma(\Cs_2)}$
\begin{align*}
\left(\Gamma_{\Sigma^\phi}\circ\Gg_{\Cs_1}(x)\right)([\omega]_{AB})
&=\left(\Gamma_{\Sigma^\phi}(\hat{x})\right)([\omega]_{AB})
=\left(\Lambda^\phi\circ(\lambda^\phi)^\bullet(\hat{x})\right)([\omega]_{AB})
=\Lambda^\phi\circ\hat{x}([\omega\circ\lambda^\phi]_{\lambda^\phi(AB)})
\\
&=\Lambda^\phi\left(x+\ker(\omega\circ\lambda^\phi)_{\lambda^\phi(AB)}\right) 
=\phi(x) +\ker(\omega)_{AB}
=\widehat{\phi(x)}([\omega]_{AB})
\\
&=\left(\Gg_{\Cs_2}\circ\phi(x)\right)([\omega]_{AB}).
\end{align*} 
The isometric property of $\Gg$ is immediate: $\|\Gg_\Cs(x)\|^2=\|\Gg_\Cs(x)^*\circ\Gg_\Cs(x)\|=\|\Gg_\Cs(x^*\circ x)\|=\|x^*\circ x\|=\|x\|^2$, for all $x\in \Cs_{AB}$, by the C*-property and by the usual isometric property of the Gel'fand-Na\u\i mark transform for the commutative C*-algebra $\Cs_{BB}$ that contains $x^*\circ x$. The injectivity of $\Gg_\Cs$ follows from its isometricity. 

\medskip 

The image $\Gg_\Cs(\Cs)\subset\Gamma(\Sigma(\Cs))$ is closed under composition, involution and identity (since $\Gg_\Cs$ is a morphism of C*-categories); furthermore $\Gg_\Cs(\Cs)$ is complete (since $\Gg_\Cs$ is isometric) and hence closed inside the C*-category $\Gamma(\Sigma(\Cs))$. 
The surjectivity of $\Gg_\Cs$ follows as soon as we can prove that $\Gg_\Cs(\Cs)$ is uniformly dense in $\Gamma(\Sigma(\Cs))$. 

\medskip 

The density of $\Gg_\Cs(\Cs_{AA})$ into $\Gamma(\Sigma(\Cs_{AA}))$, for all objects $A\in\Cs^0$, is a consequence of the Gel'fand-Na\u\i mark theorem for the commutative unital C*-algebra $\Cs_{AA}$, since we have the canonical isomorphism of C*-algebras $\Gamma(\Sigma(\Cs_{AA}))\simeq C(\Sp(\Cs_{AA}))$. 

\medskip 

For general $A,B\in\Cs^0$, we directly prove the density of $\Gg_\Cs(\Cs_{AB})$ into $\Gamma(\Sigma(\Cs))_{AB}$. Consider an arbitrary continuous section $\sigma\in\Gamma(\Sigma(\Cs))_{AB}$ vanishing at infinity. For every point $[\omega]_{AB}\in\Xs^\Cs_{AB}$, from the saturation of the bundle $\Es^\Cs_{AB}$, we can find an element $x_{[\omega]}\in\Cs_{AB}$ whose Gel'fand transform $\hat{x}_{[\omega]}\in\Gamma(\Sigma(\Cs))_{AB}$ satisfies $\hat{x}_{[\omega]}([\omega]_{AB})=\sigma([\omega]_{AB})$. 
From the continuity of $\hat{x}_{[\omega]}$, and the continuity of the norm of the Fell-type Banach bundle $\Es^\Cs_{AB}$, we actually know that there exists an open neighborhood $\U_{[\omega]}$ of $[\omega]_{AB}$ such that $\|\sigma([\rho]_{AB}) -\hat{x}_{[\omega]}([\rho]_{AB})\|_{\Es^\Cs}<\epsilon$, for all $[\rho]_{AB}\in\U_{[\omega]}$. 
The family $\{\U_{[\omega]} \ | \ [\omega]_{AB}\in\Xs^\Cs_{AB}\}$ is an open covering of the locally compact Hausdorff space $\Xs^\Cs_{AB}$. 
Let us consider $\hat{\Xs}^\Cs_{AB}:=\Xs^\Cs_{AB}\uplus\{\infty_{AB}\}$, the Alexandroff compactification of $\Xs^\Cs_{AB}$; since both $\sigma$ and $\hat{x}_{[\omega]}$ are vanishing at infinity, we actually have that $\hat{\U}_{[\omega]}:=\U_{[\omega]}\cup\{\infty_{AB}\}$ is an open neighborhood of the point at infinity $\infty_{AB}$ and hence the family $\{\hat{\U}_{[\omega]} \ | \ [\omega]_{AB}\in\Xs^\Cs_{AB}\}$ is an open covering of the compact Hausdorff topological space $\hat{\Xs}^\Cs_{AB}$. By compactness of $\hat{\Xs}^\Cs_{AB}$ we can extract a finite open subcovering $\{\hat{\U}_{[\omega_j]} \ | \ j=1,\dots,N\}$ of $\hat{\Xs}^\Cs_{AB}$ and hence a finite open covering $\{\U_{[\omega_j]} \ | \ j=1,\dots,N\}$ of $\Xs^\Cs_{AB}$. Since $\hat{\Xs}^\Cs_{AB}$ is compact Hausdorff, it is paracompact and hence it admits a partition of unity $\{f_j \ | \ j=1,\dots,N\}$ with continuous functions with compact support subordinated to the covering $\{\hat{\U}_{[\omega_j]} \ | \ j=1,\dots,N\}$. 

\medskip 

Since the target map $t^\Cs_{AB}: \Xs^\Cs_{AB}\to\Xs^\Cs_{AA}$ is an homeomorphism onto an open subset of the normal space $\Xs^\Cs_{AA}$, the $(t^\Cs_{AB})^{-1}$-pull-back of the functions $f_j$, for $j=1,\dots,N$, gives continuous functions on the open set $t^\Cs_{AB}(\Xs^\Cs_{AB})$ that admit continuous extensions $g_j$ to all of $\Xs^\Cs_{AA}$, with respective constant values $f_j(\infty_{AB})$ on $\Xs^\Cs_{AA}-t^\Cs_{AB}(\Xs^\Cs_{AB})$. 
Making full usage of the already mentioned canonical isomorphism $C(\Xs^\Cs_{AA})\simeq\Gamma(\Es^\Cs_{AA})$, we can identify the continuous functions $\{g_j \ | \ j=1,\dots,N\}\subset C(\Xs^\Cs_{AA})$ with continuous sections in $\Gamma(\Es^\Cs_{AA})$. 
The continuous section $\tau_\epsilon:=\sum_{j=1}^Ng_j\circ\hat{x}_{[\omega_j]}$ is vanishing at infinity (since all the $g_j$ are bounded), hence $\tau_\epsilon\in\Gamma(\Sigma(\Cs))_{AB}$ and furthermore satisfies $\|\sigma([\omega]_{AB}) -\tau_\epsilon([\rho_{AB}])\|<\epsilon$, for all $[\rho]_{AB}\in\Xs^\Cs_{AB}$. 
Using the Gel'fand isomorphism on the diagonal Hom-set $\Cs_{AA}$, we have the existence of $\{z_j \ | \ j=1,\dots,N\}\in\Cs_{AA}$ such that $\hat{z}_j=g_j$, for all $j=1,\dots,N$, and hence 
$\tau_\epsilon=\sum_{j=1}^N \hat{z}_j\circ \hat{x}_j\in\Gg_\Cs(\Cs_{AB})$ and this completes the proof of uniform density of $\Gg_\Cs(\Cs_{AB})$ into $\Gamma(\Sigma(\Cs))_{AB}$ and hence the surjectivity of $\Gg_\Cs$. 
\end{proof} 

\section{The Evaluation Natural Isomorphism $\Eg$} \label{sec: 7}

Here we obtain the non-full version of\cite[theorem~6.3]{BCL11}. 
\begin{definition}
Given a spaceoid $(\Es,\pi,\Xs_{\theta,\Os})\in\Tf^0$, its \emph{evaluation transform} $\Eg_\Es\in \Hom_\Tf(\Es;\Sigma(\Gamma(\Es)))$ is defined by $\Eg_\Es:=(\eta^\Es,\Omega^\Es)$, with $\eta^\Es:=(\eta^\Es_\Xs,\eta^\Es_\Rs)$ where: 
\begin{gather*}
\eta^\Es_\Xs: \Xs\to \Xs^{\Gamma(\Es)}, 
\quad 
\eta^\Es_\Xs: p_{AB} \mapsto [\ev_{p_{AB}}]_{AB}, 
\quad 
\ev_{p_{AB}} :\sigma\mapsto \sigma(p_{AB})\in\Es_{p_{AB}}, 
\\
\eta^\Es_\Rs: \Os \to \Os^{\Gamma(\Es)}, \quad \eta_\Rs^\Es: AB\mapsto AB, 
\\
\Omega^\Es: (\eta_\Xs^\Es)^\bullet (\Es^{\Gamma(\Es)})\to \Es, \quad 
\Omega^\Es_{p_{AB}}:\left(\sigma+\ke(\eta^\Es_\Xs(p_{AB}))_{\eta^\Es_\Rs(AB)}\right)_{p_{AB}}\mapsto \sigma(p_{AB}). 
\end{gather*}
\end{definition}

\begin{theorem}\label{th: Eg}
The evaluation transform $(\Es,\pi,\Xs_{\theta,\Os})\mapsto \Eg_\Es$ is a natural isomorphism $\Ig_\Tf\xrightarrow{\Eg}\Sigma\circ\Gamma$ 
between the identity functor $\Ig_\Tf$ of the category $\Tf$, discussed in section~\ref{sec: 3}, and the endofunctor $\Sigma\circ\Gamma$, obtained composing the section functor $\Gamma$, defined in section~\ref{sec: 4} with 
the spectrum functor $\Sigma$, defined in section~\ref{sec: 5}.
\end{theorem}
\begin{proof} 
Since $\Os^{\Gamma(\Es)}=\Os$, we have $\Rs^{\Sigma(\Gamma(\Es))}=\Rs_\Os$ and $\eta_\Rs$ is the identity isomorphism of discrete pair groupoids. 

\medskip 

For a given $p_{AB}\in\Xs$, consider the evaluation map $\ev_{p_{AB}}:\Gamma(\Es)_{AB}\to \Es_{p_{AB}}$ given by $\sigma\mapsto \sigma(p_{AB})$. 
Making use of remark~\ref{rem: p}, we have a maximal pair subgroupoid $\Xs^p\subset\Xs$ of composable points $p_{CD}\in\Xs_{CD}$, with $CD\in\Rs^p:=\theta(\Xs^p)$ and hence we can uniquely extend the map $\ev_{p_{AB}}$ to the C*-category $\Gamma(\Es)|_{\Rs^p}$ as follows: 
\begin{equation*}
\ev_p:\Gamma(\Es)|_{\Rs^p}\to \Es^p, \quad \quad \sigma\mapsto \sigma(p_{CD})\in\Es_{p_{CD}}, 
\quad \quad \forall CD\in\Rs^p, \quad \quad \forall \sigma\in\Gamma(\Es)_{CD}.
\end{equation*} 
With some abuse of notation, we continue to denote in the same way 
$\ev_p:\Gamma(\Es)\to \Es^p$
the unique canonical ``extension to zero'' of the previous map, where now we define 
$\Es^p|_{CD}:=\{0_{p_{CD}}\}$, for all $CD\notin\Rs^p$. 

\medskip 

With the only possible extensions (to zero) of compositions identities and involutions, $\Es^p$ is a C*-category of rank less or equal to one, with $\Es^p|_{\Xs^p}=\Es|_{\Xs^p}$ of rank one, and $\ev_p:\Gamma(\Es)\to\Es^p$ is a $*$-functor. 

\medskip 

The one-dimensional C*-category $\Es^p$ is (non-canonically) isomorphic to the unique extension to zero of  the rank-one C*-category $!^\bullet(\CC)|_{\Xs^p}$, but any possible choice of a pair of such isomorphisms are related by a unique element of $\alpha\in\Aut_\Af(!^\bullet(\CC)|_{\Xs^p})$ and hence there exists a unique well-defined element in $[\Gamma(\Es);\CC]$, that with abuse of notation we denote as $[\ev_p]\in[\Gamma(\Es);\CC]$ and so $\eta_\Xs:p_{AB}\mapsto[\ev_p]_{AB}$ is perfectly well-defined. 

\medskip 

Since $\ev_p$ is a $*$-functor, from the already proved properties of formal composition in $\Xs$ and hence in $\Xs^{\Gamma(\Es)}$, we have that $\eta_\Xs:\Xs\to\Xs^{\Gamma(\Es)}$ is a $*$-functor between groupoids with $\theta^{\Gamma(\Es)}\circ \eta_\Xs=\eta_\Rs\circ\theta$. 

\medskip 

Since $\eta_\Xs$ is the composition $\Xs\xrightarrow{\ev}[\Gamma(\Es);\Es^p]\xrightarrow{}\Xs^{\Gamma(\Es)}$, where the second is a continuous quotient map, the continuity of $\eta_\Xs$ is following from the continuity of the map 
$\ev: p\mapsto \ev_p$ with respect to the weak topology induced by evaluations on $[\Gamma(\Es),\Es^p]$ and this is a consequence of the continuity, for all $\sigma\in\Gamma(\Es)$, 
of the map $p_{AB}\mapsto (\ev_{\sigma}\circ\ev)(p)=\ev_\sigma(\ev_{p_{AB}})=\ev_{p_{AB}}(\sigma)=\sigma(p_{AB})$. 

\medskip 

On diagonal Hom-sets, by the well-known Gel'fand evaluation isomorphism $\Eg^{\Xs_A}:\Xs_{AA}\simeq\Xs^{\Gamma(\Es)}_{AA}$, we already have that the restriction map $\Xs_{AA}\xrightarrow{\eta_\Xs}\Xs^{\Gamma(\Es)}_{AA}$ is a homeomorphism, for all $A\in\Os$. 
We need to prove the homeomorphism property for off-diagonal Hom-sets. 

\medskip 

Considering the commuting diagrams below with the respective source/target maps:
\begin{equation*}
\xymatrix{
\Xs_{AB} \ar[d]_{s^\Xs} \ar[rrr]^{\eta_\Xs} & & & \ar[d]^{s^{\Xs^{\Gamma(\Es)}}}\Xs_{AB}^{\Gamma(\Es)}
\\
\Xs_{AA} \ar[rrr]_{\Eg^{\Xs_A}} & & & \Xs_{AA}^{\Gamma(\Es)}
}, 
\quad\quad 
\xymatrix{
\Xs_{AB} \ar[d]_{t^\Xs} \ar[rrr]^{\eta_\Xs} & & & \ar[d]^{t^{\Xs^{\Gamma(\Es)}}}\Xs_{AB}^{\Gamma(\Es)}
\\
\Xs_{AA} \ar[rrr]_{\Eg^{\Xs_A}} & & & \Xs_{AA}^{\Gamma(\Es)}
}
\end{equation*}
since we already know that sources (respectively targets) are homeomorphisms with open subsets of $\Xs_{A}$ and $\Xs^{\Gamma(\Es)}_A$ and $\Eg^{\Gamma(\Es)_[AA]}$ is the Gel'fand evaluation homeomorphism, we obtain that the map $\eta_\Xs$ is a homeomorphism between open locally compact sets. 

\medskip 

Since for any two $\sigma_1,\sigma_2\in\Gamma(\Es)_{AB}$, we have  $\sigma_1-\sigma_2\in\ke(\eta^\Es_\Xs(p_{AB}))$ if and only if  $\sigma_1(p_{AB})=\sigma_2(p_{AB})\in\Es_{p_{AB}}$, we see that $\Omega^\Es$ is a well-defined map fibrewise acting between the Fell bundles and injective. 

\medskip 

A direct calculation shows that $\Omega^\Es$ is a $*$-functor: $\forall \sigma_1,\sigma_2,\sigma\in\Gamma(\Es)_{AB}, \ \rho\in\Gamma(\Es)_{BC}, \ p_{AB}\in\Xs_{AB}, \ p_{BC}\in\Xs_{BC}$, 
\begin{align}\notag
&\Omega^\Es(\sigma_1+\sigma_2 +\ke(\eta^\Es_\Xs(p_{AB})) 
=(\sigma_1+\sigma_2)(p_{AB}) =\Omega^\Es(\sigma_1 +\ke(\eta^\Es_\Xs(p_{AB}))+\Omega^\Es(\sigma_2 +\ke(\eta^\Es_\Xs(p_{AB})), 
\\ \label{eq: Omega-circ}
&\begin{aligned}
\Omega^\Es(\sigma\circ\rho +\ke(\eta^\Es_\Xs(p_{AC})) 
&=(\sigma\circ\rho)(p_{AC})
=\sigma(p_{AB})\circ\rho(p_{BC}) 
\\
&=\Omega^\Es(\sigma+\ke(\eta^\Es_\Xs(p_{BC})\circ\Omega^\Es(\rho+\ke(\eta^\Es_\Xs(p_{BC}))),
\end{aligned}
\\ \notag
&\Omega^\Es(\sigma^* + \ke(\eta^\Es_\Xs(p_{BA})))=\sigma^*(p_{BA})=\sigma(p_{AB})^*
=\Omega^\Es(\sigma+\ke(\eta^\Es_\Xs(p_{AB})))^*, 
\\ \notag 
&\Omega^\Es(\iota^{\Gamma(\Es)}_A + \ke(\eta^\Es_\Xs(p_{AA}))) =\iota^{\Gamma(\Es)}_A(p_{AA}) = \iota^\Es_{{p_{AA}}},  
\end{align}
and, in equation~\eqref{eq: Omega-circ} for points that are not composable, both terms are zero of the respective fibers.  

\bigskip 

To prove the continuity of $\Omega^\Es$, consider an arbitrary point $e\in\Es_{AB}$, by the saturation property of $\Es$ there exists at least one continuous section $\sigma\in\Gamma(\Es_{AB})$ such that $\sigma(\pi(e))=e$. 
By lemma~\ref{lem: lcpr} it is actually not restrictive to consider a section $\sigma\in\Gamma_o(\Es)_{AB}$
vanishing at infinity. 
Since tubular neighborhoods $\T^\Es_{\epsilon}(\sigma)\cap\pi^{-1}(\U)$, with $\U$ an open neighborhood of $\pi(e)\in\Xs$, are a base of neighborhoods of $e\in\Es$, consider the tubular neighborhood $\T^{\Es^{\Gamma(\Es)}}_{\epsilon}(\hat{\sigma})\cap\pi^{-1}(\U)$ around the point $\hat{\sigma}(\pi(e))\in\Es^{\Gamma(\Es)}$ determined by the Gel'fand transform $\hat{\sigma}$ of the element $\sigma\in\Gamma(\Es)_{AB}$. Notice that $\Omega^\Es(\hat{\sigma}(\pi(e)))=e$ and furthermore $\Omega^\Es(\T^{\Gamma(\Es)}_{\epsilon}(\hat{\sigma})\cap\pi^{-1}(\U))\subset\T^\Es_\epsilon(\sigma) \cap \pi^{-1}(\U)$, that gives the required continuity. 

\medskip 

Since both $\Es^{\Gamma(\Es)}$ and $\Es$ are line-bundles and fibrewise $\Omega^\Es$ is non-trivial, $\Omega^\Es$ is also surjective. From the already proved fibrewise injectivity, $\Omega^\Es$ is bijective and continous. Its inverse map is also continuous, since now the tubular base-neighborhoods are in bijective correspondence $\Omega^\Es(\T^{\Es^{\Gamma(\Es)}}_\epsilon)(\hat{\sigma})\cap\pi^{-1}(\U)=\T^\Es_\epsilon(\sigma)\cap\U$.  

\medskip 

In this way we got that $\Omega^\Es$ is an isomorphism of Fell bundles and  $\Eg=(\eta^\Es,\Omega^\Es)$ is an isomorphism of spaceoids. 

\medskip 

The natural transformation property of $\Eg$ is obtained in exactly the same way as in the proof of the commuting square in~\cite[page~597]{BCL11}. 
For every morphism in the category $\Tf$ of spaceoids   $(\Es^1,\pi^1,\Xs^1_{\theta^1,\Os^1})\xrightarrow{(f,F)}(\Es^2,\pi^2,\Xs^2_{\theta^2,\Os^2})$, we must show that $\Sigma_{\Gamma_{(f,F)}}\circ\Eg_{\Es^1}=\Eg_{\Es^2}\circ(f,F):\Es^1\to\Sigma(\Gamma(\Es^2))$. 
This means $(\lambda^{\Gamma_{(f,F)}},\Lambda^{\Gamma_{(f,F)}})\circ(\eta^{\Es^1},\Omega^{\Es^1})
=(\eta^{\Es^2},\Omega^{\Es^2})\circ(f,F)$
or, explicitly calculating compositions:  
\begin{equation*}
\left(\lambda^{\Gamma_{(f,F)}}\circ\eta^{\Es^1}, \ \Omega^{\Es^1}\circ(\eta^{\Es^1})^\bullet(\Lambda^{\Gamma_{(f,F)}}) \circ\Theta^{\Es^{\Gamma(\Es^2)}}_{\lambda^{\Gamma_{(f,F)}},\eta^{\Es^1}}\right)
=\left(\eta^{\Es^2}\circ f, \ F\circ f^\bullet(\Omega^{\Es^2}) \circ\Theta^{\Es^{\Gamma(\Es^2)}}_{\eta^{\Es^2},f}\right), 
\end{equation*} 
For all $p_{AB}\in\Xs^1$, for all $\sigma\in\Gamma(\Es^2)_{f_\Rs(AB)}$, we get:
\begin{align*}
&\lambda^{\Gamma_{(f,F)}}\circ\eta^{\Es^1}(p_{AB})
=\lambda^{\Gamma_{(f,F)}}([\ev_{(p_{AB})}]_{AB})
=[\ev_{(p_{AB})}\circ \Gamma_{(f,F)}]_{AB}, 
&&\eta^{\Es^2}\circ f(p_{AB})
=[\ev_{f_\Xs(p_{AB})}]_{f_{\Rs}(AB)}, 
\\
&\ev_{(p_{AB})}\circ \Gamma_{(f,F)}(\sigma)=\ev_{p_{AB}}(F(f^\bullet(\sigma)))=F(\sigma(f(p_{AB}))), 
&&\ev_{f_\Xs(p_{AB})}(\sigma)=\sigma(f(p_{AB})). 
\end{align*}
Since $F$ is a $*$-functor, if $F(e^*e)$ is non zero also $F(e)$ is nonzero, for all non-zero $e\in\Es$, and hence we have an isomorphism of full 1-C*-sub-categories $F(f^\bullet(\Es^2|_{\Xs^2_{f(p)}}))\subset \Es^1|_{\Xs^1_p}$ and considering its unique zero extension we obtain the required equality of the equivalence classes $[\ev_{(p_{AB})}\circ \Gamma_{(f,F)}]_{AB}=[\ev_{f_\Xs(p_{AB})}]_{f_{\Rs}(AB)}$. 

\medskip 

Finally notice that, by the previous calculation, we have, for all $\sigma+\ke(\lambda_\Xs^{\Gamma_{(f,F)}}\circ\eta_\Xs^{\Es^1}(p_{AB}))\in\Es^{\Gamma(\Es^2)}$: 
\begin{align*}
\Omega^{\Es^1}\circ(\eta^{\Es^1})^\bullet(\Lambda^{\Gamma_{(f,F)}}) 
&\circ 
\Theta^{\Es^{\Gamma(\Es^2)}}_{\lambda^{\Gamma_{(f,F)}},\eta^{\Es^1}}  \left(\sigma+\ke(\lambda_\Xs^{\Gamma_{(f,F)}}\circ\eta_\Xs^{\Es^1}(p_{AB}))_{\lambda_\Rs^{\Gamma_{(f,F)}}\circ\eta_\Rs^{\Es^1}(AB)}\right)_{p_{AB}}
\\
&=\Omega^{\Es^1}\circ(\eta^{\Es^1})^\bullet(\Lambda^{\Gamma_{(f,F)}})  
\left(\left(\sigma+\ke(\lambda_\Xs^{\Gamma_{(f,F)}}\circ\eta_\Xs^{\Es^1}(p_{AB}))_{\lambda_\Rs^{\Gamma_{(f,F)}}\circ\eta_\Rs^{\Es^1}(AB)}\right)_{\eta^{\Es^1}(p_{AB})}\right)_{p_{AB}} 
\\
&=\Omega^{\Es^1}\left(\Gamma_{(f,F)}(\sigma) +\ke(\eta^{\Es^1}_\Xs(p_{AB}))_{\eta^{\Es^1}_\Rs(AB)}\right)_{p_{AB}}
\\
&=\Omega^{\Es^1}\left(F\circ(f^\bullet(\sigma)) +\ke(\eta^{\Es^1}_\Xs(p_{AB}))_{\eta^{\Es^1}_\Rs(AB)}
\right)_{p_{AB}}
\\
&=F\Big(f^\bullet(\sigma)(p_{AB})\Big)
\\
&=F(\sigma(f(p_{AB}))) 
\\
&=F\circ f^\bullet(\Omega^{\Es^2}) 
\left(
\left(
\sigma +\ke(\eta_\Xs^{\Es^2}\circ f_\Xs(p_{AB}))_{\eta_\Rs^{\Es^2}\circ f_\Rs(AB)}
\right)_{f(p_{AB})}
\right)_{p_{AB}} 
\\
&=F\circ f^\bullet(\Omega^{\Es^2})\circ \Theta^{\Es^{\Gamma(\Es^2)}}_{\eta^{\Es^2},f}
\left(
\sigma +\ke(\eta_\Xs^{\Es^2}\circ f_\Xs(p_{AB}))_{\eta_\Rs^{\Es^2}\circ f_\Rs(AB)}
\right)_{p_{AB}} 
\\
&=F\circ f^\bullet(\Omega^{\Es^2})\circ \Theta^{\Es^{\Gamma(\Es^2)}}_{\eta^{\Es^2},f}
\left(
\sigma +\ke(\lambda_\Xs^{\Gamma_{(f,F)}}\circ\eta_\Xs^{\Es^1}(p_{AB}))_{\lambda_\Rs^{\Gamma_{(f,F)}}\circ\eta_\Rs^{\Es^1}(AB)}
\right)_{p_{AB}}. 
\end{align*}
This completes our proof of naturality for the isomorphism $(\eta^\Es,\Omega^\Es)$. 
\end{proof}

\section{Gel'fand Na\u\i mark Duality for Non-full Commutative C*-categories} \label{sec: 8}

Finally we get our non-full extension of the duality~\cite[theorem~6.4]{BCL11}. 
\begin{theorem}
There {\color{brown} is} a categorical duality $\xymatrix{\Tf \rtwocell^{\Gamma}_{\Sigma}{'} & \Af}$ with units $\id_\Af\xrightarrow{\Gg}\Gamma\circ\Sigma$ and $\id_\Tf\xrightarrow{\Eg}\Sigma\circ\Gamma$.
\end{theorem}
\begin{proof}
The proof is already following from theorem~\ref{th: Gg} and theorem~\ref{th: Eg}. 
\end{proof} 

\section{Spectral Theorem for Non-full Imprimitivity Hilbert C*-Bimodules}\label{sec: 9}

As an application of proposition~\ref{prop: spectral}, our spectral theorem for non-full C*-categories already entails a generalization of the spectral theorem for imprimitivity C*-modules over commutative (unital) C*-algebras given in~\cite{BCL08}. 
The result is an explicit form of Rieffel's correspondence~\cite{BGR77} (see also~\cite[theorem~2.1]{KQW24} for a recent exposition) for certain types of Hilbert C*-bimodules.  

\medskip 

First of all, let us specify the type of Hilbert C*-bimodules that can be treated. 
\begin{definition}
An \emph{non-full imprimitivity Hilbert C*-bimodule} ${}_\As\Ms_\Bs$ over the pair of unital commutative \hbox{C*-algebras} $\As$ and $\Bs$, consists of a left-$\As$, right-$\Bs$ bimodule $\Ms$ that is equipped with a left $\As$-valued inner product $(x,y)\mapsto{}_\As\ip{x}{y}$ and a right $\Bs$-valued inner product $(x,y)\mapsto\ip{x}{y}_\Bs$, for all $x,y\in\Ms$, such that: ${}_\As\ip{x}{y}\cdot z=x\cdot\ip{y}{z}_\Bs$, for all $x,y,z\in\Ms$ and $\Ms$ is complete in the norms induced by the inner products. 
\end{definition} 
Notice that as soon as the two inner products above are full, the above definition reduces to the usual imprimitivity case, for which well-known spectral results are available. 

\medskip 

We need to make essential usage of a slight generalization, to the non-full situation, of the linking C*-category introduced in~\cite{BGR77} for imprimitivity C*-bimodules. 
\begin{lemma}\label{lem: linking}
To every non-full imprimitivity Hilbert C*-bimodule, ${}_\As\Ms_{\Bs}$ over a pair of unital commutative C*-algebras $\As,\Bs$, 
we can always associate a two-object \emph{linking non-full C*-category} $\Cs({}_\As\Ms_\Bs)$ as follows: 
\begin{equation*}
\Cs({}_\As\Ms_\Bs):=
\begin{bmatrix}
\As & \Ms 
\\
\Ms^* & \Bs
\end{bmatrix}, 
\quad \Cs({}_\As\Ms_\Bs)_{AA}:=\As, \quad \Cs({}_\As\Ms_\Bs)_{AB}:=\Ms, \quad 
\Cs({}_\As\Ms_\Bs)_{BA}:=\Ms^*, \quad \Cs({}_\As\Ms_\Bs)_{BB}:=\Bs, 
\end{equation*}
where ${}_\Bs{\Ms^*}_{\As}:=\{x^* \ | \ x\in\Ms\}$ denotes the Rieffel dual of ${}_\As\Ms_\Bs$ as a non-full imprimitivity left-$\Bs$, right-$\As$ C*-bimodule with actions $b\cdot x^*\cdot a:=(axb)^*$, for all $a,b,x\in\As\times\Bs\times\Ms$ and inner products given by 
${}_\Bs\ip{x^*}{y^*}:=\ip{x}{y}_\Bs^*$ and $\ip{x^*}{y^*}_\As:={}_\As\ip{x}{y}^*$, for all $x,y\in\Ms$; 
and where we define: 
\begin{itemize}
\item
composition via actions of C*-bimodules and $x^*\circ y:=\ip{x}{y}_\Bs$, 
$x\circ y^*:={}_\As\ip{x}{y}$, for all $x,y\in\Ms$
\item
involution as involution in the C*-algebras $\As$, $\Bs$ and by $x\mapsto x^*$, $x^*\mapsto x$, when $x\in\Ms$, 
\item
identities given as: $\iota^{\Cs(\Ms)}_A:=1_\As$ and $\iota^{\Cs(\Ms)}_B:=1_\Bs$, 
\item 
norm as already defined in each of the Hom-sets. 
\end{itemize}
\end{lemma}
\begin{proof} 
The algebraic properties of the involutive category $\Cs({}_\As\Ms_\Bs)_{AB}$ are immediate from the definition of composition, identity and involution (the fullness condition is not relevant). 

\medskip 

We recall that, with the same argument presented in~\cite[remark 2.14]{BCL08}, on the non-full imprimitivity bimodule ${}_\As\Ms_\Bs$ there is a well-defined norm $\|x\|^2_\Ms:=\|x^*\circ x\|_\Bs=\|x\circ x^*\|_\As$, for all $x\in\Ms$ and hence the completeness of $\Ms$ in each one of the inner products are equivalent. The positivity conditions $x^*\circ x\in\Bs_+$ and $x\circ x^*\in\As_+$, for all $x\in\Ms$, are already part of the definition of C*-algebra valued inner product of our non-full imprimitivity Hilbert C*-bimodule. The C*-property is a consequence of the above defined norm on $\Ms$. 
\end{proof}

\begin{theorem}
Every non-full imprimitivity Hilbert C*-module ${}_\As\Ms_\Bs$ over commutative unital C*-algebras is canonically isomorphic to the family of continuous sections vanishing at infinity of a complex line bundle over the locally compact Hausdorff graph of a local homeomorphism between open sets of the Gel'fand spectra of the two C*-algebras.
\end{theorem}
\begin{proof}
For the non-full Hilbert C*-bimodule ${}_\As\Ms_\Bs$, we construct as in lemma~\ref{lem: linking} the linking non-full commutative C*-category $\Cs({}_\As\Ms_\Bs)$. 

\medskip 

We now apply to $\Cs({}_\As\Ms_\Bs)$ our spectral theorem~\ref{prop: spectral} for commutative non-full C*-categories: 
in this way we obtain a spectral spaceoid $\Sigma(\Cs({}_\As\Ms_\Bs))$ that is a Fell line bundle $\Es^{\Cs(\Ms)}$ over a topological groupoid $\Xs^{\Cs(\Ms)}$, whose $AB$-Hom-set $\Es^{\Cs(\Ms)}_{AB}$ is a complex line-bundle over a unique span of homeomorphism with open subsets of the compact Hausdorff Gel'fand spectra of the C*-algebras $\As$ and $\Bs$: 
\begin{equation*}
\Sp(\As)\supset{\Xs^{\Cs(\Ms)}_{A}}\xleftarrow{t^{\Xs^{\Cs(\Ms)}}}\Xs^{\Cs(\Ms)}_{AB}\xrightarrow{s^{\Xs^{\Cs(\Ms)}}}\Xs^{\Cs(\Ms)}_{B}\subset\Sp(\Bs).
\end{equation*}
From our previous theorem~\ref{th: Gg}, we have that the restriction of the Gel'fand transform $\Gg^{\Cs(\Ms)}$ of the non-full commutative C*-category $\Cs(\Ms)$ to the $AB$-Hom-set provides the required canonical isomorphism of non-full Hilbert C*-bimodules between ${}_\As\Ms_\Bs$ and the non-full imprimitivity Hilbert C*-bimodule $\Gamma(\Es^{\Cs(\Ms)}_{AB})$ of sections of the complex line bundle $\Es^{\Cs(\Ms)}_{AB}\xrightarrow{\pi^{\Cs(\Ms)}}\Xs^{\Cs(\Ms)}_{AB}$ vanishing at infinity. 
\end{proof} 
Although we are not going to pursue this topic here, we notice that one can directly utilize the technical  details in the proof of proposition~\ref{prop: X-topology} to explicitly construct the unique homeomorphism between open subsets of the Gel'fand spectra of the two unital C*-algebras of the non-full imprimitivity bimodule ${}_{\As}\Ms_\Bs$. 

\section{Outlook} \label{sec: 10}

Some possible further generalizations of the present work might be considered such as the: 
\begin{itemize}
\item
elimination of the ``object-bijectivity'' condition for the morphisms in the category $\Af$ of small commutative C*-categories and respectively the bijective nature of the map $f_\Rs$ in the definition of morphism between the base spaces of (non-full) topological spaceoids;

\item
usage of ``non-unital'' commutative C*-categories (see~\cite{Mi02}), and respectively non-unital Fell bundles in the definition of total space of the spaceoids, a situation that corresponds to the usual extension of Gel'fand-Na\u\i mark duality to the case of non-unital commutative C*-algebras (and hence locally compact Hausdorff spaces as spectra); 
\item 
study of possible alternative descriptions of the non-full spectral spaceoids defined here, as already undertaken in the full case in~\cite{BCL13}, or even more intriguingly, introducing suitable universal ``Alexandroff zerifications'' of our non-full spaceoids, providing a unified description of spectral theory of non-unital and unital C*-categories. 
\end{itemize}

It is expected that general results obtainable using non-commutative spectral spaceoids for non-commutative C*-algebras, see~\cite{BCP19} and forthcoming work, will subsume the present material on spectral theory of C*-categories (since to every C*-category one can associate its non-commutative enveloping C*-algebra and simply study its non-commutative spectral spaceoid). 

\medskip 

Of separate independent interest is the usage of the techniques and insight acquired here for the development of a more general spectral theory for non-full Hilbert C*-bimodules along the lines presented in~\cite{BCL08}.  

\medskip 

We did not even enter here into the extremely intriguing study of ``vertical-categorification'' of the notion of spectral spaceoid. This will have to wait for a suitable definition of the notion of commutative strict/weak higher C*-category (see~\cite{BCLS20} for a possible view on strict $n$-C*-categories). 


\bigskip 

\emph{Notes and Acknowledgments:}  
P.B.~thanks Starbucks Coffee (Langsuan, Jasmine City, \hbox{T-square}, Gaysorn Plaza, USBII Tower) 
where he spent most of the time dedicated to this research project; he also thanks Fiorentino Conte of ``The Melting Clock'' for the great hospitality during many crucial on-line meetings.

\end{document}